\documentclass[11pt]{amsart}

\usepackage[english]{babel}     
\usepackage[utf8]{inputenc}
\usepackage{amsmath}
\usepackage{amsfonts}
\usepackage{amssymb}
\usepackage{mathrsfs}
\usepackage{stmaryrd}
\usepackage{hyperref}
\usepackage{xcolor}
\usepackage{comment}
\usepackage{enumerate}
\usepackage{tikz-cd}
\usepackage{pgfplots}
\usepackage{doi}

\usepackage{todonotes}
\usepackage{comment}
\specialcomment{auxproof}
{\mbox{}\newline\textbf{BEGIN: AUX-PROOF}\dotfill\newline}
{\mbox{}\newline\textbf{END: AUX-PROOF}\dotfill\newline}
\excludecomment{auxproof}

\newtheorem{theorem}{Theorem}[section]
\newtheorem{theoremx}{Theorem}

\newtheorem{lemma}[theorem]{Lemma}
\newtheorem{proposition}[theorem]{Proposition}
\newtheorem{corollary}[theorem]{Corollary}

\theoremstyle{definition}
\newtheorem{definition}[theorem]{Definition}

\newtheorem{remark}[theorem]{Remark}

\theoremstyle{remark}
\numberwithin{equation}{section}

\let\epsilon\varepsilon

\newcommand{\R}			{{\operatorname{\mathbb{R}}}}

\DeclareMathOperator{\supp}{supp}

\DeclareMathOperator{\sign}{sign}

\newcommand{\lambdaw}{\widetilde{\lambda}}

\newcommand{\argu}{\textrm{arg} }

\newcommand{\Lloc}{L_{loc}^1}
\newcommand{\triplevert}{|\!|\!|}

\title[Schur multipliers of second order divided difference functions]{On the best constants of Schur multipliers of second order divided difference functions}

\date{\noindent \today.  MSC2010 keywords: 47B10, 47L20, 47H60.  MC is  supported by the NWO Vidi grant VI.Vidi.192.018 `Non-commutative harmonic analysis and rigidity of operator algebras'.}

\author[Martijn Caspers and Jesse Reimann]{Martijn Caspers and Jesse Reimann}

\address{TU Delft, EWI/DIAM,
	P.O.Box 5031,
	2600 GA Delft,
	The Netherlands}

\email{M.P.T.Caspers@tudelft.nl}
\email{J.Reimann@tudelft.nl}

\begin{document}

\maketitle

\begin{abstract}
We give a new proof of the boundedness of bilinear Schur multipliers of second order divided difference functions, as obtained  earlier by Potapov, Skripka and Sukochev in their proof of Koplienko's conjecture on the existence of higher order spectral shift functions. Our proof is based on recent methods involving bilinear transference and the H\"ormander-Mikhlin-Schur multiplier theorem. Our approach provides a significant  sharpening of the known asymptotic bounds of bilinear Schur multipliers of second order divided difference functions. Furthermore, we   give a new lower bound of these bilinear Schur multipliers, giving again a fundamental improvement on the best known bounds obtained by Coine, Le Merdy, Potapov, Sukochev and Tomskova.

More precisely, we prove that for $f \in C^2(\mathbb{R})$ and  $1 < p, p_1, p_2 < \infty$ with $\frac{1}{p} = \frac{1}{p_1} + \frac{1}{p_2}$ we have
 \[
\Vert  M_{f^{[2]}}: S_{p_1} \times S_{p_2} \rightarrow S_p \Vert \lesssim  \Vert f'' \Vert_\infty D(p, p_1, p_2),
 \]
 where the constant $D(p, p_1, p_2)$ is specified in Theorem~\ref{Thm=TheoremA} and $D(p, 2p, 2p) \approx p^4 p^\ast$ with $p^\ast$ the H\"older conjugate of $p$.
 We further show that for $f(\lambda) = \lambda \vert \lambda \vert$, $\lambda \in \mathbb{R}$,   for every $1 < p < \infty$ we have
 \[
p^2 p^\ast \lesssim \Vert  M_{f^{[2]}}: S_{2p} \times S_{2p} \rightarrow S_p \Vert.
 \]
 Here $f^{[2]}$ is the second order divided difference function of $f$ with  $M_{f^{[2]}}$ the associated Schur multiplier. In particular it follows that our estimate $D(p, 2p, 2p)$ is optimal for $p \searrow 1$.
\end{abstract}

\section{Introduction}

In~\cite{PSS-Inventiones}, Potapov, Skripka and Sukochev resolved a fundamental open conjecture by Koplienko~\cite{kopl_trace}. This conjecture asserts the existence of  so-called \emph{spectral shift functions} $\eta_{n,H,V}$, for which the expression \begin{equation}\label{eqn: spectral_shift}
    {\rm Tr} \left(f(H+V)-\sum_{k=0}^{n-1}\frac{1}{k!}\frac{d^k}{dt^k}f(H+tV)\Bigr|_{t=0}\right)=\int_{\mathbb{R}}f^{(n)}(t)\eta_{n,H,V}(t)dt
\end{equation}
is well-defined for the trace $ {\rm Tr} $ on bounded operators on a Hilbert space,  $H$ a self-adjoint operator, and  $V$ in the Schatten class $S_n$. The existence of the spectral shift function goes back to the fundamental work of Krein~\cite{Krein1,Krein2} and Lifschitz~\cite{Lifschitz}, and has ample applications in perturbation theory, mathematical physics, and noncommutative geometry. See~\cite{Gesztesy} for an overview.

The key result in~\cite{PSS-Inventiones} is~\cite[Remark 5.4]{PSS-Inventiones}, which is a direct consequence of the more general result proved in~\cite[Theorem 5.3]{PSS-Inventiones}. It asserts that multiple operator integrals of higher order divided difference functions are bounded maps on Schatten classes. The precise statement of~\cite[Remark 5.4]{PSS-Inventiones} in the second order case, up to the boundedness constant, is recorded below as Theorem~\ref{Thm=TheoremA_intro}. 

In the linear case, i.e.\ order one, the search for optimal proofs and constants for operator integrals of divided difference functions has attracted great attention and a considerable number of the most important problems have been solved.
The existence of first order spectral shift functions was first resolved in~\cite{PoSu11}, and soon after the proofs were optimised  to yield sharp estimates for double operator integrals of divided difference functions. In particular, the best constants were found in~\cite{CMPS}, and weak-$L^1$  and $\mathrm{BMO}$ end-point estimates have been obtained in~\cite{CPSZ} and~\cite{CJSZ} respectively. Furthermore, in the range $0 < p < 1$, the boundedness of double operator integrals of divided difference functions has fully been clarified recently by McDonald and Sukochev~\cite{McDonaldSukochev}. For $p=1$, the best known result goes back to Peller~\cite{Peller}. 
Finally,  a rather general H\"ormander-Mikhlin-Schur multiplier theorem was established in the groundbreaking work~\cite{ParcetAnnals}, yielding the main results of~\cite{PoSu11} and~\cite{CMPS} as a special case. 

When we consider the higher order problem of finding good bounds on multilinear operator integrals of divided difference functions as in~\cite{PSS-Inventiones},  nothing is known about optimal bounds or end-point estimates except for the case of the generalised absolute value map~\cite{CSZ-Israel}. 
Since the key results from~\cite{PSS-Inventiones}  were proven, which is over a decade ago,  significant advances have been made in the theory of Schur multipliers. 
This motivates our re-examination of this result, as  we investigate here whether recent proof methods offer new insights.   Let us first state our main result and then comment on the proof methods. 

\vspace{0.3cm}

\noindent {\bf Upper bounds.} We first define the second order divided difference functions. Let  $f \in C^2(\mathbb{R})$, then the first order divided difference function is defined by the difference quotient
\[
f^{[1]}(\lambda,\mu):=\frac{f(\lambda)-f(\mu)}{\lambda-\mu}
\]
for $\lambda\ne\mu$, and by setting $f^{[1]}(\lambda,\lambda):= f'(\lambda)$. The second order divided difference function is then defined by
\[
f^{[2]}(\lambda_0,\lambda_1,\lambda_2):=\frac{f^{[1]}(\lambda_{i-1},\lambda_{i})-f^{[1]}(\lambda_{i},\lambda_{i+1})}{\lambda_{i-1}-\lambda_{i+1}}
\]
where $i$ is chosen such that $\lambda_{i-1} \not = \lambda_{i+1}$, with $\lambda_3$ interpreted as $\lambda_0$, and otherwise we set $f^{[2]}(\lambda,\lambda,\lambda):= f''(\lambda)$. The function~$f^{[2]}$ is well-defined and invariant under permutation of the variables.  
Our main result is now stated as follows.   The definition of a Schur multiplier will be recalled in Section~\ref{Sect=Preliminaries}.  Throughout the paper we use the notation $p^\ast = \frac{p}{p-1}$ for the H\"older conjugate of $1 < p < \infty$.

\begin{theoremx}\label{Thm=TheoremA_intro}
For every $f \in C^2(\mathbb{R})$ and for every $1<p, p_1, p_2 < \infty$ with $\frac{1}{p} = \frac{1}{p_1} + \frac{1}{p_2}$ we have
\[
\Vert M_{f^{[2]}}: S_{p_1} \times S_{p_2} \rightarrow S_{p} \Vert \lesssim D(p, p_1, p_2) \Vert f'' \Vert_\infty,
\]
where
\[
\begin{split}
D(p, p_1, p_2)  =  & C(p,p_1, p_2) (\beta_{p_1}  + \beta_{p_2}  ) +  \beta_{p_1} \beta_{p_2}(\beta_p  + \beta_{p_1}  + \beta_{p_2} ), \\ 
 C(p,p_1, p_2) = & \beta_{p}\beta_{p_1}\beta_{p_2}   
    +\min(\beta_{p_1}^2\beta_{p},\beta_{p}^2\beta_{p_1})
 + \min(\beta_{p_2}^2\beta_{p},\beta_{p}^2\beta_{p_2}) +\min(\beta_{p_2}^2\beta_{p_1},\beta_{p_1}^2\beta_{p_2}),
\end{split}
\]
and $\beta_q = q q^\ast$.
\end{theoremx}

In particular, if we set $p_1=p_2=2p$ we get the following asymptotic behaviour for the constant  $D(p, 2p, 2p)$. For $p\to\infty$,  $D(p, 2p, 2p)$ is of order at most $O(p^4)$,  and of order~$O(p^\ast)$  when $p \searrow 1$. To see the latter limit, just note that $(2p)^\ast \nearrow 2$ and in particular does not blow up. This improves on the constant obtained by the proof method in~\cite{PSS-Inventiones} by eight orders, see Remark~\ref{Remark=OldOrder}. Note that in Theorem~\ref{Thm=TheoremB_intro} below, we justify that our constants must be quite close to the optimal ones. 

\vspace{0.3cm}

\noindent {\bf Proof methods.}  We now describe the novel parts of 
 our proof. Essentially, there are four aspects: avoidance of triangular truncations, bilinear transference, the use of the H\"ormander-Mikhlin-Schur multiplier theorem~\cite{ParcetAnnals}, and finally, in combining the estimates  we use a range of bilinear multipliers that map to $S_1$. 
 
 To start with, our proof relies on the following decomposition of the divided difference function into two-variable terms and three-variable Toeplitz form terms. 

\begin{equation}\label{Eqn=IntroDec} 
f^{[2]}(\lambda_0, \lambda_1, \lambda_2) = \underbrace{   \frac{\lambda_0 - \lambda_1}{ \lambda_0 - \lambda_2} }_{\substack{ \textrm{Three-variable} \\ \textrm{Toeplitz term} }}
\underbrace{
f^{[2]}(\lambda_0, \lambda_1,\lambda_1)
}_{ \textrm{Two-variable term } } 
+
\underbrace{ 
\frac{\lambda_1 - \lambda_2}{ \lambda_0 - \lambda_2} 
}_{\substack{ \textrm{Three-variable} \\ \textrm{Toeplitz term} }}
\underbrace{  f^{[2]}(\lambda_1, \lambda_{1}, \lambda_{2} )}_{\textrm{ Two-variable term  }  }.
\end{equation}

\noindent  This yields a decomposition of the corresponding Schur multiplier into linear Schur multipliers and bilinear Toeplitz form Schur multipliers, which we can treat separately. Crucially, we refine the decomposition~\eqref{Eqn=IntroDec} such that we can avoid the use of triangular truncations. This alone improves the upper bound on the norm of the Schur multiplier by three orders in $p$ compared to~\cite{PSS-Inventiones}. 

  The boundedness of a linear Toeplitz form Schur multiplier is implied by the boundedness of an associated Fourier multiplier through the transference method~\cite{BoFe84,NeRi11,CaSa15}. This transference method was recently extended to multilinear Toeplitz form Schur multipliers~\cite{CKV,CJKM}. We apply this to reduce our proof of the boundedness of the bilinear Toeplitz form Schur multipliers to the boundedness of the associated bilinear Fourier multiplier. 

For this, we use that
it is possible to show that this Fourier multiplier is a so-called Calder\'on-Zygmund operator. Such operators are known to be well-behaved under extension to UMD spaces in the linear case, such as for example the Schatten classes $S_p$, $p\in(1,\infty)$, see e.g.\ the monograph~\cite{aibs_vol_1}. 
This result was recently extended to multilinear Calder\'on-Zygmund operators~\cite{DiPlinioMathAnn}. Unfortunately, the proofs in~\cite{DiPlinioMathAnn} do not keep track of the constants, though following the proof gives an explicit constant. We have therefore carefully outlined the proof of~\cite{DiPlinioMathAnn} in the appendix of our paper, as the $p$-dependence of the bound when considering Schatten classes concerns our main result. A very important observation is that we are dealing in this paper with  Calder\'on-Zygmund operators that are Fourier multipliers, hence the paraproducts appearing in the multilinear dyadic representation theorem used in~\cite{DiPlinioMathAnn}  vanish. This also yields an improvement on the bounds of our  Calder\'on-Zygmund operators. 

 For non-Toeplitz form Schur multipliers, the transference method is generally difficult to apply, if at all possible. However, a recent result on the boundedness of linear Schur multipliers, including those of non-Toeplitz form, gives a rather simple sufficient condition for their boundedness. 
In~\cite{ParcetAnnals}, it was shown that a H\"ormander-Mihlin type condition implies boundedness of the Schur multipliers $M_m$, even if the symbol~$m$ is not of Toeplitz form. In fact, a slightly weaker condition is sufficient, as mixed derivatives need not be considered. It turns out that these H\"ormander-Mihlin type conditions can be used to effectively estimate the linear (non-Toeplitz) terms occuring in~\eqref{Eqn=IntroDec}. 
 
Finally we need to combine the estimates we get for the three-variable Toeplitz terms with the ones for the two variable terms. Each of these terms yield a constant of order $O(p^\ast)$ for $p \searrow 1$ and so a naive combination of the estimates would yield order $O( (p^\ast)^2)$.  Interestingly, we have found a way to combine  the two estimates so that for the asymptotics for $p \searrow 1$ only one of the terms is relevant, and we are able to control the norm of our Schur multiplier with order $O(p^\ast)$ again. For this we prove that certain bilinear multipliers that appear in our decomposition actually map boundedly to $S_1$.


\vspace{0.3cm}

\noindent {\bf Lower bounds.}   In the final part of this paper we establish a lower bound for the bilinear Schur multiplier appearing in Theorem~\ref{Thm=TheoremA_intro}. An alternative form of this problem was already considered in~\cite{CLPST}, where it was shown that there exists a function $f \in C^2(\mathbb{R})$ for which $M_{f^{[2]}}$ does not map $S_2 \times S_2$ to $S_1$ boundedly. Outside of $[-1, 1]$, this function is given by $f(s) = s \vert s\vert$, and it is $C^2$ inside  $[-1, 1]$. Such functions are generalised versions of the absolute value map and have played an important role in perturbation theory ever since the results of Kato~\cite{Kato} and Davies~\cite{Davies} on Lipschitz properties of the absolute value map. A weak type estimate for generalised absolute value maps was obtained in~\cite{CSZ-Israel}.  

We use the generalised absolute value function to provide lower bounds of bilinear Schur multipliers in the following way. Note that since this function is not $C^2$, we make sense of the second order derivative as a weak derivative.  
\begin{theoremx}\label{Thm=TheoremB_intro}
    Let  $f(s) = s \vert s\vert, s \in \mathbb{R}$. For every $1 < p < \infty$, we have  
\begin{equation}\label{Eqn=IntroMf}
\Vert M_{f^{[2]}}: S_{2p} \times S_{2p} \rightarrow S_{p}  \Vert \gtrsim  p^2 p^\ast.
\end{equation}
\end{theoremx}

  Our proof method is as follows. For Schur multipliers whose   symbol is continuous on an open subset $\Omega \subseteq \mathbb{R}^2$, restricting the symbol to any discrete subset $X \times Y \subseteq \Omega$ yields a new Schur multiplier whose norm is not larger than the norm of the original Schur multiplier,   see \cite[Theorem 1.19]{LafforgueDelaSalle} for this restriction theorem. Further,  Davies (see \cite[Lemma 10]{Davies}) showed that one can approximate the triangular truncation map by restrictions of the divided difference function of the usual absolute value map to discrete sets.  Here we show that also for the second order divided difference function of the generalised absolute value map, we can find restrictions to discrete sets that approximate the triangular truncation map.   
 In turn, sharp lower bounds for the triangular truncation map are known and go  back to Krein's  analysis of singular values of the Volterra operator \cite{GohbergKrein}. By combining these ideas, we are able to find good lower bounds for our bilinear Schur multipliers of second order divided difference functions. Remarkably, we obtain a square power $p^2$ for the asymptotics $p \rightarrow \infty$ and a linear term $p^\ast$ for $p \searrow 1$.

Theorem~\ref{Thm=TheoremB_intro} closely relates to the main result of~\cite{CLPST}; in fact  it implies a mild variation of the main theorem of~\cite{CLPST}.  
In Remark~\ref{Rmk=Comparison} we conceptually compare our proof to~\cite{CLPST} and argue that it gives a fundamentally better lower bound than what the method from~\cite{CLPST} would give.  

Note in particular that for $p \searrow 1$, Theorems~\ref{Thm=TheoremA_intro} and~\ref{Thm=TheoremB_intro} yield that the asymptotics of the norm of~\eqref{Eqn=IntroMf} for general $f$ are precisely of order $O(p^\ast)$.  The asymptotics for $p \rightarrow \infty$ are narrowed down to an order between $O(p^2)$ and $O(p^4)$, and both the lower and upper bounds we find here are fundamentally better than what was previously known. 

\vspace{0.3cm}

\noindent {\bf Structure of the paper.} Section~\ref{Sect=Preliminaries} contains preliminaries on Schur multipliers and Calder\'on-Zygmund operators. In Section~\ref{sect: decomposition fn}, we present a decomposition of the Schur multiplier of second order divided difference functions into linear terms and bilinear Toeplitz form terms. Their boundedness is shown in Section~\ref{sect: linear_bddness_hms} (linear terms) and  Sections~\ref{Sect=S1Estimate} and ~\ref{Sect=Bilinear} (bilinear terms). In Section~\ref{sect: proof_thrm_A}, we prove Theorem~\ref{Thm=TheoremA_intro}, as well as an additional extrapolation result. Theorem~\ref{Thm=TheoremB_intro} is proven in Section~\ref{Sect=LowerBound}. In Appendix~\ref{Sect=AppendixConstants} we have incorporated all arguments that are needed to obtain the explicit constants of Theorem~\ref{thrm: diplinio_statement_bilinear}; this essentially requires a careful analysis of the proofs in~\cite{DiPlinioMathAnn} and references given there. We decided to give full details here as this contributes directly to our main result. 

\vspace{0.3cm}
 
\noindent {\bf Acknowledgement.} The authors wish to thank the anonymous referee for a detailed report that led to several improvements of our manuscript. 

\section{Preliminaries}\label{Sect=Preliminaries}
We recall the following preliminaries, for which we refer to~\cite{SkripkaTomskova} for multilinear operator integrals, to~\cite{GrafakosOldNew} for harmonic analysis, and to~\cite{GrafakosTorres} for (scalar-valued) multilinear Calder\'on-Zygmund theory. 

\subsection{General notation}
We let the natural numbers $\mathbb{N}$ be all integer numbers greater than or equal to $1$. 
We shall write $A \lesssim B$ for saying that expression $A$ is always smaller than~$B$ up to an absolute constant, and $A\approx B$ for $A\lesssim B \lesssim A$. We write $f=O(g)$ if we have $|f(\lambda)|\lesssim g(\lambda)$ as~$\lambda$ approaches some specified limit (usually $\lambda \to\infty$).  For $f \in C^n(\mathbb{R})$ we let $f^{(n)}$ denote the $n$-th order derivative. We call a function \emph{smooth} if it is a $C^{\infty}$-function on its domain.  For a continuous function $f:D\to\mathbb{C}$, $D\subseteq\mathbb{R}^d$, we define its \emph{support} to be the closure of the set $\{x\in D\mid f(x)\ne 0\}$. We let $C_c(\mathbb{R})$ denote the continuous functions on $\mathbb{R}$ with compact support.  The Fourier transform of a Schwartz function $m$ is defined as

\begin{equation}\label{Eqn=FourierTransform}
(\mathcal{F}m)(x) =    (2\pi)^{-\frac{d}{2}} \int_{\mathbb{R}^d} m(\xi)e^{-i \xi \cdot x} d\xi.
\end{equation}
   We extend $\mathcal{F}$ in the usual way to the space of tempered distributions.
   
   For $p \in (1, \infty)$ we set~$p^\ast = p/(p-1)$, which is the H\"older conjugate of $p$.
The set $\Delta \subseteq \mathbb{R}^{3d}$ is the set of diagonal elements $(\lambda, \lambda, \lambda)$, $\lambda \in \mathbb{R}^d$; we shall often require this only for $d=1$. We use  notations like $\{ \lambda = \mu \}$ to denote the set  $\{ (\lambda, \mu) \in \mathbb{R}^2 \mid \lambda = \mu\}$.
 The Euclidean norm of a vector~$\xi \in \mathbb{R}^d$ is denoted by $\vert \xi \vert = (\sum_{i=}^d \vert \xi_i \vert^2)^{\frac{1}{2}}$.

We call a function $\varphi: \mathbb{R}^d \setminus \{0 \} \rightarrow \mathbb{C}$ \emph{homogeneous} if it is homogeneous of order $0$, i.e.\ if  
for every $r > 0$, $\xi \in \mathbb{R}^d \setminus \{0 \}$ we have $\varphi(r \xi) = \varphi(\xi)$. Moreover, $\varphi$ is called \emph{even} if $\varphi(-\xi) = \varphi(\xi)$ and \emph{odd} if $\varphi(-\xi) = -\varphi(\xi)$. We may define a function $\mathbb{R}^d \setminus \Delta \rightarrow \mathbb{C}$ to be homogeneous, even, and odd with precisely the same definitions.

 \subsection{Function spaces} We let $C_b(\mathbb{R})$ denote the complex valued continuous bounded functions on $\mathbb{R}$. Furthermore, we let $\Lloc(\mathbb{R}^d)$  denote the  locally integrable functions $\mathbb{R}^d \rightarrow \mathbb{C}$. 
 The Banach space of $p$-integrable functions $\mathbb{R}^d \rightarrow \mathbb{C}$ with norm ${\Vert f \Vert_p = (\int_{\mathbb{R}^d} \vert f(x)\vert^p dx)^{\frac{1}{p}}}$ is denoted by $L^p(\mathbb{R}^d)$.

\subsection{Schatten classes} For $p \in [1, \infty)$, $S_p(\mathbb{R}^d)$ denotes the Schatten $p$-class of $B(L^2(\mathbb{R}^d))$, consisting of all compact operators
$x \in B(L^2(\mathbb{R}^d))$ for which $\Vert x \Vert_p = {\rm Tr}(\vert x \vert^p)^{1/p}$ is finite. Furthermore, $S_\infty(\mathbb{R}^d)$ denotes the compact operators in   $B(L^2(\mathbb{R}^d))$. For $p = 2$ we may identify~$S_2(\mathbb{R}^d)$ linearly with $L^2(\mathbb{R}^d \times  \mathbb{R}^d)$. This way, a kernel $A \in L^2(\mathbb{R}^d \times \mathbb{R}^d)$ corresponds to the operator~$(A \xi)(t) = \int_{ \mathbb{R}^d } A(t,s) \xi(s) ds$ in $S_2(\mathbb{R}^d)$. We shall mostly be concerned with $d=1$ and write~$S_p = S_p(\mathbb{R})$.  Note that for $1<p<\infty$, the dual space of $S_p$ is $S_{p^*}$, where $p^*$ is the H\"older conjugate of $p$.

\subsection{Schur multipliers}\label{Sect=SchurPrelim}
For $m \in L^\infty(\mathbb{R}^{2d})$, the multiplication map $M_m: A \mapsto m A$ acts boundedly on  $L^2(\mathbb{R}^d \times \mathbb{R}^d)$ and hence on $S_2(\mathbb{R}^d)$. Now let us consider $d=1$ and introduce multilinear Schur multipliers as follows. Let $m \in L^\infty(\mathbb{R}^{n+1})$. Then by~\cite[Proposition 5]{CLS-AIF} there exists a unique bounded linear map
\[
M_m: S_2 \times \ldots \times S_2 \rightarrow S_2: (A_1, \ldots, A_n) \mapsto M_m(A_1, \ldots, A_n),
\]
where the kernel of $M_m(A_1, \ldots, A_n)$ is given by
\[
M_m(A_1, \ldots, A_n)(s_0,s_n) =    \int_{\mathbb{R}^{n-1}}   m(s_0, \ldots, s_n) A_1(s_0, s_1) \ldots A_n(s_{n-1}, s_n) ds_1 \ldots ds_{n-1}, \:\: s_0, s_n \in \mathbb{R}.
\]
Moreover, this map is bounded by $\Vert m \Vert_\infty$; this follows from the Cauchy-Schwartz inequality as observed in~\cite[Proposition 5]{CLS-AIF}.
{We recall that in the linear case the following elementary (in)equalities for $1 < p < \infty$ hold
\begin{equation}\label{Eqn=SchurElement}
\begin{split}
\Vert M_{m}: S_{p} \rightarrow   S_{p}  \Vert =& 
\Vert M_{m}: S_{p^\ast} \rightarrow   S_{p^\ast}  \Vert, \\
  \Vert M_{m}: S_{2} \rightarrow   S_{2} \Vert  \leq &
\Vert M_{m}: S_{p}  \rightarrow   S_{p}  \Vert, \\ 
 \Vert M_{m}: S_{2}    \rightarrow   S_{2} \Vert = & \Vert m \Vert_{L^\infty(\mathbb{R}^2)}.
\end{split}
\end{equation}
where the first equality follows from duality, the second from complex interpolation between $p$ and $p^\ast$, and the last from the fact that we identified $S_2$ with $L^2(\mathbb{R} \times \mathbb{R})$ on which $m$ acts as a multiplication operator. }
 
We may similarly define Schur multipliers on discrete sets as follows.  Let $X$ be any set and let~$\ell^2(X)$ be the Hilbert space of square summable functions on $X$. For $p \in [1, \infty)$, let $S_p(\ell^2(X))$ be the Schatten $S_p$-space of $B(\ell^2(X))$ consisting of all  operators $x \in B(\ell^2(X))$ for which the norm {$\Vert x \Vert_p := {\rm Tr}(\vert x \vert^p)^{1/p}$} is finite.  
For $x \in X$, let $p_x$ be the orthogonal projection onto the span of $\delta_x \in \ell^2(X)$,   where $\delta_x(x)=1$ and $\delta_x(y)=0$ for $y\neq x$.  Let $m \in \ell^\infty(X \times X \times X)$ and consider the 
bilinear Schur multiplier 
   \begin{equation}\label{Eqn=DiscMult}
\begin{split}
M_{m}: S_{2}(\ell^2(X)) \times S_{2}(\ell^2(X) ) \rightarrow & S_{2}(\ell^2(X)) \\
(x, y) \mapsto & \sum_{\lambda_1, \lambda_2, \lambda_3 \in X}  m(\lambda_1, \lambda_2, \lambda_3) p_{\lambda_1} x p_{\lambda_2} y p_{\lambda_3}.
\end{split}
\end{equation} 
As before, this map is bounded by   $\Vert m \Vert_\infty$, see \cite[Proposition 5]{CLS-AIF}.  

For  sets $F, G$ consider the disjoint union $X = F \cup G$ and let $p_F$ and $p_G$ be the orthogonal projections of  $\ell^2(X)$ onto $\ell^2(F)$ and $\ell^2(G)$ respectively. Define $S_p(\ell^2(F), \ell^2(G)) = p_F S_p(\ell^2(X), \ell^2(X)) p_G$. Then by \eqref{Eqn=DiscMult} we see that $M_m$ maps $S_2(\ell^2(F), \ell^2(G)) \times S_2(\ell^2(G), \ell^2(F))$ to $S_2(\ell^2(F), \ell^2(F))$. 

In either the continuous or discrete case, let $1 \leq p, p_1, \ldots, p_n   < \infty$ with $p^{-1} = \sum_{i=1}^n p_i^{-1}$. We may consider the restriction of $M_m$ where its $i$-th inputs are restricted to the space $S_2 \cap S_{p_i}$. If this restriction takes values in $S_p$ and has a bounded multilinear extension to $S_{p_1} \times \ldots \times S_{p_n}$,   then this extension, also denoted by $M_m$, is called an $(p_1, \ldots, p_n)$-Schur multiplier. In the discrete case, our terminology is the same but with  $S_r$ replaced by $S_r(\ell^2(X))$.

\subsection{Divided difference functions}\label{subsec: prelim_divdiff}

\begin{definition}[Divided difference functions]\label{def: divdiff}
    Let $f\in C^n(\mathbb{R})$, $n\in\mathbb{N}_0$. We define the \textit{n-th order divided difference function} $f^{[n]}$ of $f$ inductively as follows. The first order divided difference function is constructed as  $f^{[0]}(\lambda_0) \;:=\; f(\lambda_0)$. Then we set
     \begin{align}
        f^{[n]}(\lambda_0,...,\lambda_n)&\;:=\; \begin{cases}
            \frac{f^{[n-1]}(\lambda_0,\dotsc,\lambda_{j-1},\lambda_{j+1},\dotsc,\lambda_{n})-f^{[n-1]}(\lambda_0,\dotsc,\lambda_{i-1},\lambda_{i+1},\dotsc,\lambda_n)}{\lambda_i-\lambda_j}, & \text{if }\lambda_i\ne\lambda_j \text{ for some } i\ne j, \\
            \frac{f^{(n)}(\lambda_0)}{n!}, & \lambda_0=\dotsc=\lambda_n,
        \end{cases}\label{eqn: def_divdiff}
    \end{align}
    where $\lambda_0,\dotsc,\lambda_n\in\mathbb{R}$.  
For $\lambda,\mu\in\mathbb{R}$, we set \begin{equation*}
    f^{[n]}(\lambda^{(k)},\mu^{(n+1-k)}):=f^{[n]}(\underbrace{\lambda,\dotsc,\lambda}_\text{$k$ times},\underbrace{\mu,\dotsc,\mu}_{\substack{n+1-k\\ \text{times}}}).
\end{equation*}
\end{definition}
We shall use repeatedly that divided difference functions are invariant under permutation of the variables, which can be checked by induction from its definition (or see~\cite{Devore}). 

\begin{remark}
For $n=2$ and $f(\lambda) = \lambda \vert \lambda \vert$ we define $f^{[2]}$ in the same way as in Definition~\ref{def: divdiff}, except that we set $f^{[2]} (\lambda, \lambda, \lambda) = 0$. Note that this alternative definition is required, as $f$ is not a $C^2$-function.
\end{remark}
\begin{remark}
We have from e.g.~\cite[Lemma 5.1]{PSS-Inventiones} that
\begin{equation}\label{Eqn=DiffDiff}
\Vert f^{[n]} \Vert_\infty \leq \frac{\Vert f^{(n)} \Vert_\infty }{n!}.
\end{equation}
\end{remark}

\subsection{Fourier multipliers and Calder\'on-Zygmund operators}\label{subsec: fm_czo}  In analogy to the linear definition, we define a bilinear \emph{Fourier multiplier} with symbol $m\in L^{\infty}(\mathbb{R}^{d}\times \mathbb{R}^{d})$ as follows. For Schwartz functions $f_1, f_2$ on $\mathbb{R}^d$, we set  
\begin{equation*}
    T_m(f_1,f_2)(x):=\frac{1}{(2\pi)^{d}}\int_{\mathbb{R}^{d}\times \mathbb{R}^d}m(\xi_1,\xi_2)(\mathcal{F}f_1)(\xi_1)(\mathcal{F}f_2)(\xi_2)e^{i(\xi_1+\xi_2)\cdot x}d\xi.
\end{equation*}
Note that as $\mathcal{F}f_1$ and $\mathcal{F}f_2$ are Schwartz, the integral is over an integrable function and hence this formula is well-defined. 

We recall the following from~\cite{DiPlinioMathAnn}, which we need only for $d=1$. 
Let $T$ be an bilinear operator defined by an integral kernel, i.e.\ there exists a function $K:\mathbb{R}^{3d}\setminus\Delta\to \mathbb{C}$ such that for compactly supported bounded measurable functions $f_1, f_2 \in L^\infty_c(\mathbb{R}^d)$, 
\[
\langle T(f_1,f_2),f_{3}\rangle = \int_{\mathbb{R}^{3d}}K(x_{3},x_1,x_{2})\prod_{j=1}^{3}f_j(x_j)dx
\]
whenever $\mathrm{supp}f_i\cap\mathrm{supp}f_j=\emptyset$ for some $i\ne j$. Such an operator $T$ is called a \emph{Calder\'on-Zygmund operator} if there exists some $\alpha\in(0,1]$ and $C_{K}>0$ such that the following conditions hold:
    \begin{itemize}
        \item (Size condition) for all $x=(x_1,x_2,x_3)\in\mathbb{R}^{3d}\setminus\Delta$, $$|K(x)|\le\frac{C_K}{(|x_1-x_2|+|x_1-x_3|)^{2d}},$$ 
        \item (Smoothness condition) for all $j=1,2,3$, $$|K(x)-K(x')|\le\frac{C_K|x_j-x_j'|^{\alpha}}{(|x_1-x_2|+|x_1-x_3|)^{2d+\alpha}}$$ holds whenever $x,x'\in\mathbb{R}^{3d}\setminus\Delta$ such that $x_i=x_i'$ for $i\ne j$ and $$2|x_j-x_j'|\le\max(|x_1-x_2|,|x_1-x_3|),$$
        \item (Boundedness) for some (equivalently, for all) exponents $p_1,p_2\in(1,\infty)$ and $q_{3}\in(1/2,\infty)$ such that $1/p_1+1/p_2=1/q_{3}$, $$\|T(f_1,f_2)\|_{L^{q_{3}}(\mathbb{R}^d)}\lesssim \|f_1\|_{L^{p_1}(\mathbb{R}^d)}\|f_2\|_{L^{p_2}(\mathbb{R}^d)}.$$ 
    \end{itemize}

\section{Decomposing second order divided difference functions}\label{sect: decomposition fn}
The aim of this section is to show that the bilinear Schur multiplier of the second order divided difference function $f^{[2]}$ admits a decomposition as sums of compositions of bilinear Schur multipliers that are independent of~$f$ and of Toeplitz form as well as linear Schur multipliers. Such decompositions appear already in~\cite{PSS-Inventiones}, but we require a different decomposition that allows us to incorporate the application of triangular truncations into the bilinear part.

  Let $\epsilon > 0$ be small and fixed. Define the sets
\[
\begin{split}
A_{1, \epsilon}   =  &     (-2\epsilon, \pi/2+2\epsilon) \cup (\pi-2\epsilon, 3\pi/2+2\epsilon),\\
A_{2, \epsilon}   =  &  (\pi/2+\epsilon, 3 \pi/4+\epsilon) \cup ( 3\pi/2 + \epsilon, 7 \pi/4 + \epsilon),\\
A_{3, \epsilon}     =&   (3\pi/4 - \epsilon,  \pi - \epsilon) \cup (7 \pi/4 - \epsilon, 2 \pi - \epsilon).
\end{split}
\]
For a point $\xi=(\xi_1, \xi_2) \in \mathbb{R}^2 \setminus \{ 0 \}$ and $A \subseteq \mathbb{R}$ we say~$\argu (\xi_1, \xi_2) \in A$ in case there exists $\theta \in A$ such that $\xi = (\cos(\theta), \sin(\theta))$.
We cut $\mathbb{R}^2 \setminus \{ 0 \}$ into the following areas:
\begin{equation} \label{Eqn=Regions}
\begin{split}
\Delta_{1, \epsilon} = & \{ (\xi_1, \xi_2) \in \mathbb{R}^2 \setminus \{ 0 \} \mid    \argu(\xi_1, \xi_2)  \in A_{1, \epsilon} \}, \\
\Delta_{2, \epsilon} = & \{ (\xi_1, \xi_2) \in \mathbb{R}^2 \setminus \{ 0 \} \mid    \argu(\xi_1, \xi_2)  \in A_{2, \epsilon}  \}, \\
\Delta_{3, \epsilon} = & \{ (\xi_1, \xi_2) \in \mathbb{R}^2 \setminus \{ 0 \} \mid    \argu(\xi_1, \xi_2)  \in  A_{3, \epsilon}    \}.
\end{split}
\end{equation}
All these sets are radial in the sense that if $\xi \in \Delta_{j, \epsilon}$ then $r \xi \in  \Delta_{j, \epsilon}$ for any $r >0$. All $\Delta_{j, \epsilon}$ are open and satisfy  $-\Delta_{j, \epsilon} = \Delta_{j, \epsilon}$.  Further, the sets $\Delta_{j, \epsilon}$, $j =1,2,3$ cover  $\mathbb{R}^2 \setminus \{ 0 \}$. See Figure~\ref{fig:deltas} for an illustration.

\begin{figure}[h]
    \centering
    \includegraphics[scale=0.5]{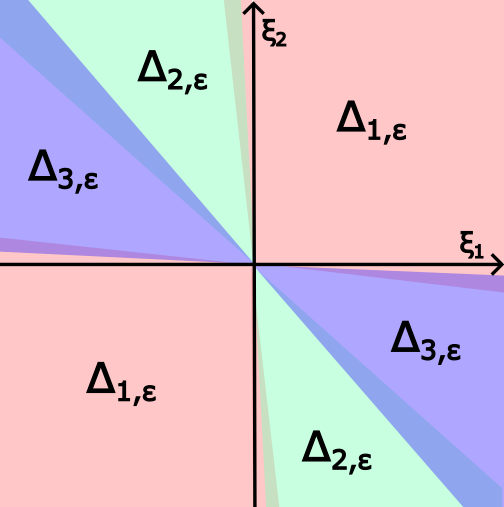}
    \caption{The sets $\Delta_{i,\varepsilon}$ as defined in~\eqref{Eqn=Regions}. Note that the sets are partially overlapping.}
    \label{fig:deltas}
\end{figure}

Let $\mathbb{T} \subseteq \mathbb{R}^2$ be the unit circle. Let  $\theta_1', \theta_2', \theta_3': \mathbb{T}  \rightarrow [0,1]$  be a partition of unity of the sets~$\Delta_{j, \epsilon} \cap \mathbb{T}$, $j = 1,2,3$. We may assume without loss of generality that the support of $\theta_j'$ is contained in $\Delta_{j, \epsilon} \cap \mathbb{T}$.  Furthermore, we may replace $\theta_j'(\xi)$, $\xi \in \mathbb{T}$, by    $\frac{1}{2}(\theta_j'(\xi) + \theta_j'(-\xi))$ and may therefore assume without loss of generality that    $\theta_j'(\xi) = \theta_j'(-\xi)$. Set $\theta_j(\xi) := \theta_j'(\xi \slash \vert \xi \vert)$ for~$\xi \in \mathbb{R}^2 \backslash \{ 0 \}$. Then 
   $\theta_1, \theta_2, \theta_3: \mathbb{R}^2 \setminus \{ 0 \} \rightarrow [0,1]$ are smooth, even, homogeneous functions such that $\theta_1 + \theta_2 + \theta_3 = 1$ on $\mathbb{R}^2 \setminus \{ 0 \}$ and such that the support of $\theta_j$ is contained in $\Delta_{j, \epsilon}$.

 Let
 \begin{equation}\label{eqn=theta_tilde}
 \widetilde{\theta}_j(\lambda_0, \lambda_1, \lambda_2) = \theta_j( \lambda_1 - \lambda_0, \lambda_2 - \lambda_1 ), \quad  (\lambda_1, \lambda_2, \lambda_3) \in \mathbb{R}^3 \setminus \Delta,
 \end{equation}
 where we recall $\Delta = \{ (\lambda, \lambda, \lambda) \mid \lambda \in \mathbb{R} \}$.
We  obtain for $(\lambda_0,\lambda_1, \lambda_2) \in \mathbb{R}^3 \setminus \Delta$ and $f \in C^2(\mathbb{R})$  that
\begin{equation}\label{Eqn=ThetaSum}
f^{[2]}(\lambda_0, \lambda_1, \lambda_2) = \sum_{j=1}^3 \widetilde{\theta}_j(\lambda_0, \lambda_1,   \lambda_2) f^{[2]}(\lambda_0, \lambda_1, \lambda_2).
\end{equation}
We shall now decompose each of these three summands. A general decomposition method can be found in~\cite[Lemma 5.8]{PSS-Inventiones}; however, in the special case of divided difference functions, both the statement and the proof are more straightforward in our version below.

\begin{lemma}\label{Lem=Decomposition}
Let $f \in C^n(\mathbb{R})$, $n \geq 1$ and let $\lambda_0, \ldots, \lambda_n \in \mathbb{R}$ be such that for some $i,j \in \{ 0, \ldots, n \}$ we have $\lambda_i \not = \lambda_j$. Let $\mu \in \mathbb{R}$. Then,
\[
\begin{split}
f^{[n]}(\lambda_0, \ldots, \lambda_n) &=  \frac{\lambda_i - \mu}{ \lambda_i - \lambda_j} f^{[n]}(\lambda_0, \ldots, \lambda_{j-1}, \mu, \lambda_{j+1}, \ldots, \lambda_n   ) 
\\
&\;+ \frac{\mu - \lambda_j}{ \lambda_i - \lambda_j} f^{[n]}(\lambda_0, \ldots, \lambda_{i-1}, \mu, \lambda_{i+1}, \ldots, \lambda_n   ).
\end{split}
\]
\end{lemma}
\begin{proof}
    Since $f^{[n]}$ is invariant under permutation of its variables~\cite{Devore}, we assume without loss of generality that $(i,j)=(0,1)$ to simplify the notation. It follows for $\mu\ne\lambda_i$, $i=0,1$, that \begin{align*}
        f^{[n]}(\lambda_0,\lambda_1,\lambda_2,\dotsc,\lambda_n) &= \frac{1}{\lambda_0-\lambda_1}\left(f^{[n-1]}(\lambda_0,\lambda_2,\lambda_3,\dotsc,\lambda_n)-f^{[n-1]}(\lambda_1,\lambda_2,\dotsc,\lambda_n)\right) \\
        &= \frac{1}{\lambda_0-\lambda_1}\left(f^{[n-1]}(\lambda_0,\lambda_2,\lambda_3,\dotsc,\lambda_n)-f^{[n-1]}(\mu,\lambda_2,\lambda_3,\dotsc,\lambda_n)\right) \\
        &\quad + \frac{1}{\lambda_0-\lambda_1}\left(f^{[n-1]}(\mu,\lambda_2,\lambda_3,\dotsc,\lambda_n)-f^{[n-1]}(\lambda_1,\lambda_2,\dotsc,\lambda_n)\right) \\
        &= \frac{\lambda_0-\mu}{\lambda_0-\lambda_1}f^{[n]}(\lambda_0,\mu,\lambda_2,\lambda_3,\dotsc,\lambda_n) + \frac{\mu-\lambda_1}{\lambda_0-\lambda_1}f^{[n]}(\mu,\lambda_1,\lambda_2,\dotsc,\lambda_n).
    \end{align*}
    Note the same formula holds for $\lambda_0 = \mu$ or $\lambda_1=\mu$ as long as $\lambda_0\ne \lambda_1$.   Indeed, assume without loss of generality $\lambda_0=\mu\ne \lambda_1$, then \begin{equation*}
        \underbrace{\frac{\lambda_0-\mu}{\lambda_0-\lambda_1}}_{=0}f^{[n]}(\lambda_0,\mu,\lambda_2,\lambda_3,\dotsc,\lambda_n) + \underbrace{\frac{\mu-\lambda_1}{\lambda_0-\lambda_1}}_{=1}f^{[n]}(\mu,\lambda_1,\lambda_2,\dotsc,\lambda_n)= f^{[n]}(\lambda_0,\lambda_1,\lambda_2,\dotsc,\lambda_n).
    \end{equation*}
\end{proof}

We define the following functions for $(\lambda_0, \lambda_1, \lambda_2) \in \mathbb{R}^3$. Let
\[
\begin{split}
\psi_1(\lambda_0, \lambda_1, \lambda_2) = & \frac{\lambda_0 - \lambda_1}{\lambda_0 - \lambda_2}, \qquad   \lambda_0 \not = \lambda_2, \\
\psi_2(\lambda_0, \lambda_1, \lambda_2) = \psi_1(\lambda_2, \lambda_0, \lambda_1) = & \frac{\lambda_2 - \lambda_0}{\lambda_2 - \lambda_1}, \qquad  \lambda_2 \not = \lambda_1, \\
\psi_3(\lambda_0, \lambda_1, \lambda_2) = \psi_1(\lambda_1, \lambda_2, \lambda_0) = & \frac{\lambda_1 - \lambda_2}{\lambda_1 - \lambda_0}, \qquad   \lambda_0 \not = \lambda_1.
\end{split}
\]
and
\[
\phi_f(\lambda, \mu) = f^{[2]}(\lambda, \mu, \mu), \quad  \mathring{\phi}_f(\lambda,    \mu) = f^{[2]}(\lambda, \lambda, \mu), \quad \lambda, \mu \in \mathbb{R}.
\]
As divided difference functions are permutation invariant, we have $\mathring{\phi}_f(\lambda, \mu) = \phi_f(\mu, \lambda)$.

At this point we note that $\widetilde{\theta}_j \psi_j, j =1,2,3$ extends to a bounded continuous function on $\mathbb{R}^3 \setminus \Delta$.   Indeed, let $\lambda \in \mathbb{R}^3 \setminus \Delta$ be in the support of $\widetilde{\theta}_j$. Note that the support is by definition a closed set contained in $\Delta_{j,\varepsilon}$, and that $\Delta_{j,\varepsilon}$ does not intersect the rays   $\lambda_0  = \lambda_2$ (for $j=1$), $\lambda_2  = \lambda_1$ (for~$j=2$), or~$\lambda_0  = \lambda_1$ (for $j=3$), see~\eqref{eqn=theta_tilde}. Hence $\widetilde{\theta}_j \psi_j$ is bounded on the support of $\widetilde{\theta}_j$. We may thus extend  $\widetilde{\theta}_j \psi_j$ by setting it equal to zero outside the support of  $\widetilde{\theta}_j$. This extended function is a smooth even homogeneous function on $\mathbb{R}^3 \setminus \Delta$.

We now apply the decomposition of Lemma~\ref{Lem=Decomposition} in the case   $n=2$. In case  $(\lambda_0, \lambda_1, \lambda_2) \in \Delta_{1, \epsilon}$ we have, as also noted in the previous paragraph, that $\lambda_0 \not = \lambda_2$, and we get
\begin{equation}\label{Eqn=Reduction1}
\begin{split}
 f^{[2]}(\lambda_0, \lambda_1, \lambda_2) = &
  \frac{\lambda_0 - \lambda_1}{\lambda_0 - \lambda_2} f^{[2]}(\lambda_0, \lambda_1, \lambda_1) +
 \frac{\lambda_1 - \lambda_2}{\lambda_0 - \lambda_2} f^{[2]}(\lambda_1, \lambda_1, \lambda_2) \\
=  & \psi_1(\lambda_0, \lambda_1, \lambda_2) \phi_f(\lambda_0, \lambda_1) + (1-\psi_1)(\lambda_0, \lambda_1, \lambda_2) \mathring{\phi}_f(\lambda_1, \lambda_2).
 \end{split}
\end{equation}
Similarly, we may use the permutation invariance of divided difference functions and find for $(\lambda_0, \lambda_1, \lambda_2) \in \Delta_{2, \epsilon}$ that
\begin{equation}\label{Eqn=Reduction2}
\begin{split}
 f^{[2]}(\lambda_0, \lambda_1, \lambda_2) = f^{[2]}(\lambda_2, \lambda_0, \lambda_1) =  &
  \frac{\lambda_2 - \lambda_0}{\lambda_2 - \lambda_1} f^{[2]}(\lambda_2, \lambda_0, \lambda_0) +
 \frac{\lambda_0 - \lambda_1}{\lambda_2 - \lambda_1} f^{[2]}(\lambda_0, \lambda_0, \lambda_1) \\
 = & \psi_2(\lambda_0, \lambda_1, \lambda_2) \mathring{\phi}_f(\lambda_0, \lambda_2) + (1-\psi_2)(\lambda_0, \lambda_1, \lambda_2) \mathring{\phi}_f(\lambda_0, \lambda_1).
\end{split}
\end{equation}
Finally, for  $(\lambda_0, \lambda_1, \lambda_2) \in \Delta_{3, \epsilon}$, we have that
\begin{equation}\label{Eqn=Reduction3}
\begin{split}
 f^{[2]}(\lambda_0, \lambda_1, \lambda_2) =
  f^{[2]}(\lambda_1, \lambda_2, \lambda_0) = &
  \frac{\lambda_1 - \lambda_2}{\lambda_1 - \lambda_0} f^{[2]}(\lambda_1, \lambda_2, \lambda_2) +
 \frac{\lambda_2 - \lambda_0}{\lambda_1 - \lambda_0} f^{[2]}(\lambda_2, \lambda_2, \lambda_0) \\
 = & \psi_3(\lambda_0, \lambda_1, \lambda_2) \phi_f(\lambda_1, \lambda_2) + (1-\psi_3)(\lambda_0, \lambda_1, \lambda_2) \phi_f(\lambda_0, \lambda_2).
 \end{split}
\end{equation}
Combining~\eqref{Eqn=ThetaSum},~\eqref{Eqn=Reduction1},~\eqref{Eqn=Reduction2}, and~\eqref{Eqn=Reduction3} we find that
\begin{equation}\label{Eqn=DecomposeFirst}
\begin{split}
 f^{[2]}(\lambda_0, \lambda_1, \lambda_2) = & \widetilde{\theta}_1(\lambda_0, \lambda_1,   \lambda_2) \left(  \psi_1(\lambda_0, \lambda_1, \lambda_2) \phi_f(\lambda_0, \lambda_1) + (1-\psi_1)(\lambda_0, \lambda_1, \lambda_2) \mathring{\phi}_f(\lambda_1, \lambda_2) \right) \\
 & + \widetilde{\theta}_2(\lambda_0, \lambda_1,   \lambda_2)  \left( \psi_2(\lambda_0, \lambda_1, \lambda_2) \mathring{\phi}_f(\lambda_0, \lambda_2) + (1-\psi_2)(\lambda_0, \lambda_1, \lambda_2) \mathring{\phi}_f(\lambda_0, \lambda_1) \right) \\
 & +  \widetilde{\theta}_3(\lambda_0, \lambda_1,   \lambda_2) \left(
 \psi_3(\lambda_0, \lambda_1, \lambda_2) \phi_f(\lambda_1, \lambda_2) + (1-\psi_3)(\lambda_0, \lambda_1, \lambda_2) \phi_f(\lambda_0, \lambda_2)
 \right).
 \end{split}
\end{equation}

This decomposition~\eqref{Eqn=DecomposeFirst} is not yet optimal for our purposes. 
In Section~\ref{Sect=Bilinear}, we shall require that the symbols of the bilinear Toeplitz form Schur multipliers in our decomposition are odd (instead of even) homogeneous.  This in particular implies the vanishing of the paraproduct terms that occur in transference methods for the bilinear term, improving the bound on the norm of our Schur multiplier. In order to achieve this, we include an extra sign function in the three-variable terms, for which we compensate by including a sign function in the two-variable terms. Set
\begin{align*}
\begin{array}{rlrl}
\epsilon(\lambda, \mu) = & {\rm sign}(\mu - \lambda),&   \epsilon_1(\lambda_0, \lambda_1, \lambda_2) = & \sign(\lambda_1 - \lambda_0),\\
 \epsilon_2(\lambda_0, \lambda_1, \lambda_2) = & \sign(\lambda_2 - \lambda_1),  &\epsilon_3(\lambda_0, \lambda_1, \lambda_2) = & \sign(\lambda_2 - \lambda_0),
 \end{array}
\end{align*}
\noindent
where we use the convention that  $\sign(0) = 1$. Then we obtain the following decomposition, that we record here as a proposition.

\begin{proposition}
Let $f \in C^2(\mathbb{R})$  and let  $(\lambda_0, \lambda_1, \lambda_2) \in \mathbb{R}^3 \setminus \Delta$. Then,
\[
\begin{split}
 f^{[2]}(\lambda_0, \lambda_1, \lambda_2) = & \epsilon_1(\lambda_0, \lambda_1, \lambda_2) \widetilde{\theta}_1(\lambda_0, \lambda_1,   \lambda_2)   \psi_1(\lambda_0, \lambda_1, \lambda_2) \: \cdot \: \epsilon(\lambda_0, \lambda_1) \phi_f(\lambda_0, \lambda_1) \\
 &  + \epsilon_2(\lambda_0, \lambda_1, \lambda_2)
  \widetilde{\theta}_1(\lambda_0, \lambda_1,   \lambda_2)   (1-\psi_1)(\lambda_0, \lambda_1, \lambda_2)  \: \cdot \:  \epsilon(\lambda_1, \lambda_2) \mathring{\phi}_f(\lambda_1, \lambda_2)   \\
 & + \epsilon_3(\lambda_0, \lambda_1, \lambda_2)  \widetilde{\theta}_2(\lambda_0, \lambda_1,   \lambda_2)   \psi_2(\lambda_0, \lambda_1, \lambda_2) \: \cdot \:  \epsilon(\lambda_0, \lambda_2) \mathring{\phi}_f(\lambda_0, \lambda_2) \\
 & + \epsilon_1(\lambda_0, \lambda_1, \lambda_2) \widetilde{\theta}_2(\lambda_0, \lambda_1,   \lambda_2)   (1-\psi_2)(\lambda_0, \lambda_1, \lambda_2)  \: \cdot \: \epsilon(\lambda_0, \lambda_1) \mathring{\phi}_f(\lambda_0, \lambda_1)  \\
 & +  \epsilon_2(\lambda_0, \lambda_1, \lambda_2)  \widetilde{\theta}_3(\lambda_0, \lambda_1,   \lambda_2)
 \psi_3(\lambda_0, \lambda_1, \lambda_2)  \: \cdot \: \epsilon(\lambda_1, \lambda_2) \phi_f(\lambda_1, \lambda_2)\\
  & + \epsilon_3(\lambda_0, \lambda_1, \lambda_2)  \widetilde{\theta}_3(\lambda_0, \lambda_1,   \lambda_2)  (1-\psi_3)(\lambda_0, \lambda_1, \lambda_2)  \: \cdot \: \epsilon(\lambda_0, \lambda_2) \phi_f(\lambda_0, \lambda_2).
 \end{split}
\]
\end{proposition}
\begin{remark}
In the previous expression we separated the two-variable terms from the three-variable terms with a `$\cdot$'.
\end{remark}
For the corresponding Schur multipliers we find the following decomposition.

\begin{proposition}\label{Prop=SchurDecomposition}
Let $f \in C^2(\mathbb{R})$. For $x,y \in S_2$ we have
\begin{equation}\label{Eqn=MainDecompositionStart}
\begin{split}
 M_{f^{[2]}}(x,y) = &   M_{ \epsilon_1 \widetilde{\theta}_1  \psi_1 }(  M_{\epsilon \phi_f}(x), y ) +  M_{ \epsilon_2 \widetilde{\theta}_1  (1-\psi_1) }(  x, M_{\epsilon \mathring{\phi}_f }( y)  )  \\
 & +   M_{\epsilon \mathring{\phi}_f}(    M_{ \epsilon_3 \widetilde{\theta}_2  \psi_2 }( x, y )  ) +  M_{ \epsilon_1 \widetilde{\theta}_2  (1-\psi_2) }(  M_{ \epsilon \mathring{\phi}_f }( x),  y   )  \\
 & +      M_{ \epsilon_{2} \widetilde{\theta}_3  \psi_3 }( x, M_{\epsilon \phi_f}( y )  ) +  M_{\epsilon \phi_f}(  M_{ \epsilon_3 \widetilde{\theta}_3  (1-\psi_3) }( x,  y   )  ).  \\
\end{split}
\end{equation}
\end{proposition}
\begin{proof}
Note by Section~\ref{Sect=SchurPrelim} (or~\cite[Proposition 5]{CLS-AIF}) that all linear and bilinear Schur multipliers appearing in~\eqref{Eqn=MainDecompositionStart} are bounded as maps on  $S_2 \rightarrow S_2$ or $S_2 \times S_2 \rightarrow S_2$.
The proposition is now a consequence of a mild variation of~\cite[Lemma 3.2]{PSS-Inventiones}, which can  easily be verified directly in the same way.
\end{proof}

Now we outline our proof strategy for the next sections. All the linear Schur multipliers appearing in the decomposition~\eqref{Eqn=MainDecompositionStart} shall be estimated in Section~\ref{sect: linear_bddness_hms}. Each of the six summands in~\eqref{Eqn=MainDecompositionStart} contains a bilinear Schur multiplier. The last four of these summands  shall be estimated in  Section~\ref{Sect=S1Estimate}. The first two summands shall be estimated  in   Section~\ref{Sect=Bilinear}. In fact, the methods of   Section~\ref{Sect=Bilinear} can be used to estimate all six bilinear terms in~\eqref{Eqn=MainDecompositionStart}. However, the constants obtained in Section~\ref{Sect=S1Estimate} have better asymptotics, which is particularly relevant for the asymptotics for $p \searrow 1$ (as in Theorem~\ref{Thm=TheoremA_intro}) for the third and sixth summand.  

Strictly speaking, the sign functions $\epsilon$, $\epsilon_i$ in the last four summands of~\eqref{Eqn=MainDecompositionStart} are not needed for the estimates in Section~\ref{Sect=S1Estimate}. We have included them to show that these terms can also be estimated with the methods of Section~\ref{Sect=Bilinear}.


\section{Bounding linear terms with the H\"ormander-Mikhlin-Schur multiplier theorem}\label{sect: linear_bddness_hms}

In this section, we show the boundedness of the linear Schur multipliers $M_{\phi_f}$ and $M_{\mathring{\phi}_f}$  defined in Section~\ref{sect: decomposition fn}. Note that while the majority of this paper is concerned with second order divided difference functions, we will prove the results in this section for general $n$-th order divided difference functions.

We will use the following theorem.
\begin{theorem}[{\cite[Theorem A]{ParcetAnnals}}]\label{thrm: CGPT}
Let $\phi\in C^{\lfloor\tfrac{d}{2}\rfloor+1}(\mathbb{R}^{2d}\setminus\{\lambda=\mu\})$, $p\in(1,\infty)$, and let~$M_\phi$ be the Schur multiplier associated with $\phi$. Then \begin{equation*}
    \|M_\phi\|_{S_p\to S_p} \lesssim  p p^\ast   \triplevert \phi\triplevert_{\mathrm{HMS}}
\end{equation*}
with $\triplevert \phi\triplevert_{\mathrm{HMS}}:=\sum_{|\gamma|\le\lfloor\tfrac{d}{2}\rfloor+1}\|(\lambda,\mu)\mapsto|\lambda-\mu|^{|\gamma|}(|\partial_\lambda^{\gamma}\phi(\lambda,\mu)|+|\partial_\mu^{\gamma}\phi(\lambda,\mu)|)\|_{\infty}$.
\end{theorem}
We want to apply Theorem~\ref{thrm: CGPT} to multipliers with symbol $\phi_f(\lambda,\mu)=f^{[n]}(\lambda^{(k)},\mu^{(n+1-k)})$ for some $1\le k\le n$. Here, we use the notation introduced in Section~\ref{subsec: prelim_divdiff}. We need the following two lemmas.

\begin{lemma}\label{lemma: divdiff derivatives} Let $n\ge 1$, $0\le k \le n+1$, and let $f\in C^{n+1}(\mathbb{R})$. Then the partial derivatives of the map $(\lambda,\mu)\mapsto f^{[n]}(\lambda^{(k)},\mu^{(n+1-k)})$ are given by \begin{align*}
    &\partial_\lambda f^{[n]}(\lambda^{(k)},\mu^{(n+1-k)})=kf^{[n+1]}(\lambda^{(k+1)},\mu^{(n+1-k)}), \\
    &\partial_\mu f^{[n]}(\lambda^{(k)},\mu^{(n+1-k)})=(n+1-k)f^{[n+1]}(\lambda^{(k)},\mu^{(n+2-k)}).
\end{align*}
Furthermore, $\left((\lambda,\mu)\mapsto f^{[n]}(\lambda^{(k)},\mu^{(n+1-k)})\right)\in C^{1}(\mathbb{R}^2\setminus\{ \lambda=\mu \})$.
\end{lemma}
\begin{proof} Since $f^{[n]}$ is invariant under permutation of its variables, it is sufficient to calculate the partial derivatives in $\lambda$. For $n=1$, there are three cases to consider:
    \begin{itemize}
        \item $k=0$: $\partial_\lambda f^{[1]}(\mu,\mu)=0$.
        \item $k=2$: $\partial_\lambda f^{[1]}(\lambda,\lambda)=\partial_\lambda f'(\lambda)=f''(\lambda)=2f^{[2]}(\lambda,\lambda,\lambda)$, where we used Definition~\ref{def: divdiff}.
        \item $k=1$: We use the product rule to show \begin{align*}
            \partial_\lambda f^{[1]}(\lambda,\mu)=\partial_\lambda \frac{f(\lambda)-f(\mu)}{\lambda-\mu} = \frac{f'(\lambda)}{\lambda-\mu}-\frac{f(\lambda)-f(\mu)}{(\lambda-\mu)^2} = \frac{f^{[1]}(\lambda,\lambda)-f^{[1]}(\lambda,\mu)}{\lambda-\mu}= f^{[2]}(\lambda,\lambda,\mu).
        \end{align*}
    \end{itemize}
    By definition, continuity of $f^{[1]}$ follows from continuity of $f$. Furthermore, its derivatives are continuous in $\lambda\ne \mu$ by continuity of $f''$ and $f^{[1]}$.

Now let $n\in\mathbb{N}$. For $k=0$, the statement is immediate. For $0<k\le n+1$, we use the product rule and induction to show \begin{align*}
        &\partial_\lambda f^{[n]}(\lambda^{(k)},\mu^{(n+1-k)})\\
        =&\frac{\partial_\lambda (f^{[n-1]}(\lambda^{(k)},\mu^{(n-k)}) -  f^{[n-1]}(\lambda^{(k-1)},\mu^{(n+1-k)}))}{\lambda-\mu}-\frac{f^{[n-1]}(\lambda^{(k)},\mu^{(n-k)}) - f^{[n-1]}(\lambda^{(k-1)},\mu^{(n+1-k)})}{(\lambda-\mu)^2} \\
        =& \frac{k f^{[n]}(\lambda^{(k+1)},\mu^{(n-k)}) - (k-1) f^{[n]}(\lambda^{(k)},\mu^{(n+1-k)})-f^{[n]}(\lambda^{(k)},\mu^{(n+1-k)})}{\lambda-\mu} \\
        =& kf^{[n+1]}(\lambda^{(k+1)},\mu^{(n+1-k)}).
    \end{align*}
    Continuity of $(\lambda,\mu)\mapsto f^{[n]}(\lambda^{(k)},\mu^{(n+1-k)})$ in $\lambda\ne \mu$ follows by induction from continuity of the corresponding $f^{[n-1]}$-terms. As in the base case, continuity of its first derivatives in $\lambda\ne \mu$ follows from continuity of $f^{(n+1)}$ and $f^{[n]}$. 
\end{proof}
\begin{lemma}\label{lemma: CGPT estimates}
    For $n\in\mathbb{N}$, $0\le k \le n+1$, $0\le\gamma\le\min\{k,n+1-k\}$, and $(\lambda,\mu)\in\mathbb{R}^2\setminus\{\lambda=\mu\}$, $$|\lambda-\mu|^{\gamma}|\partial_\lambda^{\gamma}f^{[n]}(\lambda^{(k)},\mu^{(n+1-k)})|\le 2^{\gamma}\frac{(k+\gamma-1)!}{(k-1)!}\frac{\|f^{(n)}\|_{\infty}}{n!}.$$
\end{lemma}
\begin{proof} For $\gamma=0$, this statement is immediate from~\eqref{Eqn=DiffDiff}. Let now $0<\gamma\le\min\{k,n+1-k\}$.
    By repeatedly applying Lemma~\ref{lemma: divdiff derivatives}, we obtain\begin{equation*}
        \partial_\lambda^{\gamma}f^{[n]}(\lambda^{(k)},\mu^{(n+1-k)})=\frac{(k+\gamma-1)!}{(k-1)!}f^{[n+\gamma]}(\lambda^{(k+\gamma)},\mu^{(n+1-k)}).
    \end{equation*} We now decompose $f^{[n+\gamma]}$ by applying the definition of divided difference functions multiple times as \begin{align*}
         f^{[n+\gamma]}(\lambda^{(k+\gamma)},\mu^{(n+1-k)})&=\frac{1}{\lambda-\mu}\left(f^{[n+\gamma-1]}(\lambda^{(k+\gamma)},\mu^{(n-k)}) - f^{[n+\gamma-1]}(\lambda^{(k+\gamma-1)},\mu^{(n+1-k)})\right) \\
        &= \dotsc \\
        &= \frac{1}{(\lambda-\mu)^\gamma} \sum_{j=0}^{\gamma} (-1)^j \binom{\gamma}{j} f^{[n]}(\lambda^{(k+\gamma-j)},\mu^{(n+1-k-(\gamma - j))}).
    \end{align*}
    Using the estimate $\|f^{[n]}\|_{\infty}\le \tfrac{\|f^{(n)}\|_{\infty}}{n!}$ from~\eqref{Eqn=DiffDiff}, we conclude \begin{align*}
        |\lambda-\mu|^{\gamma}|\partial_\lambda^{\gamma}f^{[n]}(\lambda^{(k)},\mu^{(n+1-k)})| &\le \frac{(k+\gamma-1)!}{(k-1)!} \sum_{j=0}^{\gamma} \binom{\gamma}{j} |f^{[n]}(\lambda^{(k+\gamma-j)},\mu^{(n+1-k-(\gamma - j))})| \\ &\le \frac{(k+\gamma-1)!}{(k-1)!}\sum_{j=0}^{\gamma} \binom{\gamma}{j} \frac{\|f^{(n)}\|_{\infty}}{n!} =2^{\gamma} \frac{(k+\gamma-1)!}{(k-1)!}\frac{\|f^{(n)}\|_{\infty}}{n!}.
    \end{align*}
\end{proof}
Altogether we can now show the following.

\begin{theorem}\label{thrm: divdiff_two_variables_bdd_multiplier}
    Let $n\in\mathbb{N}$, $f\in C^{n}(\mathbb{R})$, $1\le k\le n$, and $p\in(1,\infty)$. Set $\phi_f(\lambda,\mu):=f^{[n]}(\lambda^{(k)},\mu^{(n+1-k)})$. Then $$\|M_{\phi_f}\|_{S_p\to S_p}\lesssim \frac{2n+3}{n!}   p p^\ast  \|f^{(n)}\|_{\infty}.$$
\end{theorem}
\begin{proof}
    We can apply Theorem~\ref{thrm: CGPT}, since $\phi_f\in C^{1}(\mathbb{R}^2\setminus\{\lambda=\mu\})$ by Lemma~\ref{lemma: divdiff derivatives}. From Lemma~\ref{lemma: CGPT estimates}, we conclude 
    \begin{align*}
        \triplevert \phi_f \triplevert_{\mathrm{HMS}}&\le \|\phi_f\|_{\infty}+\| (\lambda,\mu)\mapsto|\lambda-\mu|\partial_\lambda \phi_f(\lambda,\mu)\|_{\infty} +\|(\lambda,\mu)\mapsto |\lambda-\mu|\partial_\mu \phi_f(\lambda,\mu)\|_{\infty}\\
        &\le (1 + 2k + 2(n+1-k))\frac{\|f^{(n)}\|_{\infty}}{n!}=\frac{2n+3}{n!}\|f^{(n)}\|_{\infty}.
    \end{align*}
\end{proof}

\begin{remark}\label{Rmk=LinearCase}
Recall that we set $\epsilon(\lambda, \mu) = \sign(\mu - \lambda)$. Under the assumptions of  Theorem
\ref{thrm: divdiff_two_variables_bdd_multiplier} if follows also that
\[
\|M_{\epsilon \phi_f}\|_{S_p\to S_p}  \lesssim \frac{2n+3}{n!}   p p^\ast  \|f^{(n)}\|_{\infty}.
\]
Indeed, $\epsilon \phi_f $   satisfies the same H\"ormander-Mikhlin differentiability criteria as $\phi_f $,  so that we may appeal again to Theorem~\ref{thrm: CGPT}.
\end{remark}

\section{Bilinear Schur multipliers that map to $S_1$} \label{Sect=S1Estimate}

The aim of this section is to estimate the last four of the bilinear Schur multipliers occuring in the six summands of \eqref{Prop=SchurDecomposition}. It turns out that these Schur multipliers are special, as they admit an $S_1$-bound.

\begin{theorem} \label{Thm=S1Estimate}
Let  $m: \mathbb{R}^2 \setminus \{ 0 \} \rightarrow \mathbb{C}$ be smooth and homogeneous with support contained in one of the four  quadrants $\sigma_1 \mathbb{R}_{>0} \times  \sigma_2 \mathbb{R}_{>0} $,  where $\sigma_j \in \{ +, -\}$.  Define $\widetilde{m}$ as in~\eqref{eqn=theta_tilde}. Then for every $1 \leq p < \infty$, $1 < p_1, p_2 < \infty$ with $\frac{1}{p_1} +\frac{1}{p_2} = \frac{1}{p}$ we have
\[
\Vert M_{\widetilde{m}}:   S_{p_1}  \times S_{p_2} \rightarrow S_p \Vert \lesssim C(m) p_1 p_1^\ast p_2 p_2^\ast,
\]
for a constant  $C(m) > 0$ only depending on $m$. 
 \end{theorem} 
 \begin{proof}
For simplicity assume that $\sigma_2 = +$, as the other case can be treated similarly. Set then  $\rho(\lambda) = m(\lambda, 1)$, $\lambda \in \mathbb{R}$. Then $\rho(\xi_1/\xi_2) = m(\xi_1/\xi_2, 1) = m(\xi_1, \xi_2)$, where the last equality follows as~$m$ is homogeneous and supported on $(\xi_1, \xi_2)$ with $\xi_2$ positive. Further, note once more that  $m$ is homogeneous and thus constant on rays. Since its support is a closed set contained in the quadrant  $\sigma_1 \mathbb{R}_{>0} \times    \mathbb{R}_{>0}$, it must thus be a proper radial subsector of that quadrant.  Therefore,  it follows that $\rho$ has compact support contained in $\sigma_1 (0, \infty)$. In particular, $\rho$ is a compactly supported  Schwartz function. 
 
It follows that the function $t \mapsto \rho(\sigma_1 e^t)$, $t \in \mathbb{R}$ is Schwartz. So using  Fourier inversion we write 
\[
\rho( \sigma_1 e^t ) = \int_\mathbb{R} g(s) e^{ist} ds
\]
with $g$ a Schwartz function.  Substitute $t = \log(\sigma_1 \xi_1/\xi_2)$, where $(\xi_1, \xi_2) \in  \sigma_1 \mathbb{R}_{>0} \times    \mathbb{R}_{>0} $. This gives 
\[
m(\xi_1, \xi_2) = \rho( \frac{\xi_1}{\xi_2}) = \int_\mathbb{R} g(s) \vert \xi_1 \vert^{is}  \xi_2^{-is} ds, \quad (\xi_1, \xi_2) \in  \sigma_1 \mathbb{R}_{>0} \times    \mathbb{R}_{>0}. 
\]
Let 
\[
\begin{split}
k_s^1(\lambda_0, \lambda_1) = &  \left\{
\begin{array}{ll}
\vert \lambda_1 - \lambda_0 \vert^{is}, &  \textrm{ if }   \sigma_1 (\lambda_1 - \lambda_0) >0, \\
0, & \textrm{ otherwise.}
\end{array} 
\right. \\
k_s^2(\lambda_1, \lambda_2) = &  \left\{
\begin{array}{ll}
(\lambda_2 - \lambda_1)^{is}, & \textrm{ if }    (\lambda_2 - \lambda_1) >0, \\
0, & \textrm{ otherwise.}
\end{array} 
\right.
\end{split}
\] 
Hence
\[
\widetilde{m}(\lambda_0, \lambda_1, \lambda_2) = \int_\mathbb{R} g(s)  k_s^1(\lambda_0, \lambda_1) k_{-s}^2(\lambda_1, \lambda_2)    ds.
\]
It then follows that 
\[
M_{\widetilde{m} }(x,y) = \int_\mathbb{R} g(s) M_{k_s^1  }(x)   M_{  k_{-s}^2 }(y) ds. 
\]
Note that $\triplevert k_s^1 \triplevert_{{\rm HMS}} = \triplevert k_s^2 \triplevert_{{\rm HMS}} = 1+2 \vert s \vert$.  
Thus  by Theorem~\ref{thrm: CGPT},
\[
\begin{split}
\Vert M_{\widetilde{m}}:   S_{p_1}  \times S_{p_2} \rightarrow S_p \Vert
\lesssim & \int_\mathbb{R}  g(s) (1+ 2\vert s \vert)^2  ds  \: p_1 p_1^\ast p_2 p_2^\ast. 
\end{split}
\]
This concludes the proof. 
 \end{proof}

 For the following corollary we recall the notation from Proposition~\ref{Prop=SchurDecomposition}. 

\begin{corollary} \label{Cor=S1Estimate}
Let  $a_3 := \epsilon_3 \widetilde{\theta}_2 \psi_2$, $a_4 := \epsilon_1 \widetilde{\theta}_2 (1 - \psi_2)$, $a_5 := \epsilon_2 \widetilde{\theta}_3 \psi_3$ and $a_6 := \epsilon_3 \widetilde{\theta}_3(1- \psi_3)$.  Then for every $1 \leq p < \infty, 1 < p_1, p_2 < \infty$ with $\frac{1}{p_1} +\frac{1}{p_2} = \frac{1}{p}$ we have
\[
\Vert M_{a_j}:   S_{p_1}  \times S_{p_2} \rightarrow S_p \Vert \lesssim C(a_j) p_1 p_1^\ast p_2 p_2^\ast, \quad 3 \leq j \leq 6,
\]
for a constant $C(a_j) >0$ only depending on $a_j$.
 \end{corollary} 
 \begin{proof}
Each of the functions $a_j$ is smooth, homogeneous, of Toepliz form, and supported on one of the two quadrants $-\sigma \mathbb{R}_{>0} \times \sigma \mathbb{R}_{>0}$ with $\sigma  \in \{+, -\}$.\ Therefore the conclusion follows from Theorem~\ref{Thm=S1Estimate}. 
 \end{proof}

\begin{remark}
The constant $C(a_j)$ depends in particular on the choice of $\epsilon >0$ in Section~\ref{sect: decomposition fn}, see~\eqref{Eqn=Regions}. Note that we cannot expect a  bound as in  Corollary~\ref{Cor=S1Estimate}  that is  uniform as $\epsilon \searrow 0$, since in~\cite[Theorem 5.3]{CKV} and its proof it is shown that  such Schur multipliers do not map to $S_1$. 
\end{remark}

\section{Bilinear transference}\label{Sect=Bilinear}
The aim of this section is to estimate the remaining bilinear terms occuring in Proposition~\ref{Prop=SchurDecomposition}. The crucial observation is that these multipliers are of Toeplitz form and therefore, using bilinear transference techniques, can be estimated by Fourier multipliers and Calder\'on-Zygmund operators.

\subsection{Bilinear Calder\'on-Zygmund operators and Fourier multipliers}

We say $K: \mathbb{R}^2 \rightarrow \mathbb{C}$ satisfies the \emph{size condition} if for some constant $C_1>0$ we have
\begin{equation}\label{Eqn=Size}
\vert K(z) \vert \leq \frac{C_1}{\vert z \vert^2}, \quad z \in \mathbb{R}^2 \setminus \{ 0 \}.
\end{equation}
We say that $K$ satisfies the \emph{smoothness condition} if $K$ is continuously differentiable on $\mathbb{R}^2 \setminus \{ 0 \}$ and there exists some constant $C_2 > 0$ such that
\begin{equation}\label{Eqn=Smooth}
\vert \nabla K(z) \vert \leq \frac{C_2}{\vert z \vert^3}, \quad z \in \mathbb{R}^2 \setminus \{ 0 \}.
\end{equation}
Set $\widetilde{K}(x,y,z) = K(x-y, x-z)$, $x,y,z \in \mathbb{R}$. If $K$ satisfies~\eqref{Eqn=Size} and~\eqref{Eqn=Smooth}, then
\begin{equation}\label{Eqn=Size2}
\vert \widetilde{K}(x,y,z) \vert \leq \frac{C_1}{ (\vert x-y  \vert + \vert x - z \vert)^2}, \quad (x,y,z) \in \mathbb{R}^3 \setminus \Delta,
\end{equation}
and it follows  from the chain rule that
\begin{equation}\label{Eqn=Smooth2}
\vert \nabla \widetilde{K}(x,y,z) \vert \leq \frac{C_2}{ (\vert x-y  \vert + \vert x - z \vert)^3}, \quad (x,y,z) \in \mathbb{R}^3 \setminus \Delta.
\end{equation}
It is assumingly well-known  (see e.g.\ the introduction of~\cite{GrafakosTorres}) that Condition~\ref{Eqn=Smooth2} implies the following more general condition. We provide a proof for completeness as we did not find it in the literature.

\begin{lemma}\label{Lemma=KernelEasierCondition}
Suppose that $K$ satisfies  the smoothness condition.
Let $(x_1,x_2,x_3) \in \mathbb{R}^3 \setminus \Delta$.  Let  $\widetilde{x}_j \in \mathbb{R}$, $j =1,2,3$, be such that
\begin{equation}\label{Eqn=CZDomainCondition}
\vert x_j - \widetilde{x}_j \vert \leq \frac{1}{2} \max(  \vert x_1 - x_2 \vert, \vert x_1 - x_3 \vert ).
\end{equation} 
Then,
\[
\begin{split}
\vert  \widetilde{K}(x_1, x_2, x_3) -   \widetilde{K}(\widetilde{x}_1, x_2, x_3 ) \vert \lesssim &  \frac{  \vert x_1 - \widetilde{x}_1 \vert } { (\vert x_1 - x_2  \vert + \vert x_1 - x_3 \vert )^3 },\\
\vert   \widetilde{K}(x_1, x_2, x_3) -   \widetilde{K}(x_1, \widetilde{x}_2, x_3 ) \vert \lesssim    &  \frac{  \vert x_2 - \widetilde{x}_2 \vert } { (\vert x_1 - x_2  \vert + \vert x_1 - x_3 \vert )^3 },\\
  \vert   \widetilde{K}(x_1, x_2, x_3) -   \widetilde{K}(x_1, x_2, \widetilde{x}_3 ) \vert  \lesssim &  \frac{  \vert x_3 - \widetilde{x}_3 \vert } { (\vert x_1 - x_2  \vert + \vert x_1 - x_3 \vert )^3 }.
\end{split}
\]
\end{lemma}
\begin{proof}
We only prove the first estimate, the other two are proved in a similar way. It suffices to prove the case  $x_2 \not = x_3$, since $\mathbb{R}^3  \setminus \{ x_2 = x_3 \}$ is a dense subset of $\mathbb{R}^3 \setminus \Delta$ and  $K$ is continuous.  Take $x_1'$ in the interval $[x_1, \widetilde{x}_1]$ (or in $[ \widetilde{x}_1, x_1]$ in case $x_1 > \widetilde{x}_1$ ) such that
\[
\begin{split}
\vert  \widetilde{K}(x_1, x_2, x_3) -   \widetilde{K}(\widetilde{x}_1, x_2, x_3 ) \vert =  &
 \vert x_1 - \widetilde{x}_1 \vert \vert \partial_1  \widetilde{K} (x_1', x_2, x_3)  \vert.
\end{split}
\]
But then the assumptions~ \eqref{Eqn=Size2} and~\eqref{Eqn=CZDomainCondition} imply that
\begin{align*}
    \vert  \widetilde{K}(x_1, x_2, x_3) -   \widetilde{K}(\widetilde{x}_1, x_2, x_3 ) \vert \lesssim 
  \frac{  \vert x_1 - \widetilde{x}_1 \vert } { (\vert {x_1'} - x_2  \vert + \vert x_1 - x_3 \vert )^3 } 
\lesssim  \frac{  \vert x_1 - \widetilde{x}_1 \vert } { (\vert x_1 - x_2  \vert + \vert x_1 - x_3 \vert )^3 },
\end{align*}
 where the second inequality follows from~\eqref{Eqn=CZDomainCondition} since \begin{equation*}
    |x_1-x_2| \le |x_1-x_1'|+|x_1'-x_2| \le |x_1-\widetilde{x}_1|+|x_1'-x_2| \le \frac{1}{2}|x_1-x_2| + |x_1'-x_2|.
\end{equation*}
\end{proof}

Lemma~\ref{Lemma=KernelEasierCondition} shows that the conditions~\eqref{Eqn=Size} and~\eqref{Eqn=Smooth} imply that the kernel $\widetilde{K}$ satisfies the size and smoothness conditions appearing in~\cite{DiPlinioMathAnn}. Next, we show that for bilinear Fourier multipliers with odd homogeneous symbols, their associated Calder\'on-Zygmund kernels satisfy these criteria. Recall that the Fourier transform $\mathcal{F}$ was defined in the preliminaries~\eqref{Eqn=FourierTransform} in a distributional sense.

\begin{proposition}\label{Prop=SizeSmooth}
Let $m: \mathbb{R}^2 \setminus \{ 0 \} \rightarrow \mathbb{R}$ be smooth and odd homogeneous, and set ${m(0,0) = 0}$.  Then $\mathcal{F}m: \mathbb{R}^2  \rightarrow \mathbb{C}$ is a function satisfying conditions~\eqref{Eqn=Size},~\eqref{Eqn=Smooth}, and ${(\mathcal{F}m)(0,0)=0}$.
\end{proposition}
\begin{proof}
The proof is essentially  the same as~\cite[Lemma 4.3]{CPSZ} but for the convenience of the reader we give it here. We identify $\mathbb{R}^2$ with $\mathbb{C}$. Since $m$ is smooth on the circle, we may write
\[
m(e^{i \theta}) = \sum_{k \in \mathbb{Z}} \alpha_k e^{ik \theta}, \quad \theta \in [0, 2\pi),
\]
where the Fourier coefficients $\alpha_k$ decay faster than any polynomial. As $m$ is odd, it has mean zero on the circle, and thus $\alpha_0 = 0$.  It follows that
\[
m  = \sum_{0 \not = k \in \mathbb{Z}} \alpha_k g_k, \quad g_k(z) = \frac{z^k}{\vert z \vert^k}, \quad 0 \not = z \in \mathbb{C}.
\]
We have for $k \not = 0$ that $(\mathcal{F} g_k)(0) = 0$, and as in~\cite[Lemma 4.3]{CPSZ} one can show that
\[
(\mathcal{F} g_k)(z) = \frac{\vert k \vert}{2 \pi i^k} \frac{z^k}{\vert z \vert^{k+2}}, \quad 0 \not = z \in \mathbb{C}.
\]
Hence $(\mathcal{F}m)(0) = 0$ and
\[
(\mathcal{F}m)(z) =   \sum_{0 \not = k \in \mathbb{Z}} \frac{\vert k \vert \alpha_k  }{  2 \pi i^k}  \frac{z^k}{\vert z \vert^{k+2}},  \quad 0 \not = z \in \mathbb{C}.
\]
As the coefficients $\vert k \vert \alpha_k$ are summable it follows therefore that
\[
\vert (\mathcal{F}m)(z) \vert \approx O(\vert z \vert^{-2}), \quad \quad
  \vert \nabla (\mathcal{F}m)(z) \vert \approx O(\vert z \vert^{-3}),
  \]
  which finishes the proof.

\end{proof}

\begin{proposition}\label{Prop=FourierCZ}
Let $m: \mathbb{R}^2 \setminus \{ 0 \} \rightarrow \mathbb{R}$ be smooth, odd,  homogeneous, and set ${m(0,0) = 0}$. Then the Fourier multiplier~$T_m$ is a bilinear Calder\'on-Zygmund operator with kernel $-(2\pi)^{-1}\widetilde{\mathcal{F}m}$, see the definition  below~\eqref{Eqn=Smooth}. More precisely, for Schwartz functions $f_1, f_2$ we have  
\begin{equation}\label{Eqn=KernelExpansion}
   T_m(f_1, f_2)(x) = -\frac{1}{2\pi}  \int_\mathbb{R} \int_\mathbb{R}  (\mathcal{F}m)(x - y, x- z)  f_1(y) f_2(z) dy dz, \quad x \in \mathbb{R} \setminus (\supp(f_1) \cap  \supp(f_2)).
\end{equation}
\end{proposition}
\begin{proof}
We have for $x \in \mathbb{R} \setminus (\supp(f_1) \cap \supp(f_2))$ that
\[
\begin{split}
   T_m(f_1, f_2)(x)   = &  \frac{1}{2\pi}    \int_{\mathbb{R}}  \int_{\mathbb{R}} m(\xi_1, \xi_2) (\mathcal{F} f_1)(\xi_1)  ( \mathcal{F} f_2 )(\xi_2) e^{i (\xi_1 + \xi_2) x}    d\xi_1 d\xi_2 \\
   = & \frac{1}{2\pi} \int_{\mathbb{R}}  \int_{\mathbb{R}} m(\xi_1, \xi_2) ((\mathcal{F} f_1)(\xi_1)  e^{i \xi_1  x} ) ( ( \mathcal{F} f_2 )(\xi_2) e^{i \xi_2  x})    d\xi_1 d\xi_2 \\
   = & \frac{1}{2\pi} \int_{\mathbb{R}}  \int_{\mathbb{R}} m(\xi_1, \xi_2) (\mathcal{F} f_1( \: \cdot \: + x))(\xi_1)     ( \mathcal{F} f_2( \: \cdot \: + x) )(\xi_2)     d\xi_1 d\xi_2 \\
   = & \frac{1}{2\pi} \int_{\mathbb{R}}  \int_{\mathbb{R}}(\mathcal{F} m)(\xi_1, \xi_2) f_1( \xi_1 + x)      f_2( \xi_2   + x)     d\xi_1 d\xi_2 \\
   = & \frac{1}{2\pi} \int_{\mathbb{R}}  \int_{\mathbb{R}}(\mathcal{F} m)(\xi_1-x, \xi_2-x) f_1( \xi_1)      f_2( \xi_2  )    d\xi_1 d\xi_2. \\
  \end{split}
\] 
As $m$ is odd so is $\mathcal{F}m$, hence we conclude~\eqref{Eqn=KernelExpansion}.

To show that $T_m$ is indeed a Calder\'on-Zygmund operator as defined in Section~\ref{subsec: fm_czo}, it remains to show conditions~\eqref{Eqn=Size}, \eqref{Eqn=Smooth}, and boundedness of $T_m$. The first two of these conditions hold by Proposition~\ref{Prop=SizeSmooth}. Finally, the boundedness condition follows from~\cite[Theorem 8]{KenigStein}.
\end{proof}

\begin{remark}\label{rem: T(1)=0}
For Calder\'on-Zygmund operators $T$ on $\mathbb{R}$ with a convolution kernel $$\widetilde{K}(x,y_1,\dotsc,y_n)=K(x-y_1,\dotsc,x-y_n), \quad x, y_1, \ldots, y_n \in \mathbb{R},$$it holds that $\langle T(1,\dotsc,1),\phi \rangle =0$ for all $\phi\in L_c^{\infty}(\mathbb{R})$ with $\int_{\mathbb{R}} \phi dx = 0$, i.e.\ $T(1,\dotsc,1)$ vanishes in $\mathrm{BMO}$. As is common in the literature, we will refer to this as ``$T(1,\dotsc,1)=0$''. We decided to omit the detailed proof of this fact as it is commonly used in the literature. We refer the reader to  the last equation in the proof of~\cite[Proposition 6]{GrafakosTorres} which applies to our situation; though we note that the proof there is only formal. Similarly, all partial adjoints $T^{\ast 1}$ and $T^{\ast 2}$ of $T$  (defined via $\langle T^{\ast 1} (f,g),h\rangle := \langle T(h,g),f\rangle$, $\langle T^{\ast 2} (f,g),h\rangle := \langle T(f,h),g\rangle$, see~\cite{diplinio_multilinear_czo}) vanish for these operators. See e.g.\ \cite{bilinrepthrm,diplinio_multilinear_czo} for well-defined constructions of these expressions. 
Hence in particular, for a bilinear Calder\'on-Zygmund operator with convolution kernel, it holds that $\langle T(1, 1),\phi \rangle=\langle T^{\ast 1}(1, 1),\phi \rangle = \langle T^{\ast 2}(1, 1),\phi \rangle  =0$ for all $\phi\in L_c^{\infty}(\mathbb{R})$ with $\int_{\mathbb{R}} \phi dx = 0$. 
\end{remark}

\begin{auxproof}
For a bounded bilinear map $T: X_1 \times X_2 \rightarrow Y$  between Banach spaces $X_1, X_2$ and $Y$ we set the partial adjoints $T^{\ast 1}$ and $T^{\ast 2}$ as the maps $T^{\ast 1}: Y^\ast \times X_2 \rightarrow X_1^\ast$  and   $T^{\ast 2}: X_1 \times Y^\ast \rightarrow X_2^\ast$  determined by $T^{\ast 1}(f, x_2)(x_1) = f(T(x_1, x_2))$ and $T^{\ast 2}(x_1, f)(x_2) = f(T(x_1, x_2))$. That is, $T^{\ast 1}( \: \cdot \: , x_2)$  is the adjoint of  $T( \: \cdot \:  , x_2)$ for every $x_2 \in X_2$.

\begin{definition}\label{def: t11}[reference: bilinear rep thrm sect 2.2]
Let $T$ be a bilinear Calder\'on-Zygmund operator with kernel $K$. Fix $\varepsilon>0$. Let $Q$ be a cube in $\mathbb{R}$ and let $\phi\in L^{\infty}(\mathbb{R})$ with $\mathrm{supp}\;\phi\subset Q$ and $\int_{\mathbb{R}}\phi(x)dx=0$. Let $C=C_{\varepsilon}\ge 3$ be sufficiently large such that $|x-y|>\varepsilon$ for all $x\in Q$, $y\notin CQ$. Choose $c_Q\in\mathbb{R}$ to be the midpoint of the one-dimensional cube $Q$. Finally, let $Q_{\varepsilon}(x):=\{(y,z)\in\mathbb{R}\mid \max(|x-y|,|x-z|)\le\varepsilon\}$. We may then define
\begin{multline}\label{eqn=t(1,1)_def}
    \langle T_{\varepsilon,C}(1,1),\phi\rangle := \int_{\mathbb{R}} \int_{\mathbb{R}^2\setminus Q_{\varepsilon}(x)} K(x,y,z)1_{CQ}(y)1_{CQ}(z)\phi(x)dydzdx \\ +  \int_{\mathbb{R}} \int_{\mathbb{R}^2\setminus Q_{\varepsilon}(x)} (K(x,y,z)-K(c_Q,y,z))1_{(CQ\times CQ)^c}(y,z)\phi(x)dydzdx. 
\end{multline}

We say that $T(1,1)=0$ if $\langle T_{\varepsilon,C}(1,1),\phi\rangle=0$ for all $\varepsilon>0$ and all $Q$, $\phi$, $C$, $c_Q$ as above. 
\end{definition}
\begin{proposition}\label{Prop=VanishingParaproducts}
In Proposition~\ref{Prop=FourierCZ} we have $T_m(1,1) = T_m^{\ast 1}(1,1) = T_m^{\ast 2}(1,1) =  0$.
\end{proposition}
\begin{proof}
First note that the definition above is independent of the choice of $C$, as long as it is sufficiently large. Indeed, let $C$ be such a constant and let $\Tilde{C}>C$. Then by subtracting the corresponding terms, we have\begin{multline*}
    \int_{\mathbb{R}} \int_{\mathbb{R}^2\setminus Q_{\varepsilon}(x)} K(x,y,z)(1_{\Tilde{C}Q}(y)1_{\Tilde{C}Q}(z)-1_{CQ}(y)1_{CQ}(z))\phi(x)dydzdx\\ = \int_{\mathbb{R}} \int_{\mathbb{R}^2\setminus Q_{\varepsilon}(x)} K(x,y,z)1_{(\Tilde{C}Q\times\Tilde{C}Q)\setminus (CQ\times CQ)}(y,z)\phi(x)dydzdx 
\end{multline*}
as well as \begin{multline*}
    \int_{\mathbb{R}} \int_{\mathbb{R}^2\setminus Q_{\varepsilon}(x)} (K(x,y,z)-K(c_Q,y,z))(1_{(\Tilde{C}Q\times \Tilde{C}Q)^c}(y,z)-1_{(CQ\times CQ)^c}(y,z))\phi(x)dydzdx \\=
    - \int_{\mathbb{R}} \int_{\mathbb{R}^2\setminus Q_{\varepsilon}(x)} (K(x,y,z)-K(c_Q,y,z))1_{((\Tilde{C}Q\times\Tilde{C}Q)\setminus (CQ\times CQ))^c}(y,z)\phi(x)dydzdx \\
    = - \int_{\mathbb{R}} \int_{\mathbb{R}^2\setminus Q_{\varepsilon}(x)} K(x,y,z)1_{((\Tilde{C}Q\times\Tilde{C}Q)\setminus (CQ\times CQ))^c}(y,z)\phi(x)dydzdx,
\end{multline*}
where in the last equality we used Fubini's theorem and $\int_{\mathbb{R}}\phi(x)dx=0$.
Hence we conclude $\langle T_{\varepsilon,\Tilde{C}}(1,1),\phi\rangle-\langle T_{\varepsilon,C}(1,1),\phi\rangle=0$. From now we will therefore suppress the subscript $C$ from the notation.

We will now show that both integral terms in Definition~\ref{def: t11} vanish if $K$ is of convolution type and satisfies~\eqref{Eqn=Size2} and~\eqref{Eqn=Smooth2}. Here, we follow~\cite{bilinrepthrm} in further approximating $T_{\varepsilon}$ by considering $T_{\varepsilon_1,\varepsilon_2}:=T_{\varepsilon_1}-T_{\varepsilon_2}$ for $\varepsilon_2>\varepsilon_1$ and letting $\varepsilon_2\to\infty$. 

For the first integral term in~\eqref{def: t11}, choose $C$ large enough such that $Q_{\varepsilon_2}(x)\subset CQ$ for all $x\in Q$. In this case, using $K(x,y,z)=k(x-y,x-z)$ for some $k$, a coordinate transformation yields

\begin{align*}
    &\int_{\mathbb{R}} \int_{Q_{\varepsilon_2}(x)\setminus Q_{\varepsilon_1}(x)} K(x,y,z)1_{CQ}(y)1_{CQ}(z)\phi(x)dydzdx \\
    &=\int_{\mathbb{R}} \int_{Q_{\varepsilon_2}(x)\setminus Q_{\varepsilon_1}(x)} k(x-y,x-z)1_{CQ}(y)1_{CQ}(z)\phi(x)dydzdx  \\
    &= \int_{\mathbb{R}} \int_{Q_{\varepsilon_2}(0)\setminus Q_{\varepsilon_1}(0)} k(\Tilde{y},\Tilde{z})1_{CQ}(x-\Tilde{y})1_{CQ}(x-\Tilde{z})\phi(x)d\Tilde{y}d\Tilde{z}dx.  
\end{align*}

By our choice of $C$, the indicator functions $1_{CQ}(x-\cdot)$ are equal to $1$ for all $x\in Q$, hence 

\begin{equation*}
    \int_{\mathbb{R}} \int_{Q_{\varepsilon_2}(0)\setminus Q_{\varepsilon_1}(0)} k(\Tilde{y},\Tilde{z})1_{CQ}(x-\Tilde{y})1_{CQ}(x-\Tilde{z})\phi(x)d\Tilde{y}d\Tilde{z}dx = \int_{\mathbb{R}} \int_{Q_{\varepsilon_2}(0)\setminus Q_{\varepsilon_1}(0)} k(\Tilde{y},\Tilde{z})\phi(x)d\Tilde{y}d\Tilde{z}dx. 
\end{equation*}

By the size condition~\eqref{Eqn=Size2}, we may use Fubini's theorem and obtain 
\begin{equation*}
    \int_{\mathbb{R}} \int_{Q_{\varepsilon_2}(0)\setminus Q_{\varepsilon_1}(0)} k(\Tilde{y},\Tilde{z})\phi(x)d\Tilde{y}d\Tilde{z}dx =  \int_{Q_{\varepsilon_2}(0)\setminus Q_{\varepsilon_1}(0)} k(\Tilde{y},\Tilde{z})d\Tilde{y}d\Tilde{z}\int_{\mathbb{R}}\phi(x)dx.
\end{equation*}
The integral on the right vanishes by the definition of $\phi$. The integral on the left is bounded for all $\varepsilon_2>\varepsilon_1$ since \begin{equation*}
    \int_{Q_{\varepsilon_2}(0)\setminus Q_{\varepsilon_1}(0)} k(\Tilde{y},\Tilde{z})d\Tilde{y}d\Tilde{z}  \le C_1 \int_{Q_{\varepsilon_2}(0)\setminus Q_{\varepsilon_1}(0)} \frac{1}{(|\Tilde{y}|+|\Tilde{z}|)^2}d\Tilde{y}d\Tilde{z} \le C_1\frac{\varepsilon_2^2-\varepsilon_1^2}{\varepsilon_1^2} < \infty.
\end{equation*}
Hence $$\int_{\mathbb{R}} \int_{Q_{\varepsilon_2}(x)\setminus Q_{\varepsilon_1}(x)} K(x,y,z)1_{CQ}(y)1_{CQ}(z)\phi(x)dydzdx=0$$ for all $\varepsilon_2>\varepsilon_1$, allowing us to conclude $$\int_{\mathbb{R}} \int_{\mathbb{R}^2\setminus Q_{\varepsilon}(x)} K(x,y,z)1_{CQ}(y)1_{CQ}(z)\phi(x)dydzdx=0$$ for all $\varepsilon>0$.

To estimate the second integral in~\eqref{eqn=t(1,1)_def}, we use the smoothness condition~\eqref{Eqn=Smooth2}. Note that by construction, if $(y,z)\in (CQ\times CQ)^c$, then $$|x-c_Q|\le \frac{|Q|}{2} \le \frac{1}{2}\max(|x-y|,|x-z|).$$ We obtain \begin{align*}
    \quad&\left|\int_{\mathbb{R}} \int_{\mathbb{R}^2\setminus Q_{\varepsilon}(x)} (K(x,y,z)-K(c_Q,y,z))1_{(CQ\times CQ)^c}(y,z)\phi(x)dydzdx\right| \\
    &\le \int_{\mathbb{R}} \int_{\mathbb{R}^2\setminus Q_{\varepsilon}(x)} \frac{|x-c_Q|}{(|x-y|+|x-z|)^3}1_{(CQ\times CQ)^c}(y,z)|\phi(x)|dydzdx \\
    &\le \frac{|Q|}{2} \int_{\mathbb{R}} \int_{\mathbb{R}^2\setminus Q_{\varepsilon}(x)} \frac{1}{(|x-y|+|x-z|)^3}1_{(CQ\times CQ)^c}(y,z)|\phi(x)|dydzdx.
\end{align*}
By a coordinate transformation, we now have \begin{multline*}
    \int_{\mathbb{R}} \int_{\mathbb{R}^2\setminus Q_{\varepsilon}(x)} \frac{1}{(|x-y|+|x-z|)^3}1_{(CQ\times CQ)^c}(y,z)|\phi(x)|dydzdx \\
    = \int_{\mathbb{R}} \int_{\mathbb{R}^2\setminus Q_{\varepsilon}(0)} \frac{1}{(|y|+|z|)^3}1_{(x,x)-(CQ\times CQ)^c}(y,z)|\phi(x)|dydzdx.
\end{multline*}
Note that by construction, we have $Q_{\varepsilon}(x)\subset (CQ\times CQ)$ for all $x\in Q$, hence we can always find $R_C>0$ such that\begin{multline*}
    \int_{\mathbb{R}} \int_{\mathbb{R}^2\setminus Q_{\varepsilon}(0)} \frac{1}{(|y|+|z|)^3}1_{(x,x)-(CQ\times CQ)^c}(y,z)|\phi(x)|dydzdx \\ 
    \le \int_{\mathbb{R}} \int_{\mathbb{R}^2\setminus B_{R_C}(0)} \frac{1}{(|y|+|z|)^3}|\phi(x)|dydzdx.
\end{multline*}
In particular, we may choose $R_C$ such that $R_C\to\infty$ as $C\to \infty$. We may now calculate the integrals by transforming to polar coordinates and obtain \begin{align*}
     &\int_{\mathbb{R}} \int_{\mathbb{R}^2\setminus B_{R_C}(0)} \frac{1}{(|y|+|z|)^3}|\phi(x)|dydzdx \\
     &= \|\phi\|_1\int_{\mathbb{R}^2\setminus B_{R_C}(0)} \frac{1}{(|y|+|z|)^3}|dydz \\
     &\le \|\phi\|_1\int_{\mathbb{R}^2\setminus B_{R_C}(0)} \frac{1}{\sqrt(y^2+z^2)^3}|dydz \\
     & = 2\pi \|\phi\|_1\int_{R_C}^{\infty}\frac{1}{r^2}dr \\
     &= \frac{2\pi \|\phi\|_1}{R_C}.
\end{align*}
As~\eqref{eqn=t(1,1)_def} is independent of our choice of $C$, we conclude \begin{equation*}
    \int_{\mathbb{R}} \int_{\mathbb{R}^2\setminus Q_{\varepsilon}(x)} (K(x,y,z)-K(c_Q,y,z))1_{(CQ\times CQ)^c}(y,z)\phi(x)dydzdx = 0.
\end{equation*}
\end{proof}
\end{auxproof}

\subsection{Completely bounded estimates and constants for bilinear multipliers}

The following is a special case of the main theorem of~\cite{DiPlinioMathAnn}, specialised to our setting of Proposition~\ref{Prop=FourierCZ} and Schatten classes. Unfortunately~\cite{DiPlinioMathAnn} does not keep track of the constants, though they can be made explicit by following the proof. We have outlined the proof of~\eqref{multiplier_on_R bound} in  Appendix~\ref{Sect=AppendixConstants}. Note that Remark~\ref{rem: T(1)=0} implies the vanishing of the paraproduct terms in~\cite{DiPlinioMathAnn}, which allows for a significantly better bound of~\eqref{multiplier_on_R bound} compared to general Calder\'on-Zygmund operators, see Remark~\ref{Rmk=ParaVanish}.

\begin{theorem}[{Special case of~\cite[Theorem 1.1]{DiPlinioMathAnn}}]\label{thrm: diplinio_statement_bilinear}
    Let $T$ be a bilinear Calderon-Zygmund operator on~$\mathbb{R}$. Then the bilinear operator $$T_{cb}(\sum_{j=1}^Jf_j\otimes y_j,\sum_{k=1}^Kg_k\otimes z_k):= \sum_{j,k}T(f_j,g_k)\otimes y_jz_k$$ with $f_j,g_k\in L^{\infty}_c(\mathbb{R})$, $y_j\in S_{p_1}$, $z_k\in S_{p_2}$, extends to a bounded operator  
    \[
        T_{cb}:L^{p_1}(\mathbb{R},S_{p_1})\times L^{p_2}(\mathbb{R},S_{p_2})\to L^{p}(\mathbb{R},S_{p})
    \]
    for $p_1,p_2,p\in(1,\infty)$ such that $1/p_1+1/p_2=1/p$. Moreover, if for every $\phi\in L_c^{\infty}(\mathbb{R})$  with $\int_{\mathbb{R}} \phi dx = 0$, we have
    \begin{equation}\label{Eqn=T1}
   \langle  T(1,1), \phi \rangle  = \langle T^{\ast 1}(1,1), \phi \rangle =  \langle T^{\ast 2}(1,1), \phi \rangle  =  0,
    \end{equation}
     then
   \begin{equation}\label{multiplier_on_R bound}
        \Vert T_{cb}:L^{p_1}(\mathbb{R},S_{p_1})\times L^{p_2}(\mathbb{R},S_{p_2})\to L^{p}(\mathbb{R},S_{p}) \Vert \lesssim   C(p, p_1, p_2),
    \end{equation}
    where
    \begin{equation}\label{Eqn=CFunction}
          C({p,p_1,p_2})= \beta_{p}\beta_{p_1}\beta_{p_2}   
    +\min(\beta_{p_1}^2\beta_{p},\beta_{p}^2\beta_{p_1})
 + \min(\beta_{p_2}^2\beta_{p},\beta_{p}^2\beta_{p_2}) +\min(\beta_{p_2}^2\beta_{p_1},\beta_{p_1}^2\beta_{p_2}),
     \end{equation}
    and $\beta_q = q q^\ast$, $1 < q < \infty$.
\end{theorem}

\begin{remark}\label{Rmk=ParaVanish}
Without the condition~\eqref{Eqn=T1} the paraproducts in the representation theorem  described in Section~\ref{Sect=AppendixDefs} do not vanish. Theorem~\ref{thrm: diplinio_statement_bilinear} remains true but with a worse constant $C'(p,p_1, p_2)$ given by 
\begin{align*}
    C'(p,p_1, p_2) &=  C(p,p_1, p_2) + \min(C''(p,p_1),C''(p,p_2)) \\
    & \quad +
    \min(C''(p_1,p_2),C''(p_1,p))+ 
    \min(C''(p_2,p_1),C''(p_2,p)) , \\
     C''(p,q) &= \beta_p^3\beta_q^2 C_{\mathrm{BMO}_{q}},
\end{align*}  
where $C_{\mathrm{BMO}_{p}}=2e(ep\Gamma(p))^{1/p}$ refers to the constant in the John-Nirenberg inequality, see e.g.\ \cite{GrafakosOldNew}. For $p\to\infty$, we have $C_{\mathrm{BMO}_{p}}=O(p)$.
 The constant $C'$ is derived through a combination of the permutation argument that we present at the end of Appendix~\ref{Sect=AppendixConstants}, and explicit calculations found in~\cite{JesseThesis}. The facts we present in this remark shall not be used in this paper.
\end{remark}

Next, we translate this statement to Fourier multipliers. This allows us to use transference to estimate bilinear Schur multipliers such as the ones in Proposition~\ref{Prop=SchurDecomposition} by their corresponding Fourier multipliers.

\begin{theorem}\label{Thm=TmEstimate}
Let $m: \mathbb{R}^2 \setminus \{ 0 \} \rightarrow \mathbb{R}$ be smooth, odd, homogeneous, and set $m(0,0) = 0$.  Then for $1<p,p_1, p_2 < \infty$ with $\frac{1}{p} = \frac{1}{p_1} + \frac{1}{p_2}$ we have
\[
\Vert T_m: L^{p_1}(\mathbb{R}^2, S_{p_1}) \times   L^{p_2}(\mathbb{R}^2, S_{p_2})  \rightarrow L^{p}(\mathbb{R}^2, S_{p})   \Vert \lesssim C(p, p_1, p_2),
\]
where $C(p, p_1, p_2)$ is as in~\eqref{Eqn=CFunction}.
\end{theorem}
\begin{proof}
By Proposition~\ref{Prop=FourierCZ}, $T_m$ is a bilinear Calder\'on-Zygmund operator. 
 By Remark~\ref{rem: T(1)=0}, we see that~\eqref{Eqn=T1} holds. Therefore, the statement follows directly from Theorem~\ref{thrm: diplinio_statement_bilinear}.
\end{proof}

\begin{theorem}\label{Thm=Transference}
Let $m: \mathbb{R}^2 \setminus \{ 0 \} \rightarrow \mathbb{R}$ be smooth, odd, homogeneous, and set $m(0,0) = 0$. Set
\[
\widetilde{m}(\lambda_0, \lambda_1, \lambda_2) = m(\lambda_1 - \lambda_0, \lambda_2 - \lambda_1), \quad  (\lambda_0, \lambda_1, \lambda_2) \in \mathbb{R}^3.
\]
 Then
\[
\Vert  M_{\widetilde{m}}: S_{p_1} \times S_{p_2} \rightarrow S_p \Vert \leq \Vert T_m: L^{p_1}(\mathbb{R}, S_{p_1}) \times   L^{p_2}(\mathbb{R}, S_{p_2})  \rightarrow L^{p}(\mathbb{R}, S_{p})   \Vert \lesssim  C(p, p_1, p_2),
\]
with $C(p, p_1, p_2)$ as given in~\eqref{Eqn=CFunction}.
\end{theorem}
\begin{proof}
We will apply~\cite[Theorem A]{CKV} to a modification of~$m$ that is continuous at zero. Define $m(\lambda_1, \lambda_2; \mu_1, \mu_2) = m(\lambda_1 - \mu_1, \lambda_2 - \mu_2)$, $\lambda_i, \mu_i \in \mathbb{R}$. 
Let $f \in C_b(\mathbb{R})$  with compact support be such that $f \geq 0$ and $\Vert f \Vert_1 = 1$. We set
\[
m_f(\lambda_1, \lambda_2) = \int_{\mathbb{R}} \int_{\mathbb{R}}  f(\mu_1) f(\mu_2)  m(\lambda_1, \lambda_2; \mu_1, \mu_2)  d \mu_1 d\mu_2,
\]
which is continuous.
Set again $\widetilde{m}_f(\lambda_0, \lambda_1, \lambda_2) = m_f(\lambda_1 - \lambda_0, \lambda_2 - \lambda_1)$.
It now follows from~\cite[Theorem A]{CKV} that    
\[
\Vert  M_{\widetilde{m}_f}: S_{p_1} \times S_{p_2} \rightarrow S_p \Vert \leq   \Vert T_{m_f}: L^{p_1}(\mathbb{R}, S_{p_1}) \times   L^{p_2}(\mathbb{R}, S_{p_2})  \rightarrow L^{p}(\mathbb{R}, S_{p})   \Vert. 
\]
Next, observe that~\cite[Lemma 4.3]{CJKM} shows that $T_m$ and $T_{m(\: \cdot \:, \: \cdot \: ; \mu_1, \mu_2 )}$ have the same norm as bilinear maps. Therefore, it follows that  
\[
\begin{split}
& \Vert T_{m_f}: L^{p_1}(\mathbb{R}; S_{p_1}) \times   L^{p_2}(\mathbb{R}; S_{p_2})  \rightarrow L^{p}(\mathbb{R}; S_{p})   \Vert \\
\leq & 
 \int_{\mathbb{R}} \int_{\mathbb{R}}  f(\mu_1) f(\mu_2) 
\Vert T_{m}: L^{p_1}(\mathbb{R}; S_{p_1}) \times   L^{p_2}(\mathbb{R}; S_{p_2})  \rightarrow L^{p}(\mathbb{R}; S_{p})   \Vert d\mu_1 d\mu_2\\
= & \Vert T_{m}: L^{p_1}(\mathbb{R}; S_{p_1}) \times   L^{p_2}(\mathbb{R}; S_{p_2})  \rightarrow L^{p}(\mathbb{R}; S_{p})   \Vert.
\end{split}
\]
 Combining the previous two estimates with  Theorem~\ref{Thm=TmEstimate} yields that 
\begin{equation}\label{Eqn=ImportantEstimate}
\begin{split}
\Vert  M_{\widetilde{m}_f}: S_{p_1} \times S_{p_2} \rightarrow S_p \Vert 
\lesssim   C(p, p_1, p_2).
\end{split}
\end{equation} 
Now replace $f$ by  functions $f_j  \in C_c(\mathbb{R})$ satisfying $f_j \geq 0$,  $\Vert f_j \Vert_1 = 1$, $\supp(f_j)\subset \supp(f_{j-1})$, and~$\bigcap_j \supp(f_j)=\{0\}$.  
Take $x_1 \in S_{p_1} \cap S_2, x_2 \in S_{p_2}\cap S_2$, and $x_3 \in S_{p^\ast} \cap S_2$. Assume that each of these operators is rank one with respective kernels  $A_i(s,t) = \xi_i(s) \eta_i(t)$ (see Section~\ref{Sect=SchurPrelim}), where we assume $ \xi_i, \eta_i \in C_c(\mathbb{R})$, $ i=1,2,3$.  Then as $j \rightarrow \infty$ we get for the Schatten class duality pairing
\begin{equation}\label{Eqn=ConvergeKernel0} 
\begin{split}
  & \langle M_{ \widetilde{m}_{f_j} }(x_1, x_2), x_3 \rangle_{p, p^\ast} \\
= &
\int_{\mathbb{R}^3}  \widetilde{m}_{f_j}(s_0, s_1,  s_2) \xi_1(s_0) \xi_2(s_1) \xi_3(s_2) \eta_1(s_1) \eta_2(s_2) \eta_3(s_0) ds_0 ds_1 ds_2  \\
= &
\int_{\mathbb{R}^3} \int_{\mathbb{R}^2}  f_j(\mu_1) f_j(\mu_2) m(s_1 - s_0 - \mu_1, s_2 - s_1 - \mu_2) \\ 
& \qquad  \qquad \times\xi_1(s_0) \xi_2(s_1) \xi_3(s_2) \eta_1(s_1) \eta_2(s_2) \eta_3(s_0) d\mu_1 d\mu_2 ds_0 ds_1 ds_2  \\
= &
\int_{\mathbb{R}^3} \int_{\mathbb{R}^2}  f_j(\mu_1) f_j(\mu_2) m(s_1 - s_0 , s_2 - s_1 ) \\
& \qquad \qquad\times  \xi_1(s_0 - \mu_1) \xi_2(s_1) \xi_3(s_2 +\mu_2) \eta_1(s_1) \eta_2(s_2 + \mu_2) \eta_3(s_0 - \mu_1) d\mu_1 d\mu_2 ds_0 ds_1 ds_2. 
\end{split}
\end{equation}
  As each $\xi_j$ and $\eta_j$ is continuous and compactly supported, we have 
\begin{equation}\label{Eqn=InsertExtra}
\begin{split}
& \int_{\mathbb{R}^2}  f_j(\mu_1) f_j(\mu_2) \xi_1(s_0 - \mu_1)  \xi_2(s_1) \xi_3(s_2 +\mu_2) \eta_1(s_1) \eta_2(s_2 + \mu_2) \eta_3(s_0 - \mu_1) d\mu_1 d\mu_2 \\
&  \qquad \xrightarrow{j\to\infty}  \;\xi_1(s_0 ) \xi_2(s_1) \xi_3(s_2) \eta_1(s_1) \eta_2(s_2) \eta_3(s_0 ),
\end{split}
\end{equation}
in the~$L^1(\mathbb{R}^3)$-norm, where we see the expressions in~\eqref{Eqn=InsertExtra} as functions of  $(s_0, s_1, s_2) \in \mathbb{R}^3$.
Therefore, taking the limit  $j \rightarrow \infty$ in  \eqref{Eqn=ConvergeKernel0} gives
\begin{equation}\label{Eqn=ConvergeKernel}
\begin{split}
& \langle M_{ \widetilde{m}_{f_j} }(x_1, x_2), x_3 \rangle_{p, p^\ast} \\
&\qquad\xrightarrow{j\to\infty} 
\int_{\mathbb{R}^3}     m(s_1 - s_0 , s_2 - s_1 ) \xi_1(s_0 ) \xi_2(s_1) \xi_3(s_2) \eta_1(s_1) \eta_2(s_2) \eta_3(s_0)  ds_0 ds_1 ds_2  \\
 & \qquad= \langle M_{ \widetilde{m} }(x_1, x_2), x_3 \rangle_{p, p^\ast} \\
\end{split}
 \end{equation}
By linearity, density, and uniform boundedness  of $M_{\widetilde{m}}$ and $M_{\widetilde{m}_{f_j}}$  as maps $S_2 \times S_2 \rightarrow S_2$ (see Section~\ref{Sect=SchurPrelim}) the convergence \eqref{Eqn=ConvergeKernel} holds for any $x_1 \in S_2 \cap S_{p_1}, x_2 \in S_2 \cap S_{p_2}, x_3 \in S_2 \cap S_{p^\ast}$. 
Hence, 
\[
\begin{split}
\Vert  M_{\widetilde{m}}: S_{p_1} \times S_{p_2} \rightarrow S_p \Vert = &
\sup_{\substack{x_1 \in S_2 \cap S_{p_1}, x_2 \in S_2 \cap S_{p_2}, x_3 \in S_2 \cap S_{p^\ast}, \\ \Vert x_1 \Vert_{p_1} = \Vert x_2 \Vert_{p_2} = \Vert x_3 \Vert_{p^\ast} = 1}} \vert \langle M_{ \widetilde{m}  }(x_1, x_2), x_3 \rangle_{p, p^\ast}  \vert \\
= &  \sup_{\substack{x_1 \in S_2 \cap S_{p_1}, x_2 \in S_2 \cap S_{p_2}, x_3 \in S_2 \cap S_{p^\ast}, \\ \Vert x_1 \Vert_{p_1} = \Vert x_2 \Vert_{p_2} = \Vert x_3 \Vert_{p^\ast} = 1}} \lim_{j \rightarrow \infty} \vert \langle M_{ \widetilde{m}_{f_j} }(x_1, x_2), x_3 \rangle_{p, p^\ast}  \vert \\
\leq  & \limsup_{j \rightarrow \infty} 
\Vert  M_{\widetilde{m}_{f_j}}: S_{p_1} \times S_{p_2} \rightarrow S_p \Vert,
\end{split}
\]
which concludes the proof.
 \end{proof}

\section{Proof of Theorem A and extrapolation}\label{sect: proof_thrm_A}

\subsection{Main result}  We now collect all estimates we have obtained so far in this paper.

\begin{theorem}[Theorem A]\label{Thm=TheoremA}
For every $f \in C^2(\mathbb{R})$ and for every $1<p, p_1, p_2 < \infty$ such that $\frac{1}{p} = \frac{1}{p_1} + \frac{1}{p_2}$ we have that
\[
\Vert M_{f^{[2]}}: S_{p_1} \times S_{p_2} \rightarrow S_{p} \Vert \lesssim D(p, p_1, p_2) \Vert f'' \Vert_\infty,
\]
where
\[
\begin{split}
    D(p, p_1, p_2)  =  &  C(p, p_1, p_2) ( \beta_{p_1} + \beta_{p_2} ) + \beta_{p_1} \beta_{p_2}   (\beta_p  + \beta_{p_1}  + \beta_{p_2} )  \\
  \end{split}
\]
where $C(p, p_1,p_2)$ was defined in~\eqref{Eqn=CFunction} and $\beta_q = q q^\ast$. 
\end{theorem}
\begin{proof}
Consider the decomposition of $M_{f^{[2]}}$ given in~\eqref{Eqn=MainDecompositionStart} in terms of bilinear Schur multipliers of Toeplitz form and linear Schur multipliers. It is sufficient to show that each of these maps are bounded on the corresponding Schatten classes. Each of the functions $a_1 := \epsilon_1 \widetilde{\theta}_1 \psi_1$, $a_2:= \epsilon_2 \widetilde{\theta}_1 (1 - \psi_1)$, $a_3 := \epsilon_3 \widetilde{\theta}_2 \psi_2$, $a_4 := \epsilon_1 \widetilde{\theta}_2 (1 - \psi_2)$, $a_5 := \epsilon_2 \widetilde{\theta}_2 \psi_2$ and $a_6 := \epsilon_3 \widetilde{\theta}_3(1- \psi_3)$ is smooth,  odd, homogeneous, and has value zero at zero. Note that we added the $\epsilon_i$ terms to assure that the functions are odd. Therefore by Theorem~\ref{Thm=Transference} we get the bounds
\[
\begin{split}
\Vert   M_{a_i}  : S_{p_1} \times S_{p_2} \rightarrow S_{p} \Vert \lesssim  C(p, p_1, p_2), \quad 1 \leq i \leq 6.
\end{split}
\]
We shall only use this fact for $i=1,2$. 
By Corollary~\ref{Cor=S1Estimate} we also get 
\[
\begin{split}
\Vert   M_{a_i}  : S_{p_1} \times S_{p_2} \rightarrow S_{p} \Vert \lesssim  p_1p_1^\ast p_2 p_2^\ast = \beta_{p_1} \beta_{p_2}, \quad 3 \leq i \leq 6.
\end{split}
\]
For the linear term $M_{\epsilon \phi_f}$
, we apply Remark~\ref{Rmk=LinearCase} to see that for any $1 < q < \infty$,
\[
\begin{split}
\Vert   M_{\epsilon \phi_f}: S_{q}  \rightarrow S_{q} \Vert \lesssim \Vert f'' \Vert_{\infty} q   q^\ast = \beta_q,
\end{split}
\]
and similarly for $\epsilon \mathring{\phi}_f$. These estimates together with the decomposition~\eqref{Eqn=MainDecompositionStart}  allow us to conclude
\[
\begin{split}
   \Vert M_{f^{[2]}}: S_{p_1} \times S_{p_2} \rightarrow S_{p} \Vert    \lesssim   C(p, p_1, p_2) ( \beta_{p_1} + \beta_{p_2} ) + \beta_{p_1} \beta_{p_2}   (\beta_p  + \beta_{p_1}  + \beta_{p_2} ).
\end{split}
\]
\end{proof}
 
\begin{remark}\label{rem: constant_asymptotics}
We examine the constant $D(p, 2p, 2p)$  with   $1 < p < \infty$ and its asymptotics for $p$ going either to $\infty$ or 1. Note that if $p \searrow 1$ then  $(2p)^\ast \nearrow 2$. In fact,  $(2p)^\ast$ is   uniformly bounded for $1 < p < \infty$. We therefore find for  $1 < p < \infty$ that
\[
D(p, 2p, 2p)  \approx   p^4 p^\ast.
\]
\end{remark}

\begin{remark}\label{Remark=OldOrder}
The $p$-dependence of the norm of the  triple operator integral appearing in~\cite[Remark 5.4]{PSS-Inventiones} is not made explicit in~\cite{PSS-Inventiones}. Following the proof of~\cite{PSS-Inventiones} in the bilinear case we find that $D(p, 2p, 2p) = O( p^{12})$ as $p\to\infty$. This is justified as follows.
\begin{enumerate}
\item The three triangular truncations used on~\cite[p. 533]{PSS-Inventiones} yield a factor of order~$O(p^3)$.
\item Estimating the linear terms in decomposition~\cite[Eqn. (4.3)]{PSS-Inventiones} yields a factor of order~$O(p^3)$, arising from the application of~\cite[Lemma 4.5]{PSS-Inventiones}, which is of order $O(p^3)$.
\item Estimating the bilinear terms in decomposition~\cite[Eqn. (4.5)]{PSS-Inventiones} yields a factor of order~$O(p^6)$. As shown on~\cite[p. 519]{PSS-Inventiones}, these estimates require two applications of~\cite[Lemma 4.5]{PSS-Inventiones}, which is of order~$O(p^3)$, to estimate the operator $R_s$ of~\cite{PSS-Inventiones}.
\end{enumerate}
A detailed account of these facts is contained in~\cite{JesseThesis}. Our proof thus gives a significant improvement of estimate for $D(p, 2p, 2p)$ from $O( p^{12})$ to $O(p^4)$ in case $p \rightarrow \infty$. In Section~\ref{Sect=LowerBound} we show that the order of  $D(p, 2p, 2p)$ is at least $O(p^2)$ for $p \rightarrow \infty$.
\end{remark}

 \subsection{Extrapolation} Let  $x \in B(H)$ be a compact operator. We set the decreasing rearrangement of $t \in [0, \infty)$ as
 \[
 \mu_t(x) = \inf \{ \Vert x p \Vert \mid p \in B(H) \textrm{ projection  with } {\rm Tr}(p) \leq t  \}.
 \]
 We define $M_{1, \infty}$ as the Marcinkiewicz space of all compact operators $x$ such that
 \[
 \Vert x  \Vert_{M_{1, \infty}} :=  \sup_{t \in [0, \infty)}  \log(1+t)^{-1} \int_0^t \mu_s(x) ds < \infty.
 \]
Theorem~\ref{Thm=TheoremA} now yields the following extrapolation result, which should be compared to~\cite[Corollary 5.6]{CMPS}. 

\begin{theorem}
For every $f \in C^2(\mathbb{R})$  we have
\[
\Vert M_{f^{[2]}}: S_{2} \times S_{2} \rightarrow M_{1, \infty}  \Vert < \infty.
\]
\end{theorem}
\begin{proof}
Let $s > 0$ be large, set $p = \log(s)$ and set $q = p^*= p (p-1)^{-1}$ to be the H\"older conjugate of $p$. Note that as $s \rightarrow \infty$ we thus have $q \searrow 1$.  Let $x,y \in  S_{2}$   and set $T =  M_{f^{[2]}}(x,y)$.
 Then by H\"older's inequality, Theorem~\ref{Thm=TheoremA}, and the fact that the embedding $S_{2  }\hookrightarrow  S_{2q}$ is contractive, we have
 
\[
 \begin{split}
   \int_0^s \mu_t(T)  dt \leq s^{\frac{1}{p}}  \left(   \int_0^s \mu_t(T)^q dt \right)^{\frac{1}{q}} \leq s^{\frac{1}{p}} \Vert T \Vert_q \lesssim &
    s^{\frac{1}{p}} D(q, 2q, 2q)  \Vert x \Vert_{2q}  \Vert y \Vert_{2q}\\
     \leq &
s^{\frac{1}{p}} D(q, 2q, 2q)  \Vert x \Vert_{2}  \Vert y \Vert_{2}.
\end{split}
\]
We have 
\[
s^{\frac{1}{p}} D(q, 2q, 2q) \leq 100 s^{\frac{1}{p}} q^\ast = 100 e^{\frac{1}{p} \log(s)} p = 
100 e^1 \log(s). 
\]
 So  we see that
\[
\int_0^s \mu_t(T)  dt \lesssim \log(s)  \Vert x \Vert_{2}  \Vert y \Vert_{2}.
\]
This proves the extrapolation result.
\end{proof}

\begin{remark}
The question what the best recipient space for triple operator integrals of second order divided difference functions is remains open. In particular we do not know whether for  $f \in C^2(\mathbb{R})$ we have
\[
\Vert M_{f^{[2]}}: S_{2} \times S_{2} \rightarrow S_{1, \infty} \Vert < \infty,
\]
where $S_{1, \infty}$ is the weak $S_1$-space. Only in case $f(s) = s \vert s \vert$, as well as some simple modifications of this function,  this question is answered in the affirmative~\cite{CSZ-Israel}.  In Section~\ref{Sect=LowerBound} we prove lower bounds for Schur multipliers associated with the latter function.
\end{remark}

\section{Lower bounds and proof of Theorem~\ref{Thm=TheoremB_intro}}\label{Sect=LowerBound} 
In this section we investigate the lower bounds of Schur multipliers of second order divided difference functions. In~\cite{CLPST} it was already shown that for general $f \in C^2(\mathbb{R})$ we do not necessarily have that $M_{f^{[2]}}$ maps $S_2 \times S_2$ to $S_1$. The counterexample of~\cite{CLPST} is given by the function~$f(s) = s \vert s\vert, s \in \mathbb{R}$ (or in fact a perturbation of this function around zero that makes the function $C^2$).   Here we improve on this result by providing explicit lower bounds for the corresponding problem on Schatten classes. Our proof gives in fact better asymptotics for $p \rightarrow \infty$ than~\cite{CLPST}, as we explain in   Remark~\ref{Rmk=Comparison}.

\begin{theorem}[Theorem B, Part 1]\label{Thm=LowerBound1}
Let  $f(s) = s \vert s\vert, s \in \mathbb{R}$. Then for every  $1 < p < \infty$ we have
\[
\Vert M_{f^{[2]}}: S_{2p} \times S_{2p} \rightarrow S_{p}  \Vert \gtrsim  p^2.
\]
\end{theorem}

We prove Theorem \ref{Thm=LowerBound1} through a couple of lemmas. 

\begin{lemma}\label{Lem=Limit}
Let  $f(s) = s \vert s\vert, s \in \mathbb{R}$. Let $q \in (0,1)$ and let $i,j,l \in \mathbb{N}$ be such that $i \not = j$ and $j \not = l$. Then 
\begin{equation}\label{Eqn=Limit}
\lim_{k \rightarrow \infty} f^{[2]}(q^{ki}, -q^{kj},  q^{kl})  =
\left\{
\begin{array}{ll}
-1 & \textrm{if }  j< i,  \textrm{ and }  j < l,   \\ 
1  & \textrm{otherwise.}
\end{array}
\right.
\end{equation}
\end{lemma}
\begin{proof}
Let $\lambda_0, \lambda_2 > 0$ and $\lambda_1 < 0$, then $f(\lambda_0) = \lambda_0^2, f(\lambda_1) = - \lambda_1^2, f(\lambda_2) = \lambda_2^2$. First expand
\begin{equation}\label{Eqn=f2expand}
\begin{split}
f^{[2]}(\lambda_0, \lambda_1, \lambda_2) = & \frac{ f^{[1]}(\lambda_0, \lambda_1)  - f^{[1]}(\lambda_1, \lambda_2)  }{ \lambda_0 - \lambda_2}\\
 = & \frac{1}{\lambda_0 - \lambda_2} \left(   \frac{f(\lambda_0) - f(\lambda_1)}{ \lambda_0 - \lambda_1}  - \frac{f(\lambda_1) - f(\lambda_2)}{ \lambda_1  - \lambda_2} \right) \\
 = &  \frac{1}{\lambda_0 - \lambda_2} \left(   \frac{ \lambda_0^2 + \lambda_1^2}{ \lambda_0 - \lambda_1}  - \frac{ - \lambda_1^2 - \lambda_2^2}{ \lambda_1  - \lambda_2} \right).
\end{split}
\end{equation}
We set $\lambdaw_1 := -\lambda_1$. Then $\lambda_0, \lambdaw_1, \lambda_2 >0$ and
\begin{equation}\label{Eqn=DividedDifference}
\begin{split}
f^{[2]}(\lambda_0, -\lambdaw_1, \lambda_2) = & \frac{(\lambda_0^2 + \lambdaw_1^2)(\lambdaw_1 + \lambda_2) - (\lambdaw_1^2 + \lambda_2^2)(\lambda_0 + \lambdaw_1)  }{(\lambda_0 - \lambda_2)(\lambda_0 + \lambdaw_1) (\lambdaw_1 + \lambda_2) } \\
= &  \frac{\lambda_0^2 \lambdaw_1 + \lambda_0^2 \lambda_2 + \lambdaw_1^2 \lambda_2 - \lambdaw_1^2 \lambda_0 - \lambda_2^2 \lambda_0   - \lambda_2^ 2 \lambdaw_1 }{(\lambda_0 - \lambda_2)(\lambda_0 + \lambdaw_1) (\lambdaw_1 + \lambda_2) } \\ 
= &  \frac{ (\lambda_0 + \lambda_2)\lambdaw_1 + \lambda_0 \lambda_2 - \lambdaw_1^2  }{ (\lambda_0 + \lambdaw_1) (\lambdaw_1 + \lambda_2) }. \\ 
\end{split}
\end{equation}
Let $q \in (0,1)$ as in the statement of the lemma and let $k \in \mathbb{N}$.  Set 
$\lambda_0 = q^{ki}, \lambdaw_1  = q^{kj}, \lambda_2 = q^{kl}$, 
where~$i,j,l \in \mathbb{N}$ are  natural numbers with $i \not = j$ and  $j \not = l$. By considering each of the 6 possible orderings 
 of~$\lambda_0, \lambdaw_1$, and $\lambda_2$, we see from~\eqref{Eqn=DividedDifference} that
\begin{equation}\label{Eqn=Limit}
\lim_{k \rightarrow \infty} f^{[2]}(q^{ki}, -q^{kj},  q^{kl})  =
\left\{
\begin{array}{ll}
-1 & \textrm{if }  j< i \textrm{ and } j < l.  \\
1  & \textrm{otherwise.}
\end{array}
\right.
\end{equation}
This concludes the proof.    
\end{proof}

\begin{lemma}\label{Lem=Discretization}
Let  $f(s) = s \vert s\vert, s \in \mathbb{R}$. Let $q \in (0,1)$ and for $i,j,l \in \mathbb{N}$, let ${\phi_k(i,j,l) := (q^{ki}, -q^{kj}, q^{kl})}$.  Then for all $1 < p < \infty$ we have
\begin{equation*}
     \Vert M_{f^{[2]} \circ \phi_k}: S_{2p}(  \ell^2(\mathbb{N}) ) \times S_{2p}(\ell^2(\mathbb{N}) ) \rightarrow   S_{p}(\ell^2(\mathbb{N})) \Vert  \leq  \Vert M_{f^{[2]}}: S_{2p}  \times S_{2p}  \rightarrow  S_{p}  \Vert.
\end{equation*}
\end{lemma}
\begin{proof}
   Let $F, G \subseteq \mathbb{R}$ be finite sets not containing 0. Then $F \cup G $ is contained in a set $X_\delta \subseteq \mathbb{R}$ of the form $(- \infty, -\delta) \cup (\delta, \infty)$ for some $\delta > 0$. Note that $f^{[2]}$ is continuous on $X_\delta \times X_\delta  \times X_\delta $ and so we may apply \cite[Theorem 2.2]{CKV}.  
By using respectively a restriction of the domain of a bilinear Schur multiplier, then applying \cite[Theorem 2.2]{CKV} and then again a restriction of the domain, we get 
\[
\begin{split}
& \Vert M_{f^{[2]}}: S_{2p}(\ell^2(F), \ell^2(G)) \times S_{2p}(\ell^2(G), \ell^2(F)) \rightarrow   S_{p}(\ell^2(F)) \Vert \\
&\qquad \leq   \Vert M_{f^{[2]}}: S_{2p}(\ell^2(F \cup G)) \times S_{2p}(\ell^2(F \cup G)) \rightarrow   S_{p}(\ell^2(F \cup G)) \Vert \\
&\qquad \leq  \Vert M_{f^{[2]}}: S_{2p}(L^2(X_\delta)) \times S_{2p}(L^2(X_\delta)) \rightarrow  S_{p}(L^2(X_\delta)) \Vert \\
&\qquad  \leq \Vert M_{f^{[2]}}: S_{2p}(L^2(\mathbb{R})) \times S_{2p}(L^2(\mathbb{R})) \rightarrow  S_{p}(L^2(\mathbb{R})) \Vert. 
\end{split}
\]
Now let $F, G \subseteq \mathbb{R}$ be any subsets not containing 0. The  union of all vector spaces $S_{p}(\ell^2(F_0), \ell^2(G_0)) $ with $F_0 \subseteq F, G_0\subseteq G$ finite is dense in $S_{p}(\ell^2(F), \ell^2(G))$. Therefore we have, 
\[
\begin{split}
&  \Vert M_{f^{[2]}}: S_{2p}(\ell^2(F), \ell^2(G)) \times S_{2p}(\ell^2(G), \ell^2(F)) \rightarrow   S_{p}(\ell^2(F)) \Vert \\
& \qquad = \sup_{F_0 \subseteq F, G_0 \subseteq G  \textrm{ finite}} \Vert  M_{f^{[2]}}: S_{2p}(\ell^2(F_0), \ell^2(G_0)) \times S_{2p}(\ell^2(G_0), \ell^2(F_0)) \rightarrow   S_{p}(\ell^2(F_0)) \Vert \\
 & \qquad \leq  \Vert M_{f^{[2]}}: S_{2p}(L^2(\mathbb{R})) \times S_{2p}(L^2(\mathbb{R})) \rightarrow  S_{p}(L^2(\mathbb{R})) \Vert.
\end{split}
\]
Now let 
\[ 
F_k =  \{ q^{k i } \mid i \in \mathbb{N} \}, \quad 
G_k =   - F_k = \{ - q^{k i } \mid i \in \mathbb{N} \}.
\]
 Let $\delta_x$ be as in Section~\ref{Sect=SchurPrelim}  and define unitary maps 
\[
U_k: \ell^2(\mathbb{N}) \rightarrow \ell^2(F_k): \delta_n \mapsto \delta_{q^{kn}}, \quad  V_k: \ell^2( \mathbb{N} ) \rightarrow \ell^2( G_k): \delta_n \mapsto  \delta_{-q^{kn}}.
\] 
 By interpreting these maps as base change operators, one can relate $M_{f^{[2]}\circ\phi_k}$ and $M_{f^{[2]}}$ via 
\[
M_{f^{[2]} \circ \phi_k}( x, y ) = U_k^\ast M_{f^{[2]}} ( U_k x V_k^\ast, V_k y U_k ) U_k^\ast, \quad x,y \in S_2(\ell^2(\mathbb{N})).
\]
Therefore
\[
\begin{split}
& \Vert M_{f^{[2]} \circ \phi_k}: S_{2p}(\ell^2(\mathbb{N}) ) \times S_{2p}(\ell^2(\mathbb{N}) ) \rightarrow   S_{p}(\ell^2(\mathbb{N})) \Vert \\
& \qquad  =  \Vert M_{f^{[2]} }: S_{2p}(\ell^2(F_k), \ell^2(G_k)) \times S_{2p}(\ell^2(G_k), \ell^2( F_k  ) ) \rightarrow   S_{p}(F_k) \Vert \\
& \qquad  \leq \Vert M_{f^{[2]}}: S_{2p}(L^2(\mathbb{R})) \times S_{2p}(L^2(\mathbb{R})) \rightarrow  S_{p}(L^2(\mathbb{R})) \Vert.
\end{split}
\]
This concludes the proof. 
\end{proof}

\begin{proof}[Proof of Theorem \ref{Thm=LowerBound1}]     
Let $T^\pm = T_{\widetilde{h_{\pm}}}: S_{2p}(\ell^2(\mathbb{N})) \rightarrow S_{2p}(\ell^2(\mathbb{N}))$ be the triangular truncation given by the Schur multiplier with symbol,
\[
\widetilde{h_\pm}(\lambda, \mu) = h_\pm(\lambda - \mu), \quad  
h_\pm(\lambda) = \left\{
\begin{array}{ll}
1 & \textrm{ if } \pm \lambda < 0, \\
0     & \textrm{ if } \pm \lambda \geq 0,
\end{array}
\right.
\] 
There exist constants $C, D >0$ such that for all $1<p<\infty$,
\begin{equation}\label{Eqn=TriangleEstimate0}
 C  p    < \Vert T^{\pm}: S_{2p}( \ell^2(\mathbb{N}) )  \rightarrow S_{2p}( \ell^2(\mathbb{N}) )  \Vert < D p.
\end{equation}
The lower bound of this inequality, which is  well-known and most relevant to us,  follows for instance from the explicit sequence of singular values of the Volterra operator due to Krein (see~\cite[Theorem IV.8.2 and IV.7.4]{GohbergKrein}).    Now set $M^+ = T^+ - T^-$ and $M^- = T^- - T^+$.  Let $P$ be the projection of $S_{2p}(\ell^2(\mathbb{N}))$ onto the diagonal elements. Then $P$ is a contraction (see~\cite[Lemma 2.1]{CPSZ-JOT}). {Note that $M^\pm$ is a Schur multiplier acting on $S_{2p}( \ell^2(\mathbb{N}) )$ with symbol $\widetilde{H}_{\pm}(\lambda, \mu) = H_\pm(\lambda - \mu)$, $\lambda, \mu \in \mathbb{N}$, and $H_{\pm}(\lambda) = \pm 1$ if $\pm \lambda < 0$. Similarly, $P$ is a Schur multiplier with symbol $p(\lambda, \mu) = 1$ if $\lambda = \mu$ and $p(\lambda, \mu) = 0$ otherwise, where again $\lambda, \mu \in \mathbb{N}$. In particular,  $T^+ = \frac{1}{2} (M^+ + {\rm Id} - P)$. }Therefore, from  \eqref{Eqn=TriangleEstimate0} we get by the reverse triangle inequality
\[
  2C p - 2 < \Vert (M^+ + {\rm Id} - P)   \Vert - \Vert  {\rm Id} - P \Vert \leq   
  \Vert  M^+ \Vert,
\]
where all norms are  operator norms of linear maps on $ S_{2p}( \ell^2(\mathbb{N}) )$. {For $1 <  p < \infty$ we have by~\eqref{Eqn=SchurElement}, which also hold for discrete symbols, 
\[
1 = \Vert \widetilde{H}_{+} \Vert_{\ell^\infty(\mathbb{N} \times \mathbb{N})}  = \Vert M^+:   S_{2}( \ell^2(\mathbb{N}) ) \rightarrow  S_{2}( \ell^2(\mathbb{N}) )\Vert\leq \Vert M^+:   S_{2p}( \ell^2(\mathbb{N}) ) \rightarrow  S_{2p}( \ell^2(\mathbb{N}) )\Vert.
\]}
We thus obtain $\max( 2Cp - 2, 1) \leq \Vert M^+: S_{2p}( \ell^2(\mathbb{N}) ) \rightarrow  S_{2p}( \ell^2(\mathbb{N}) ) \Vert$. Furthermore, for any~${1 < p < \infty}$ we have $\frac{2}{3} C p < \max( 2Cp - 2, 1) $. Altogether, we see that for all  
  $1 < p < \infty$ we have 
\begin{equation}\label{Eqn=TriangleEstimate}
\frac{2}{3} C  p    < \Vert M^{+}: S_{2p}( \ell^2(\mathbb{N}) )  \rightarrow S_{2p}( \ell^2(\mathbb{N}) ) \Vert.
\end{equation}
Now fix $1 < p < \infty$. By~\eqref{Eqn=TriangleEstimate}  we can for any $\epsilon > 0$ choose $x \in S_{2p}(\ell^2(\mathbb{N})) $ such that
\begin{equation}\label{Eqn=AsymptoticsTrunc}
\Vert M^+(x)  \Vert_{2p} > \frac{2}{3} C    p   (\Vert x \Vert_{2p} - \epsilon).
\end{equation}
It follows that
\[
\begin{split}
 \Vert M^-( x^\ast) M^+(  x )\Vert_p =
\Vert M^+(x)^\ast M^+(x)\Vert_p 
 = 
 \Vert   M^+(x)\Vert_{2p}^2 > \frac{4}{9}  C^2 p^2  (\Vert x \Vert_{2p} - \epsilon)^2.
\end{split}
\]
Now for $i,j,l \in \mathbb{N}$ as in Lemma \ref{Lem=Discretization} we define $\phi_k(i,j,l) := (q^{ki}, -q^{kj}, q^{kl})$.
Then the limit in Lemma \ref{Lem=Limit}    shows that for $x,y \in S_{2p}(\ell^2(F))$ with $F \subseteq \mathbb{N}$ finite    we have   
\begin{equation}\label{Eqn=LimitDiscreteMultipliers}
 \begin{split}
 \lim_k   M_{f^{[2]} \circ \phi_k}((1-P)(y),(1-P)(x))   = &
 \lim_k  \!\!\!\!\!\!\!\! \sum_{\substack{\lambda_0, \lambda_1, \lambda_2 \in F,\\ \lambda_0 \not = \lambda_1, \lambda_1 \not = \lambda_2 }} \!\!\!\!
  (f^{[2]} \circ \phi_k)(\lambda_0, \lambda_1, \lambda_2)  p_{\lambda_0} x p_{\lambda_1} y p_{\lambda_2} \\
= &    M^-((1-P)(y)) M^+((1-P)(x)) \\ 
= &  M^-(y) M^+(x). 
\end{split}
\end{equation} 
 As we are deal with finite dimensional spaces, the limit \eqref{Eqn=LimitDiscreteMultipliers} holds in the norm of $S_{p}(\ell^2(F))$. Moreover, as~$M_{f^{[2]} \circ \phi_k}$ is bounded uniformly in $k$ by Lemma \ref{Lem=Discretization}, it follows by density of the span of $\{ S_{2p}(\ell^2(F)) \mid F \subseteq \mathbb{N} \textrm{ finite} \}$ in $S_{2p}(\ell^2(\mathbb{N}))$ that this convergence holds for any~$x,y \in S_{2p}(\ell^2(\mathbb{N}))$.

We now have the estimates 
\[
\begin{split}
\frac{4}{9}  C^2 p ^2 (\Vert x \Vert_{2p} - \epsilon)^2 < & \Vert M^-( x^\ast) M^+(  x ) \Vert_p \\
 \leq  &
   \limsup_k \Vert M_{f^{[2]} \circ \phi_k}((1-P)(x)^\ast,  (1-P)(x) ) \Vert_{p}.
   \end{split}
   \]
  Thus by Lemma \ref{Lem=Discretization} we get, 
   \[
   \begin{split}
 \frac{4}{9} C^2 p ^2 (\Vert x \Vert_{2p} - \epsilon)^2 < & \Vert M_{f^{[2]}}: S_{2p} \times S_{2p} \rightarrow S_p \Vert \Vert(1-P)(x) \Vert_{2p}^2 \\
    \leq &  4 \Vert M_{f^{[2]}}: S_{2p} \times S_{2p} \rightarrow S_p \Vert \Vert x \Vert_{2p}^2. 
    \end{split}
\]
Hence 
\[
\Vert M_{f^{[2]}}: S_{2p} \times S_{2p} \rightarrow S_{p}  \Vert \gtrsim  p^2.
\]
\end{proof}
 
 Note that if $p \searrow 1$, then $2p \searrow 2$ and hence the norm in~\eqref{Eqn=TriangleEstimate} remains bounded. Therefore, we need a different proof to treat  $p \searrow 1$, which we present below as a separate theorem. Since many parts of the proof are similar to the proof of Theorem \ref{Thm=LowerBound1} we present it in a more concise manner.  
 
\begin{theorem}[Theorem B, Part 2]\label{Thm=LowerBound2}
Let  $f(s) = s \vert s\vert, s \in \mathbb{R}$. Then for every  $1 < p < \infty$ we have
\[
\Vert M_{f^{[2]}}: S_{2p} \times S_{2p} \rightarrow S_{p}  \Vert \gtrsim  p^\ast. 
\]
\end{theorem}
\begin{proof}
   Assume that $\lambda_0 >0$,  $\lambda_1 \geq  0$,  $\lambda_2 < 0$  so that $f(\lambda_0) = \lambda_0^2$, $f(\lambda_1) = \lambda_1^2$, $f(\lambda_2) = -\lambda_2^2$. In the proof we will take $\lambda_1$ to be   very close to zero and infinitesimally smaller  than both $\lambda_0$ and  $\vert \lambda_2 \vert$.  
As in~\eqref{Eqn=f2expand}, we expand 
\begin{equation}\label{Eqn=Secondexpand}
\begin{split}
f^{[2]}(\lambda_0, \lambda_1, \lambda_2) = &    \frac{1}{\lambda_0 - \lambda_2} \left(   \frac{ \lambda_0^2 - \lambda_1^2 }{ \lambda_0 - \lambda_1}  - \frac{ \lambda_1^2 +  \lambda_2^2}{ \lambda_1  - \lambda_2} \right).
\end{split}
\end{equation}
Take some $q \in (0,1)$ fixed and let $k \in \mathbb{N}$.  Assume that $\lambda_0 = \lambda_0(k) = q^{ki}, \lambda_1 = q^{k(i+l)}, \lambda_2 = \lambda_2(k) = - q^{kl}$ for $i, l \in \mathbb{N}$ different natural numbers. By our definition zero is not included in $\mathbb{N}$, and therefore $\lambda_1$ is strictly smaller than both $\lambda_0$ and $\vert \lambda_2\vert$.

Again we see from~\eqref{Eqn=Secondexpand} that  
\begin{equation}\label{Eqn=LimitSecond}
\lim_{k \rightarrow \infty} f^{[2]}(q^{ki}, q^{k(i+l)},  -q^{kl})  =
\left\{
\begin{array}{ll}
1 & \textrm{if }  i< l,  \\
-1  & \textrm{if }  l < i. 
\end{array}
\right.
\end{equation}
Now for $i, j, l \in \mathbb{N}$, let
\[
\phi_k(i,j,l) = (q^{ki}, q^{k(i+l)}, -q^{kl}).
\]
  Let the diagonal projection $P$ and the  Schur multiplier $M^+$ be defined as in the proof of Theorem~\ref{Thm=LowerBound1}. 
    Then from the limit~\eqref{Eqn=LimitSecond}   and the fact that    as in the proof of Theorem~\ref{Thm=LowerBound1} we can show that $ M_{f^{[2]} \circ \phi_k}$ is bounded uniformly in $k$, we can  show that
\[
 M^+ (y  x)  =  \lim_k (1-P)( M_{f^{[2]} \circ \phi_k}( y ,x) ), \quad y, x \in S_{2p}(\mathbb{N}),
\]
  with convergence in the norm of $S_{p}(\mathbb{N})$. 
 
We recall from~\eqref{Eqn=TriangleEstimate} and by duality,  that there exist~$C, D >0$ such that for every $1 < p < \infty$ we have  
\begin{equation}\label{Eqn=TriangleEstimateSecond}
C  p p^\ast   < \Vert  M^+: S_{p}(\mathbb{N} )  \rightarrow S_{p}(\mathbb{N} )  \Vert < D p p^\ast.
\end{equation}
For any $\epsilon > 0$ and $1< p < \infty$  we can choose $z \in S_{p}(\mathbb{N} ) $ such that
\[
\Vert   M^+(z)  \Vert_{p} >  C p   (\Vert z \Vert_{p} - \epsilon).
\]
Write $z = y x$ with $y,x \in S_{2p}(\mathbb{N} ) $ such that $\Vert z \Vert_p = \Vert y \Vert_{2p} \Vert x \Vert_{2p}$. 
We now have the estimates
\[
\begin{split}
 C p^\ast (\Vert z \Vert_{p} - \epsilon) < & \Vert M^+(  y , x)   \Vert_p \\
 \leq  &
   \limsup_k \Vert (1-P)( M_{f^{[2]} \circ \phi_k}( x,  y )  ) \Vert_{p} \\
    \leq  &
   \limsup_k \Vert  M_{f^{[2]} \circ \phi_k}( x,  y )  \Vert_{p}.
   \end{split}
   \]
Then   by~\cite[Theorem 2.2]{CKV}, 
   \[
   \begin{split} 
 C p^\ast (\Vert z \Vert_{p} - \epsilon) < & \Vert M_{f^{[2]}}: S_{2p} \times S_{2p} \rightarrow S_p \Vert \Vert x  \Vert_{2p}\Vert y  \Vert_{2p} \ \\
    \leq &  \Vert M_{f^{[2]}}: S_{2p} \times S_{2p} \rightarrow S_p \Vert \Vert z  \Vert_{p}.  
    \end{split}
\]
Hence we have obtained 
\[
\Vert M_{f^{[2]}}: S_{2p} \times S_{2p} \rightarrow S_{p}  \Vert \gtrsim  p^\ast.
\]

\end{proof}

\begin{remark}\label{Rmk=Comparison}
We argue that our  result of Theorem~\ref{Thm=LowerBound1} is fundamentally better than the methods employed in~\cite{CLPST}. 
In principle, the method of proof in~\cite{CLPST} can be adjusted to  yield that $\Vert M_{f^{[2]}}: S_{2p} \times S_{2p} \rightarrow S_{p}  \Vert \gtrsim  pp^\ast$ for the same function $f$ as in Theorems~\ref{Thm=LowerBound1} and~\ref{Thm=LowerBound2}. Indeed, the idea of~\cite{CLPST} is to first prove the reduction inequality
\[
 \Vert M_{f^{[2]}}: S_{2p} \times S_{2p} \rightarrow S_{p}  \Vert  \geq \sup_{\lambda_1 \in \mathbb{R}} \Vert M_{f^{[2]}(\: \cdot\:, \lambda_1, \: \cdot\:)}: S_{p}  \rightarrow S_{p}  \Vert.
\]
The right hand side has order $O(p p^\ast)$, which can be seen from Theorem~\ref{thrm: CGPT} for instance. So the reduction of~\cite{CLPST} is not efficient enough to capture the optimal constants.
\end{remark}

\noindent {\it Conflict of interest and data availability statement.} On behalf of all authors, the corresponding author states that there is no conflict of interest. Data sharing is not applicable to this article as no datasets were generated or analysed during the current study.

\appendix
\section{Proof of Theorem~\ref{thrm: diplinio_statement_bilinear} following~\cite{DiPlinioMathAnn}}\label{Sect=AppendixConstants}
\subsection{Dyadic definitions and notations}\label{Sect=AppendixDefs}
We first give a brief overview over the dyadic notions used in the proof of Theorem~\ref{thrm: diplinio_statement_bilinear}. Unless noted otherwise, all definitions are from~\cite[Section 2.2]{DiPlinioMathAnn}. While the concepts introduced in this section are well-defined on $\mathbb{R}^d$, we restrict our discussion to $d=1$, as it simplifies the notation and is in fact the only relevant case to our special case of Theorem~\ref{thrm: diplinio_statement_bilinear}.

See e.g.\ \cite[Section 2.2]{DiPlinioMathAnn} for the definitions on~$\mathbb{R}^d$ for $d>1$.

\vspace{0.3cm}

\noindent {\bf Dyadic grids.}   
The \emph{standard dyadic grid} on $\mathbb{R}$ is defined as $$\mathcal{D}_0\;:=\; \{2^{-k}([0,1)+m)\mid k,m\in\mathbb{Z}\}.$$
Let~$\Omega=\{0,1\}^{\mathbb{Z}}$ and equip $\Omega$ with a probability measure such that its coordinates are independent and uniformly distributed on $\{0,1\}$. The \emph{random dyadic grid} on $\mathbb{R}$ associated with~${\omega=(\omega_k)_{k\in\mathbb{Z}}\in\Omega}$ is defined as\begin{align*}
        \mathcal{D}_{\omega}&\;:=\; \{Q+\omega\mid Q\in\mathcal{D}_0\},\\
        Q+\omega&\;:=\; Q+\sum_{\substack{k\in\mathbb{Z}\\2^{-k}<|Q|}}2^{-k}\omega_k,
    \end{align*} where $|Q|$ denotes the length of the cube $Q$  in $\mathbb{R}$. By a \emph{dyadic grid} $\mathcal{D}$ we refer to $\mathcal{D}=\mathcal{D}_{\omega}$ for some $\omega\in\Omega$. For $Q\in \mathcal{D}$, $\mathcal{D}$ dyadic grid, define $Q^{(k)}$ as the cube $R\in\mathcal{D}$ such that $Q\subset R$ and $2^k|Q|=|R|$. Further set $\mathrm{ch}_{\mathcal{D}}(Q):= \{Q'\in\mathcal{D}\mid Q'\subsetneq Q \text{ and there exists no } Q''\in\mathcal{D}\text{ such that } Q' \subsetneq Q'' \subsetneq Q\}.$ We refer to this set as the \emph{children} of $Q$ in $\mathcal{D}$. The index denoting the dyadic grid may be omitted.

\vspace{0.3cm}

\noindent {\bf Haar functions.}
Let $\mathcal{D}$ be a dyadic grid on $\mathbb{R}$ and let $Q\in\mathcal{D}$. Let $Q_{\mathrm{left}}$ (resp.\ $Q_{\mathrm{right}}$) denote the left (resp.\ right) half of $Q$. For $\eta\in\{0,1\}$, we define the \emph{Haar function} \begin{equation*}
    h^{\eta}_Q\;:=\;\begin{cases}
        |Q|^{-1/2}1_{Q}, & \eta=0,\\
        |Q|^{-1/2}(1_{Q_{\mathrm{left}}}-1_{Q_{\mathrm{right}}}), & \eta=1.
    \end{cases}
\end{equation*} To simplify the notation, we set $h_Q\;:=\;h_Q^{1}$. Note that $\int_{\mathbb{R}}h_Q(x)dx=0$, hence we refer to $h_Q$ as a \emph{cancellative Haar function}. Furthermore, note that for any $Q_0\in\mathcal{D}$, an orthonormal basis of~$L^2(Q_0)$ is given by the family $\{h^0_{Q_0}\}\cup\{h_Q\mid Q\subseteq Q_0\text{ dyadic cube}\}$. 

From the Haar functions we construct the \emph{dyadic martingale difference} of a locally integrable function $f$ as $(D_Q f)_Q$, where
\begin{equation*}
        D_Qf= \langle f\rangle_{Q_{\mathrm{left}}}1_{Q_{\mathrm{left}}} + \langle f\rangle_{Q_{\mathrm{right}}}1_{Q_{\mathrm{right}}} - \langle f\rangle_{Q}1_{Q} = \langle f, h_Q\rangle h_Q,
    \end{equation*}  $\langle f \rangle_Q$ denotes the average of $f$ over a region $Q$, and $\langle f, g\rangle:= \int_{\mathbb{R}} f(x)g(x)dx$. Further define $$\Delta_Q^lf:=\sum_{\substack{R\in\mathcal{D}\\R^{(l)}=Q}}D_Rf=\sum_{\substack{R\in\mathcal{D}\\R^{(l)}=Q}}\sum_{R'\in\mathrm{ch}_{\mathcal{D}}(R)}(\langle f\rangle_{R'}-\langle f\rangle_{R})1_{R'}.$$
    
\noindent {\bf Shifts, paraproducts, and representation of Calder\'on-Zygmund Operators.}
The proof of Theorem~\ref{thrm: diplinio_statement_bilinear} heavily relies on a dyadic representation theorem for Calder\'on-Zygmund operators, see~\cite{bilinrepthrm}. For the convenience of the reader, we repeat the relevant definitions here; see~\cite{DiPlinioMathAnn} for the general $n$-linear case. 

Let $X$ be a Banach space and $\mathcal{D}$ a dyadic grid on $\mathbb{R}$. A \emph{bilinear dyadic shift} $S_{\mathcal{D}}^k$ of complexity $k=(k_1,k_2,k_3)\in\mathbb{N}^{3}_0$ is defined on ${f,g\in L^{\infty}_c(\mathbb{R},X)}$ as\begin{align}
        &S_{\mathcal{D}}^{k}(f,g):=\sum_{Q\in\mathcal{D}}A_Q^{k}(f,g),\label{eqn=def_shift_1}\\
        &A_Q^k(f,g):=\sum_{\substack{I_1,I_2,I_{3}\subseteq Q \\ |I_j|=2^{-k_j}|Q|}} \alpha_{I_1,I_2,I_{3},Q}\langle f,\Tilde{h}_{I_1}\rangle\langle g,\Tilde{h}_{I_2}\rangle \Tilde{h}_{I_{3}},\label{eqn=def_shift_2}
    \end{align}
    where exactly one of $\Tilde{h}_{I_{1}},\Tilde{h}_{I_{2}},\Tilde{h}_{I_{3}}$ is a non-cancellative Haar function and the other two are cancellative Haar functions. The index corresponding to the cancellative Haar function is denoted by~$j_0$. Furthermore,  the coefficients $\alpha_{I_1,I_2,I_3,Q} \in \mathbb{C}$ must satisfy \begin{equation}\label{eqn=shiftscalars}
        |\alpha_{I_1,I_2,I_3,Q}|\le \frac{1}{|Q|^{2}} \prod_{j=1}^{3}|I_j|^{1/2}.
    \end{equation}

    A \emph{bilinear paraproduct} is defined on $f,g\in L^{\infty}_c(\mathbb{R})$ as \begin{align*}
        \pi_{\mathcal{D}}(f,g):= \sum_{Q\in\mathcal{D}}a_Q\langle f,\Tilde{h}_{1,Q}\rangle\langle g,\Tilde{h}_{2,Q}\rangle\Tilde{h}_{3,Q},
    \end{align*}
    where $(\Tilde{h}_{1,Q},\Tilde{h}_{2,Q},\Tilde{h}_{3,Q})$ are such that there is exactly one $j_0\in\{1,2,3\}$ such that for all $Q\in\mathcal{D}$ we have  $\Tilde{h}_{j_0,Q}=h_{Q}$ and $\Tilde{h}_{j,Q}=1_{Q}/|Q|$ for all $j\ne j_0$. The scalar sequence  $(a_Q)_{Q\in\mathcal{D}}$ is such that \begin{align*}
            \sup_{Q_0\in\mathcal{D}}\left(\frac{1}{|Q_0|}\sum_{\substack{Q\in\mathcal{D}\\ Q\subset Q_0}}|a_Q|^2\right)^{1/2}\le 1.
        \end{align*}
 Note that we will usually suppress the dyadic grid from the notation and refer to dyadic shifts and paraproducts as $S^k$ and $\pi$, respectively.

Let $T$ be a bilinear Calder\'on-Zygmund operator and $f,g,h\in L^{\infty}_c(\mathbb{R})$. Then \begin{equation}\label{eqn: bilinear CZO decomp}
    \langle T(f,g),h\rangle = C_T \mathbb{E}_{\omega}\sum_{k\in\mathbb{N}_0^{3}}\sum_u2^{-\max_i k_i/2}\langle U_{\mathcal{D}_{\omega},u}^k(f,g),h\rangle,
\end{equation}where $C_T$ is a constant depending only on $T$, the sum over $u$ is finite, 
and $\mathcal{D}_{\omega}$ is a random dyadic grid. For $\max_j k_j>0$, $U_{\mathcal{D}_{\omega},u}^k$ denotes a bilinear dyadic shift of complexity $k$, whereas for~${\max_j k_j=0}$, $U_{\mathcal{D}_{\omega},u}^k$ denotes either a bilinear dyadic shift of complexity $0$ or a bilinear paraproduct.
Note that by Equation~(4.4) in~\cite{bilinrepthrm}, the paraproducts in this representation are constructed from a scalar sequence \begin{equation*}
    a_Q=C_{T}\langle T(1,1),h_Q\rangle,
\end{equation*}
hence $T(1,1)=0$ implies that the paraproducts in the representation of $T$ vanish. This applies in particular to the situation of Remark~\ref{rem: T(1)=0},   where we have\begin{equation}\label{eqn: bilinear CZO decomp T1}
     \langle T(f,g),h\rangle = C_T \mathbb{E}_{\omega}\sum_{k\in\mathbb{N}_0^{3}}\sum_u2^{-\max_i k_i/2}\langle S_{\mathcal{D}_{\omega},u}^k(f,g),h\rangle.
\end{equation}
\subsection{Relevant inequalities}
Before presenting the proof of Theorem~\ref{thrm: diplinio_statement_bilinear}, we first list the estimates that will be used, alongside the constants they introduce. 

Following~\cite{DiPlinioMathAnn}, it is sufficient to consider the following special case of the decoupling estimate~\cite[Theorem 6]{hanninen_operator-valued_2016}.

\begin{theorem}[Decoupling Inequality~{\cite[Theorem 6]{hanninen_operator-valued_2016}}]\label{thrm: decoupling_ineq}
    Let $p\in(1,\infty)$, let $X$ be a UMD space with UMD constant $\beta_{p,X}$, and let $\mathcal{D}$ be a dyadic grid. Further define the following:\begin{itemize}
        \item $\mathcal{D}_{j,k}:=\{Q\in\mathcal{D}\mid |Q|=2^{m(k+1)+j}\text{ for some }m\in\mathbb{Z}\}$ for $j,k\in\mathbb{Z}$ fixed,
    \item the probability space $\mathcal{V}_Q:=(Q,\mathrm{Leb}(Q),\lambda_Q)$, where $\mathrm{Leb}(Q)$ denotes the Lebesgue measurable subsets of $Q$ and $\lambda_Q$ the normalised restriction of the Lebesgue measure to $Q$,
    \item the product probability space $\mathcal{V}:=\prod_{Q\in\mathcal{D}}\mathcal{V}_Q$ with measure $\nu$ and elements $y=(y_Q)_{Q\in\mathcal{D}}$.
    \end{itemize} Let $(\varepsilon_Q)_{Q\in\mathcal{D}}$ be a Rademacher sequence. Let $(f_Q)_{Q\in\mathcal{D}}$ be a sequence of functions $\mathbb{R} \rightarrow X$ such that for all $Q\in\mathcal{D}$, $f_Q$ is 1) supported on $Q$, 2) constant on every $Q'\in\mathrm{ch}_{\mathcal{D}}(Q)$, and 3) $\langle f_Q\rangle_Q = 0$ holds. Then \begin{multline}\label{eqn: decoupling_ineq}
        \frac{1}{\beta_{p,X}^p}\mathbb{E}\int_{\mathbb{R}} \int_{\mathcal{V}}\|\sum_{Q\in\mathcal{D}_{j,k}}\varepsilon_Q1_Q(x) f_Q(y_Q)\|_{X}^pd\nu(y)dx \\
        \le \int_{\mathbb{R}} \|\sum_{Q\in\mathcal{D}_{j,k}}f_Q(x)\|_{X}^pdx \le \beta_{p,X}^p \mathbb{E}\int_{\mathbb{R}} \int_{\mathcal{V}}\|\sum_{Q\in\mathcal{D}_{j,k}}\varepsilon_Q1_Q(x) f_Q(y_Q)\|_{X}^pd\nu(y)dx.
    \end{multline}
    This inequality also holds when replacing $\mathcal{D}_{j,k}$ with $\mathcal{D}$.
\end{theorem}

\begin{theorem}[Kahane-Khintchine inequality, {\cite[Theorem 3.2.23]{aibs_vol_1}}]\label{thrm: kahane_khintchine_ineq}
    Let $(\varepsilon_n)_n$ be a Rademacher sequence on a probability space $\Omega$, and let $X$ be a Banach space. For~$p,q\in(0,\infty)$ there exists $\kappa_{p,q}<\infty$ such that for all $N\in\mathbb{N}$ and $x_1,\dotsc,x_N\in X$ we have $$\|\sum_{n=1}^N\varepsilon_nx_n\|_{L^p(\Omega,X)}\le \kappa_{p,q}\|\sum_{n=1}^N\varepsilon_nx_n\|_{L^q(\Omega,X)}.$$
\end{theorem}
\begin{remark}\label{rem: kahane_khintchine_constant}
    Relevant in this section is the case $p=2$, $q>1$. Following the proof of Theorem~\ref{thrm: kahane_khintchine_ineq} in~\cite{aibs_vol_1}, the constant $\kappa_{p,q}$ is the same as in~\cite[Theorem 3.2.17]{aibs_vol_1} for $p,q\ge 1$, namely $\kappa_{p,q}=2^{1+1/q}e\left(1+2\frac{p}{q}\right)$.
    In particular, we thus have $\kappa_{2,q}\le 12(1+4/q)\le 60$ for all $q\ge 1$.
\end{remark}

The following theorem has been specialised to our dyadic setting.   Stein's inequality is originally due to Bourgain, and for the explicit constant we refer to the proof in \cite{FigielW} which is also contained in the monograph \cite{aibs_vol_1}.  
 
\begin{theorem}[Stein's inequality, Eqn.~(2.3) of \cite{DiPlinioMathAnn},  Theorem~4.2.23 of \cite{aibs_vol_1}, or Lemma~34 of \cite{FigielW}]\label{thrm: stein_ineq}
    Let $X$ be a UMD space with UMD constant~$\beta_{p,X}$  and let~$\mathcal{D}$ be a dyadic grid. Let~$(f_Q)_{ Q\in\mathcal{D}}$ be a sequence in $L^1_{\mathrm{loc}}(X)$ such that $\mathrm{supp}\; f_Q\subseteq Q$,  $Q\in\mathcal{D}$, and such that only finitely many of them are nonzero, and let $p\in(1,\infty)$. Then $$\mathbb{E}\|\sum_{Q\in\mathcal{D}}\varepsilon_Q\langle f_Q\rangle_Q1_Q\|_{L^p(\mathbb{R},X)}\le \beta_{p,X} \mathbb{E}\|\sum_{Q\in\mathcal{D}}\varepsilon_Qf_Q\|_{L^p(\mathbb{R},X)}.$$
\end{theorem}

\begin{theorem}[Kahane contraction principle, {\cite[Proposition 3.2.10]{aibs_vol_1}}]\label{thrm: kahane_contraction_principle}
    Let $(\varepsilon_n)_n$ be a Rademacher sequence on a probability space $\Omega$, $(a_n)_n$ a finite scalar sequence, and $(x_n)_n$ a finite sequence in a Banach space~$X$. Let $1\le p\le\infty$. Then $$\|\sum_{n=1}^Na_n\varepsilon_nx_n\|_{L^p(\Omega;X)}\le\max_{n}|a_n|\|\sum_{n=1}^N\varepsilon_nx_n\|_{L^p(\Omega;X)}.$$ 
\end{theorem}

\subsection{Proof of Theorem~\ref{thrm: diplinio_statement_bilinear}}

We repeat the proof of Theorem~3.17 in~\cite{DiPlinioMathAnn}, specialised to the bilinear case for $d=1$ and $T(1,1)=0$ (in the sense of Remark~\ref{rem: T(1)=0}). By the representation theorem introduced in Section~\ref{Sect=AppendixDefs}, the proof reduces to the following theorem from~\cite[Section 4]{DiPlinioMathAnn}. 
\begin{theorem}\label{thrm: bilin_shift_bdd} Let $p,p_1,p_2\in(1,\infty)$ such that $1/p_1+1/p_2=1/p$. Set $p_3:=p^*$.
    Let $S^k$ be a  bilinear dyadic shift of complexity  $k=(k_1,k_2,k_3)\in\mathbb{N}^{3}_0$ and let $f_j\in L^{\infty}_c(\mathbb{R},S_{p_j})$, $j=1,2,3$. Define the associated trilinear form \begin{equation*}
        \Lambda_{S^k}(f_1,f_2,f_3)=\sum_{Q\in\mathcal{D}}\sum_{\substack{I_1,I_2,I_{3}\subseteq Q \\ |I_j|=2^{-k_j}|Q|}} \alpha_{I_1,I_2,I_{3},Q}\tau\left(\langle f_1,\Tilde{h}_{I_1}\rangle\langle f_2,\Tilde{h}_{I_2}\rangle \langle f_3, \Tilde{h}_{I_{3}}\rangle\right),
    \end{equation*} where $\tau$ denotes the   trace. It then holds that \begin{equation}\label{Eqn=TriForm}
        |\Lambda_{S^k}(f_1,f_2,f_3)| \lesssim C(p,p_1,p_2)\prod_{j=1}^3\|f_j\|_{L^{p_j}(\mathbb{R},S_{p_j})}.
    \end{equation} 
\end{theorem}
\begin{proof}
    The trilinear form is first rewritten as
    \begin{align}
     \Lambda_{S^k}(f_1,f_2,f_3)&=  \sum_{i=0}^{\kappa} \Lambda_{S_i^k}(f_1,f_2,f_3), \\
     \Lambda_{S_i^k}(f_1,f_2,f_3) &= \sum_{K\in\mathcal{D}_{i,\kappa}}\sum_{\substack{L_1,L_2,L_3\in\mathcal{D}\\L_j^{(l_j)}=K}}b_{L_1,L_2,L_3,K}\tau\left(\prod_{j=1}^3\langle f_j,h'_{L_j}\rangle\right),\label{eqn: new_shift} \\
        b_{L_1,L_2,L_3,K}&=  \sum_{\substack{Q_1,Q_2,Q_3\in\mathcal{D}\\Q_j^{(k_j-l_j)}=L_j}}a_{Q_1,Q_2,Q_3,K}\prod_{j=1}^3\frac{|Q_j|^{1/2}}{|L_j|^{1/2}},
    \end{align}
 where $0\le l_j\le k_j$ and $\kappa=\max k_j$. This is a new shift operator with $h'_{L_j}\in\{h^0_{L_j},h_{L_j}\}$ such that there may be more than two indices $j$ such that their associated Haar functions are cancellative, whereas in~\eqref{eqn=def_shift_2}, the Haar functions are cancellative for exactly two indices. Furthermore, the construction is such that if $h'_{L_j}$ is not cancellative, then $l_j=0$. For details on how to construct this new shift, see~\cite{DiPlinioMathAnn}. 
 
The proof now proceeds as follows. First, boundedness is shown in the case where all Haar functions $h'_{L_j}$ are cancellative. In the second case, where not all Haar functions are cancellative, the fact $h'_{L_j}=h^0_{L_j} \Rightarrow  l_j=0$ allows us to reduce the trilinear form (\ref{eqn: new_shift}) to a bilinear form with only cancellative Haar functions. For this new bilinear form, boundedness follows by the same proof method as in the first case.

\vspace{0.3cm}

\noindent {\bf Case 1.} Let $0\le i\le\kappa$ be such that all associated Haar functions in (\ref{eqn: new_shift}) are cancellative. Note that for $L_3^{(l_3)}\in\mathcal{D}_{i,\kappa}$, orthogonality of the Haar functions yields $$\sum_{K\in\mathcal{D}_{i,\kappa}}\Delta_K^{l_3}h_{L_3}  = \sum_{K\in\mathcal{D}_{i,\kappa}} \sum_{\substack{L\in\mathcal{D}\\L^{(l_3)}=K}}  D_L h_{L_3} =  \sum_{K\in\mathcal{D}_{i,\kappa}} \sum_{\substack{L\in\mathcal{D}\\L^{(l_3)}=K}} \langle h_{L_3}, h_L\rangle h_L =h_{L_3}.$$ Using the decoupling inequality from Theorem~\ref{thrm: decoupling_ineq}, we thus have
    \begin{multline*}
        \|S^k_i(f_1,f_2)\|_{L^p(\mathbb{R},S_p)} = \|\sum_{K\in\mathcal{D}_{i,\kappa}}\sum_{\substack{L_1,L_2,L_3\in\mathcal{D}\\L_j^{(l_j)}=K}}b_{L_1,L_2,L_3,K}\langle f_1,h_{L_1}\rangle\langle f_2,h_{L_2}\rangle h_{L_3}\|_{L^p(\mathbb{R},S_p)} \\ \le\beta_{p,S_p}(\mathbb{E}\int_{\mathbb{R}}\int_{\mathcal{V}}\|\sum_{K\in\mathcal{D}_{i,\kappa}}\varepsilon_K1_K(x)\sum_{\substack{L_1,L_2,L_3\in\mathcal{D}\\L_j^{(l_j)}=K}}b_{L_1,L_2,L_3,K}\prod_{j=1}^2\langle f_j,h_{L_j}\rangle h_{L_3}(y_K)\|^p_{S_p}d\nu(y)dx)^{1/p}.
    \end{multline*} We can rewrite the inner sum in the integral by using $\langle f_j,h_{L_j}\rangle=\langle \Delta^{l_j}_Kf_j,h_{L_j}\rangle$. Indeed, \begin{align*}
         \langle \Delta^{l_j}_Kf_j,h_{L_j}\rangle = \sum_{\substack{L\in\mathcal{D}\\L^{(l_j)}=K}} \langle D_Lf_j, h_{L_j}\rangle =  \sum_{\substack{L\in\mathcal{D}\\L^{(l_j)}=K}} \langle f_j, h_{L}\rangle \langle h_L, h_{L_j}\rangle = \langle f_j, h_{L_j}\rangle.
    \end{align*}

   Hence we can write
   
   \begin{align*}
& \sum_{\substack{L_1,L_2,L_3\in\mathcal{D}\\L_j^{(l_j)}=K}}b_{L_1,L_2,L_3,K}\prod_{j=1}^2\langle f_j,h_{L_j}\rangle h_{L_3}(y_K) \\
&=\sum_{\substack{L_1,L_2,L_3\in\mathcal{D}\\L_j^{(l_j)}=K}}b_{L_1,L_2,L_3,K}\prod_{j=1}^2\langle \Delta^{l_j}_Kf_j,h_{L_j}\rangle h_{L_3}(y_K) \\
        &\quad= \int_{K^2}\sum_{\substack{L_1,L_2,L_3\in\mathcal{D}\\L_j^{(l_j)}=K}}b_{L_1,L_2,L_3,K}\prod_{j=1}^2\Delta^{l_j}_Kf_j(z_j)h_{L_j}(z_j) h_{L_3}(y_K)dz.
    \end{align*}
    By setting  \begin{equation*}
        b_K(y_K,z)=|K|^2\sum_{\substack{L_1,L_2,L_3\in\mathcal{D}\\L_j^{(l_j)}=K}}b_{L_1,L_2,L_3,K}\prod_{j=1}^2h_{L_j}(z_{j}) h_{L_3}(y_K),
    \end{equation*}
    we have \begin{multline*}
        \int_{K^2}\sum_{\substack{L_1,L_2,L_3\in\mathcal{D}\\L_j^{(l_j)}=K}}b_{L_1,L_2,L_3,K}\prod_{j=1}^2\Delta^{l_j}_Kf_j(z_j)h_{L_j}(z_j) h_{L_3}(y_K)dz
        \\
        \quad=\frac{1}{|K|^2}\int_{K^2}b_K(y_K,z)\prod_{j=1}^2\Delta^{l_j}_Kf_j(z_j)dz=\int_{\mathcal{V}^2}b_K(y_K,z_K)\prod_{j=1}^2\Delta^{l_j}_Kf_j(z_{j,K})d\nu(z),
    \end{multline*}
    where $\mathcal{V}$ and $\nu$ are as defined in Theorem~\ref{thrm: decoupling_ineq}. We can use the triangle inequality and as  $\mathcal{V}^2$ is a probability space then apply Jensen's inequality to show \begin{align*}
         &\quad\|S^k_i(f_1,f_2)\|_{L^p(\mathbb{R},S_p)} \\ 
         &\le\beta_{p,S_p}\left(\mathbb{E}\int_{\mathbb{R}}\int_{\mathcal{V}}\|\int_{\mathcal{V}^2}\sum_{K\in\mathcal{D}_{i,\kappa}}\varepsilon_K1_K(x)b_K(y_K,z_K)\prod_{j=1}^2\Delta^{l_j}_Kf_j(z_{j,K})d\nu(z)\|^p_{S_p}d\nu(y)dx\right)^{1/p} \\
         &\le\beta_{p,S_p}\left(\mathbb{E}\int_{\mathbb{R}}\int_{\mathcal{V}}\left(\int_{\mathcal{V}^2}\|\sum_{K\in\mathcal{D}_{i,\kappa}}\varepsilon_K1_K(x)b_K(y_K,z_K)\prod_{j=1}^2\Delta^{l_j}_Kf_j(z_{j,K})\|_{S_p}d\nu(z)\right)^pd\nu(y)dx\right)^{1/p} \\
         &\le\beta_{p,S_p}\left(\mathbb{E}\int_{\mathbb{R}}\int_{\mathcal{V}}\int_{\mathcal{V}^2}\|\sum_{K\in\mathcal{D}_{i,\kappa}}\varepsilon_K1_K(x)b_K(y_K,z_K)\prod_{j=1}^2\Delta^{l_j}_Kf_j(z_{j,K})\|^p_{S_p}d\nu(z)d\nu(y)dx\right)^{1/p}.
    \end{align*} Note that by construction, $|b_K(y_K,z_K)|\le 1$. Indeed, by unfolding definitions and applying estimate~\eqref{eqn=shiftscalars} we have  \begin{align*}
        |b_K(y_K,z_K)|&\le|K|^2\sum_{\substack{L_1,L_2,L_3\in\mathcal{D}\\L_j^{(l_j)}=K}}|b_{L_1,L_2,L_3,K}|\prod_{j=1}^2|h_{L_j}(z_{j,K})| |h_{L_3}(y_K)| \\
        &\le |K|^2\sum_{\substack{L_1,L_2,L_3\in\mathcal{D}\\L_j^{(l_j)}=K}}\sum_{\substack{Q_1,Q_2,Q_3\in\mathcal{D}\\Q_j^{(k_j-l_j)}=L_j}}|a_{Q_1,Q_2,Q_3,K}|\prod_{j=1}^3\frac{|Q_j|^{1/2}}{|L_j|^{1/2}}\frac{1_{L_1}(z_{1,K})}{|L_1|^{1/2}}\frac{1_{L_2}(z_{2,K})}{|L_2|^{1/2}} \frac{1_{L_3}(y_{K})}{|L_3|^{1/2}}  \\
        &\le \sum_{\substack{L_1,L_2,L_3\in\mathcal{D}\\L_j^{(l_j)}=K}}\sum_{\substack{Q_1,Q_2,Q_3\in\mathcal{D}\\Q_j^{(k_j-l_j)}=L_j}}\prod_{l=1}^3|Q_l|^{1/2}\prod_{j=1}^3\frac{|Q_j|^{1/2}}{|L_j|}1_{L_1}(z_{1,K})1_{L_2}(z_{2,K})1_{L_3}(y_K).
    \end{align*} Since the size of $|Q_j|$ relative to $|L_j|$ is fixed, we can rewrite this expression as
    \begin{multline*}
    \sum_{\substack{L_1,L_2,L_3\in\mathcal{D}\\L_j^{(l_j)}=K}}\sum_{\substack{Q_1,Q_2,Q_3\in\mathcal{D}\\Q_j^{(k_j-l_j)}=L_j}}\prod_{l=1}^3|Q_l|^{1/2}\prod_{j=1}^3\frac{|Q_j|^{1/2}}{|L_j|}1_{L_1}(z_{1,K})1_{L_2}(z_{2,K})1_{L_3}(y_K).\\
        =\sum_{\substack{L_1,L_2,L_3\in\mathcal{D}\\L_j^{(l_j)}=K}}\sum_{\substack{Q_1,Q_2,Q_3\in\mathcal{D}\\Q_j^{(k_j-l_j)}=L_j}}\prod_{j=1}^32^{l_j-k_j}1_{L_1}(z_{1,K})1_{L_2}(z_{2,K})1_{L_3}(y_K).
    \end{multline*}
    Finally, we use $$\sum_{\substack{Q_1,Q_2,Q_3\in\mathcal{D}\\Q_j^{(k_j-l_j)}=L_j}} 1= \prod_{j=1}^32^{k_j-l_j}$$ and the disjointness of the children of $K$ to conclude
    \begin{align*}
        &\sum_{\substack{L_1,L_2,L_3\in\mathcal{D}\\L_j^{(l_j)}=K}}\sum_{\substack{Q_1,Q_2,Q_3\in\mathcal{D}\\Q_j^{(k_j-l_j)}=L_j}}\prod_{j=1}^32^{l_j-k_j}1_{L_1}(z_{1,K})1_{L_2}(z_{2,K})1_{L_3}(y_K)
    \\
        &=\sum_{\substack{L_1,L_2,L_3\in\mathcal{D}\\L_j^{(l_j)}=K}}\prod_{j=1}^32^{k_j-l_j}2^{l_j-k_j}1_{L_1}(z_{1,K})1_{L_2}(z_{2,K})1_{L_3}(y_K)\\
        &=\sum_{\substack{L_1,L_2,L_3\in\mathcal{D}\\L_j^{(l_j)}=K}}1_{L_1}(z_{1,K})1_{L_2}(z_{2,K})1_{L_3}(y_K)\\
        &=1_{K}(z_{1,K})1_{K}(z_{2,K})1_{K}(y_K)\\
        &\le 1.
    \end{align*} 
    Letting $\|(x_k)_{k=1}^K\|_{\mathrm{Rad}(S_{p})}:=\left(\mathbb{E}\|\sum_{k=1}^K\varepsilon_kx_k\|_{S_{p}}^2\right)^{1/2}$, we can now finish the proof of this case as follows. From the previous estimates and~\cite[Lemma 4.1]{DiPlinioMathAnn}, it follows that
    \begin{align*}
        &\|S^k_i(f_1,f_2)\|_{L^p(\mathbb{R},S_p)}    \\
       &{\displaystyle\le \beta_{p,S_p}\left(\mathbb{E}\int_{\mathbb{R}}\int_{\mathcal{V}}\int_{\mathcal{V}^2}\|\sum_{K\in\mathcal{D}_{i,\kappa}}\varepsilon_K1_K(x)b_K(y_K,z_K)\prod_{j=1}^2\Delta^{l_j}_Kf_j(z_{j,K})\|^p_{S_p}d\nu(z)d\nu(y)dx\right)^{1/p}}   \\
        &{\displaystyle\le \beta_{p,S_p}\left(\int_{\mathbb{R}}\int_{\mathcal{V}}\int_{\mathcal{V}^2}\prod_{j=1}^2\|(1_K(x)\Delta^{l_j}_Kf_j(z_{j,K}))_{K\in\mathcal{D}_{i,\kappa}}\|^p_{\mathrm{Rad}(S_{p_j})}d\nu(z)d\nu(y)dx\right)^{1/p}}.
    \end{align*}
    Using that $\mathcal{V}$ is a probability space and applying H\"older's inequality yields
    \begin{align*}
    &\beta_{p,S_p}\left(\int_{\mathbb{R}}\int_{\mathcal{V}}\int_{\mathcal{V}^2}\prod_{j=1}^2\|(1_K(x)\Delta^{l_j}_Kf_j(z_{j,K}))_{K\in\mathcal{D}_{i,\kappa}}\|^p_{\mathrm{Rad}(S_{p_j})}d\nu(z)d\nu(y)dx\right)^{1/p}
        \\
        &{\displaystyle= \beta_{p,S_p}\left(\int_{\mathbb{R}}\int_{\mathcal{V}^2}\prod_{j=1}^2\|(1_K(x)\Delta^{l_j}_Kf_j(z_{j,K}))_{K\in\mathcal{D}_{i,\kappa}}\|^p_{\mathrm{Rad}(S_{p_j})}d\nu(z)dx\right)^{1/p}} \\
        &{\displaystyle\le \beta_{p,S_p}\prod_{j=1}^2\left(\int_{\mathbb{R}}\int_{\mathcal{V}^2}\|(1_K(x)\Delta^{l_j}_Kf_j(z_{j,K}))_{K\in\mathcal{D}_{i,\kappa}}\|^{p_j}_{\mathrm{Rad}(S_{p_j})}d\nu(z)dx\right)^{1/p_j}}.
    \end{align*}
    By unfolding the definition of $\|\cdot\|_{\mathrm{Rad}}$, we can apply the Kahane-Khintchine equality to obtain
    \begin{align*}
    &\beta_{p,S_p}\prod_{j=1}^2\left(\int_{\mathbb{R}}\int_{\mathcal{V}^2}\|(1_K(x)\Delta^{l_j}_Kf_j(z_{j,K}))_{K\in\mathcal{D}_{i,\kappa}}\|^{p_j}_{\mathrm{Rad}(S_{p_j})}d\nu(z)dx\right)^{1/p_j}
                \\
        &{\displaystyle= \beta_{p,S_p}\prod_{j=1}^2\left(\int_{\mathbb{R}}\int_{\mathcal{V}^2}(\mathbb{E}\|\sum_{K\in\mathcal{D}_{i,\kappa}}\varepsilon_K1_K(x)\Delta^{l_j}_Kf_j(z_{j,K})\|_{S_{p_j}}^2)^{p_j/2}d\nu(z)dx\right)^{1/p_j}} \\
        &{\displaystyle\le \beta_{p,S_p}\prod_{j=1}^2\kappa_{2,p_j}\left(\int_{\mathbb{R}}\int_{\mathcal{V}^2}\mathbb{E}\|\sum_{K\in\mathcal{D}_{i,\kappa}}\varepsilon_K1_K(x)\Delta^{l_j}_Kf_j(z_{j,K})\|_{S_{p_j}}^{p_j}d\nu(z)dx\right)^{1/p_j}}.
    \end{align*}
    Finally, Fubini's theorem and the decoupling estimate yield
    \begin{align*}
        &\beta_{p,S_p}\prod_{j=1}^2\kappa_{2,p_j}\left(\int_{\mathbb{R}}\int_{\mathcal{V}^2}\mathbb{E}\|\sum_{K\in\mathcal{D}_{i,\kappa}}\varepsilon_K1_K(x)\Delta^{l_j}_Kf_j(z_{j,K})\|_{S_{p_j}}^{p_j}d\nu(z)dx\right)^{1/p_j}
    \\
        &{\displaystyle=\beta_{p,S_p}\prod_{j=1}^2\kappa_{2,p_j}\left(\mathbb{E}\int_{\mathbb{R}}\int_{\mathcal{V}^2}\|\sum_{K\in\mathcal{D}_{i,\kappa}}\varepsilon_K1_K(x)\Delta^{l_j}_Kf_j(z_{j,K})\|_{S_{p_j}}^{p_j}d\nu(z)dx\right)^{1/p_j}} \\
        &{\displaystyle\le \beta_{p,S_p}\prod_{j=1}^2\kappa_{2,p_j}\beta_{p_j,S_{p_j}}\|f_j\|_{L^{p_j}(\mathbb{R},S_{p_j})}},
    \end{align*}
    concluding the proof of Case 1. Altogether, this case yields the estimate \begin{equation*}
        \|S^k_i(f_1,f_2)\|_{L^p(\mathbb{R},S_p)} \lesssim  \beta_{p,S_p}\prod_{j=1}^2\kappa_{2,p_j}\beta_{p_j,S_{p_j}}.
    \end{equation*}
\vspace{0.3cm}

\noindent {\bf Case 2.} Let $0\le i\le\kappa$ be such that one Haar function in (\ref{eqn: new_shift}) is not cancellative. We assume that $h'_{L_2}=h^0_{L_2}$ and $h'_{L_j}=h_{L_j}$, $j=1,3$; the estimates for the other cases follow in the same manner. Note that (\ref{eqn: new_shift}) has been constructed such that this implies $l_2=0$, hence $L_2=K$; see~\cite{DiPlinioMathAnn} for details. We use the decoupling estimate (Theorem~\ref{thrm: decoupling_ineq}) to estimate\begin{align*}
    \|S^k_i(f_1,f_2)\|_{L^p(\mathbb{R},S_p)}&= \|\sum_{K\in\mathcal{D}_{i,\kappa}}\sum_{\substack{L_1,L_3\in\mathcal{D}\\L_j^{(l_j)}=K}}b_{L_1,L_3,K}\langle f_1,h_{L_1}\rangle|K|^{1/2}\langle f_2\rangle_Kh_{L_3}\|_{L^p(\mathbb{R},S_p)} \\
    &\le\beta_{p,S_p}\left(\mathbb{E}\int_{\mathbb{R}}\int_{\mathcal{V}}\|\sum_{K\in\mathcal{D}_{i,\kappa}}\varepsilon_K1_K(x)\langle\varphi_{K,y}\rangle_K\|_{S_p}^pd\nu(y)dx\right)^{1/p},
\end{align*}
where the function $\varphi_{K,y}:\mathbb{R}\to S_p$ is defined as\begin{equation*}
    \varphi_{K,y}(x)\;:=\;|K|^{1/2}\sum_{\substack{L_1,L_3\in\mathcal{D}\\L_j^{(l_j)}=K}}b_{L_1,L_3,K}\langle f_1,h_{L_1}\rangle f_2(x)h_{L_3}(y_K).
\end{equation*}
We can now apply Stein's inequality (Theorem~\ref{thrm: stein_ineq}) with respect to $x\in\mathbb{R}$ to obtain\begin{multline*}
   \beta_{p,S_p}\left(\mathbb{E}\int_{\mathbb{R}}\int_{\mathcal{V}}\|\sum_{K\in\mathcal{D}_{i,\kappa}}\varepsilon_K1_K(x)\langle\varphi_{K,y}\rangle_K\|_{S_p}^pd\nu(y)dx\right)^{1/p} \\
    \le \beta^2_{p,S_p}\left(\mathbb{E}\int_{\mathbb{R}}\int_{\mathcal{V}}\|\sum_{K\in\mathcal{D}_{i,\kappa}}\varepsilon_K1_K(x)\varphi_{K,y}(x)\|_{S_p}^pd\nu(y)dx\right)^{1/p}.
\end{multline*}
By H\"older's inequality we can further estimate\begin{multline*}
    \left(\mathbb{E}\int_{\mathbb{R}}\int_{\mathcal{V}}\|\sum_{K\in\mathcal{D}_{i,\kappa}}\varepsilon_K1_K(x)\varphi_{K,y}(x)\|_{S_p}^pd\nu(y)dx\right)^{1/p}\\ 
     \le 
     (\mathbb{E}\int_{\mathbb{R}}\int_{\mathcal{V}}\|\sum_{K\in\mathcal{D}_{i,\kappa}}\varepsilon_K1_K(x)|K|^{1/2}\sum_{\substack{L_1,L_3\in\mathcal{D}\\L_j^{(l_j)}=K}}b_{L_1,L_3,K}\langle f_1,h_{L_1}\rangle h_{L_3}(y_K)\|_{S_{p_1}}^p\|f_2(x)\|_{S_{p_2}}^pd\nu(y)dx)^{1/p}
     \\
     \shoveleft{\le 
     (\mathbb{E}\int_{\mathbb{R}}\int_{\mathcal{V}}\|\sum_{K\in\mathcal{D}_{i,\kappa}}\varepsilon_K1_K(x)|K|^{1/2}\sum_{\substack{L_1,L_3\in\mathcal{D}\\L_j^{(l_j)}=K}}b_{L_1,L_3,K}\langle f_1,h_{L_1}\rangle h_{L_3}(y_K)\|_{S_{p_1}}^{p_1} d\nu(y)dx)^{1/p_1}}\\
     \times\|f_2\|_{L^{p_2}(\mathbb{R},S_{p_2})}.
\end{multline*}
We now proceed as in Case 1 to estimate the remaining term. We use \begin{align*}
    |K|^{1/2}\sum_{\substack{L_1,L_3\in\mathcal{D}\\L_j^{(l_j)}=K}}b_{L_1,L_3,K}\langle f_1,h_{L_1}\rangle h_{L_3}(y_K) &= \int_{\mathcal{V}}b_K(y_k,z_K)\Delta_K^{l_1}f_1(z_K)d\nu(z),\end{align*} where we define \begin{align*}
    b_K(y_k,z_K) &= |K|^{3/2}\sum_{\substack{L_1,L_3\in\mathcal{D}\\L_j^{(l_j)}=K}}b_{L_1,L_3,K}h_{L_1}(z)h_{L_3}(y_K),
\end{align*}
\noindent
and estimate the remaining integral as \begin{align*}
    &\Biggl(\mathbb{E}\int_{\mathbb{R}}\int_{\mathcal{V}}\|\sum_{K\in\mathcal{D}_{i,\kappa}}\varepsilon_K1_K(x)|K|^{1/2}\sum_{\substack{L_1,L_3\in\mathcal{D}\\L_j^{(l_j)}=K}}b_{L_1,L_3,K}\langle f_1,h_{L_1}\rangle h_{L_3}(y_K)\|_{S_{p_1}}^{p_1} d\nu(y)dx\Biggr)^{1/p_1}\\
    &\quad=\left(\mathbb{E}\int_{\mathbb{R}}\int_{\mathcal{V}}\|\sum_{K\in\mathcal{D}_{i,\kappa}}\varepsilon_K1_K(x)\int_{\mathcal{V}}b_K(y_k,z_K)\Delta_K^{l_1}f_1(z_K)d\nu(z)\|_{S_{p_1}}^{p_1} d\nu(y)dx\right)^{1/p_1} \\
    &\quad\le\left(\mathbb{E}\int_{\mathbb{R}}\int_{\mathcal{V}}\int_{\mathcal{V}}\|\sum_{K\in\mathcal{D}_{i,\kappa}}\varepsilon_K1_K(x)b_K(y_k,z_K)\Delta_K^{l_1}f_1(z_K)\|_{S_{p_1}}^{p_1}d\nu(z) d\nu(y)dx\right)^{1/p_1} 
\end{align*}
Using Fubini's theorem and the Kahane contraction principle (Theorem~\ref{thrm: kahane_contraction_principle}) we further have
\begin{align*}
    &\left(\mathbb{E}\int_{\mathbb{R}}\int_{\mathcal{V}}\int_{\mathcal{V}}\|\sum_{K\in\mathcal{D}_{i,\kappa}}\varepsilon_K1_K(x)b_K(y_k,z_K)\Delta_K^{l_1}f_1(z_K)\|_{S_{p_1}}^{p_1}d\nu(z) d\nu(y)dx\right)^{1/p_1} \\
    &\quad=\left(\int_{\mathbb{R}}\int_{\mathcal{V}}\int_{\mathcal{V}}\mathbb{E}\|\sum_{K\in\mathcal{D}_{i,\kappa}}\varepsilon_K1_K(x)b_K(y_k,z_K)\Delta_K^{l_1}f_1(z_K)\|_{S_{p_1}}^{p_1}d\nu(z) d\nu(y)dx\right)^{1/p_1} \\
    &\quad\le \left(\int_{\mathbb{R}}\int_{\mathcal{V}}\int_{\mathcal{V}}\max_{K\in\mathcal{D}_{i,\kappa}}|b_K(y_k,z_K)|\mathbb{E}\|\sum_{K\in\mathcal{D}_{i,\kappa}}\varepsilon_K1_K(x)\Delta_K^{l_1}f_1(z_K)\|_{S_{p_1}}^{p_1}d\nu(z) d\nu(y)dx\right)^{1/p_1}.
\end{align*}
As in Case 1, we have the pointwise estimate $|b_K(y_k,z_K)|\le 1$, since\begin{align*}
    |b_K(y_k,z_K)| &\le |K|^{3/2}\sum_{\substack{L_1,L_3\in\mathcal{D}\\L_j^{(l_j)}=K}}|b_{L_1,L_3,K}||h_{L_1}(z)||h_{L_3}(y_K)| \\
    &=|K|^{3/2}\sum_{\substack{L_1,L_3\in\mathcal{D}\\L_j^{(l_j)}=K}}|b_{L_1,L_3,K}|\frac{1_{L_1}(z)}{|L_1|^{1/2}}\frac{1_{L_3}(y_K)}{|L_3|^{1/2}} \\
    &\le|K|^{3/2}\sum_{\substack{L_1,L_3\in\mathcal{D}\\L_j^{(l_j)}=K}}\sum_{\substack{Q_1,Q_2,Q_3\in\mathcal{D}\\Q_j^{(k_j-l_j)}=L_j}}|a_{Q_1,Q_2,Q_3,K}|\prod_{j=1}^3\frac{|Q_j|^{1/2}}{|L_j|^{1/2}}\frac{1_{L_1}(z)}{|L_1|^{1/2}}\frac{1_{L_3}(y_K)}{|L_3|^{1/2}}\\
    &\le\sum_{\substack{L_1,L_3\in\mathcal{D}\\L_j^{(l_j)}=K}}\sum_{\substack{Q_1,Q_2,Q_3\in\mathcal{D}\\Q_j^{(k_j-l_j)}=L_j}}\prod_{j=1}^3\frac{|Q_j|}{|L_j|}{1_{L_1}(z)}{1_{L_3}(y_K)} \\
    &\le 1.
\end{align*}
Using the decoupling estimate (Theorem~\ref{thrm: decoupling_ineq}) we thus conclude\begin{align*}
    &\left(\int_{\mathbb{R}}\int_{\mathcal{V}}\int_{\mathcal{V}}\max_{K\in\mathcal{D}_{i,\kappa}}|b_K(y_k,z_K)|\mathbb{E}\|\sum_{K\in\mathcal{D}_{i,\kappa}}\varepsilon_K1_K(x)\Delta_K^{l_1}f_1(z_K)\|_{S_{p_1}}^{p_1}d\nu(z) d\nu(y)dx\right)^{1/p_1}\\
    &\quad\le \left(\int_{\mathbb{R}}\int_{\mathcal{V}}\int_{\mathcal{V}}\mathbb{E}\|\sum_{K\in\mathcal{D}_{i,\kappa}}\varepsilon_K1_K(x)\Delta_K^{l_1}f_1(z_K)\|_{S_{p_1}}^{p_1}d\nu(z) d\nu(y)dx\right)^{1/p_1} \\
    &\quad=\left(\mathbb{E}\int_{\mathbb{R}}\int_{\mathcal{V}}\|\sum_{K\in\mathcal{D}_{i,\kappa}}\varepsilon_K1_K(x)\Delta_K^{l_1}f_1(z_K)\|_{S_{p_1}}^{p_1}d\nu(z)dx\right)^{1/p_1} \\
    &\quad\le\beta_{{p_1},S_{p_1}}\int_{\mathbb{R}} \|f_1\|_{S_{p_1}}^{p_1}dx .
\end{align*}
The second case hence yields the estimate \begin{equation*}
    \|S^k_i(f_1,f_2)\|_{L^p(\mathbb{R},S_p)} \lesssim  \beta_{p,S_p}^2 \beta_{p_1,S_{p_1}}
\end{equation*}
in the case where the index of the non-cancellative Haar function is $j_0=2$.

Boundedness of the cases $j_0=1,3$  already follows from this result  by cyclic permutation of the functions in the  trilinear estimate~\eqref{Eqn=TriForm}  for $\Lambda_{S^k_i}$.  However, we can improve the resulting constant as follows.

\vspace{0.3cm}

\noindent {\bf Case 2 is self-improving using cyclic permutations.}  
Let $j_0=1$ denote the index of the non-cancellative Haar function. By applying the decoupling estimate, Stein's inequality, and H\"older's inequality in the same manner as in the $j_0=2$ case, we obtain \begin{multline*}
    \|S_i^k(f_1,f_2)\|_{L^p(\mathbb{R},S_p)} 
    \le \beta^2_{p,S_p} \|f_1\|_{L^{p_1}(\mathbb{R},S_{p_1})} \\ 
   \times  (\mathbb{E}\int_{\mathbb{R}}\int_{\mathcal{V}}\|\sum_{K\in\mathcal{D}_{i,\kappa}}\varepsilon_K1_K(x)|K|^{1/2}
    \sum_{\substack{L_2,L_3\in\mathcal{D}\\L_j^{(l_j)}=K}}b_{L_2,L_3,K}\langle f_2,h_{L_2}\rangle h_{L_3}(y_K)\|_{S_{p_2}}^{p_2} d\nu(y)dx)^{1/p_2}.
\end{multline*}
Proceeding to estimate the remaining integral as in the $j_0=2$ case yields \begin{equation*}
    \|S_i^k(f_1,f_2)\|_{L^p(\mathbb{R},S_p)} \le \beta^2_{p,S_p} \beta_{p_2,S_{p_2}} \|f_1\|_{L^{p_1}(\mathbb{R},S_{p_1})}\|f_2\|_{L^{p_2}(\mathbb{R},S_{p_2})}.
\end{equation*}
In order to optimise the behaviour of this constant as $p \searrow 1$, we now apply the following permutation argument.

By writing out the trilinear form associated with $S_i^k$, see~\eqref{eqn: new_shift}, where we add the index of the non-cancellative Haar function as a superscript, we see that \begin{align*}
    \Lambda^{(j_0=1)}_{S_i^k}(f_1,f_2,f_3)&= \sum_{K\in\mathcal{D}_{i,\kappa}}\sum_{\substack{L_2,L_3\in\mathcal{D}\\L_j^{(l_j)}=K}}b_{L_2,L_3,K}\tau\left(\langle f_1,h^0_{K}\rangle\langle f_2,h_{L_2}\rangle\langle f_3,h_{L_3}\rangle\right) \\ &=
    \sum_{K\in\mathcal{D}_{i,\kappa}}\sum_{\substack{L_2,L_3\in\mathcal{D}\\L_j^{(l_j)}=K}}b_{L_2,L_3,K}\tau\left(\langle f_3,h_{L_3}\rangle\langle f_1,h^0_{K}\rangle\langle f_2,h_{L_2}\rangle\right) \\
    &= \Lambda^{(j_0=2)}_{S^{'k}_i}(f_3,f_1,f_2),
\end{align*}
where $S^{'k}_i$ is a dyadic shift with $j_0=2$ and the same scalar sequence $(b_{L_j,K})$ as $S_i^k$ up to renumbering. Noting that $\beta_{p^*,S_{p^*}}=\beta_{p,S_p}$ (see e.g.\ \cite{aibs_vol_1}), we can thus apply the estimate of the $j_0=2$ case to conclude \begin{equation*}
    |\Lambda^{(j_0=1)}_{S_i^k}(f_1,f_2,f_3)| = |\Lambda^{(j_0=2)}_{S^{'k}_i}(f_3,f_1,f_2)| \lesssim \beta_{p_2,S_{p_2}}^2\beta_{p,S_p} \prod_{i=1}^3 \|f_i\|_{L^{p_i}(\mathbb{R},S_{p_i})},
\end{equation*}
and by similar cyclic permutation arguments \begin{align*}
    |\Lambda^{(j_0=2)}_{S_i^k}(f_1,f_2,f_3)| & = |\Lambda^{(j_0=1)}_{S_i^{'k}}(f_2,f_3,f_1)| \lesssim\beta_{p_1,S_{p_1}}^2\beta_{p,S_p} \prod_{i=1}^3 \|f_i\|_{L^{p_i}(\mathbb{R},S_{p_i})}, \\
    |\Lambda^{(j_0=3)}_{S_i^k}(f_1,f_2,f_3)| & = |\Lambda^{(j_0=1)}_{S_i^{'k}}(f_3,f_1,f_2)| \lesssim \beta_{p_2,S_{p_2}}^2\beta_{p_1,S_{p_1}} \prod_{i=1}^3 \|f_i\|_{L^{p_i}(\mathbb{R},S_{p_i})},  \\
    |\Lambda^{(j_0=3)}_{S_i^k}(f_1,f_2,f_3)| & = |\Lambda^{(j_0=2)}_{S_i^{'k}}(f_2,f_3,f_1)| \lesssim\beta_{p_1,S_{p_1}}^2\beta_{p_2,S_{p_2}} \prod_{i=1}^3 \|f_i\|_{L^{p_i}(\mathbb{R},S_{p_i})},
\end{align*}
where $S_i^{'k}$ may denote different shifts in each line.

Combining all cases, where we consider all possible locations of the non-cancellative Haar function in Case~2, we conclude (using $\kappa_{2,q}\le 60$, see Remark~\ref{rem: kahane_khintchine_constant})
\begin{equation*}
    C({p,p_1,p_2})\lesssim \beta_{p}\beta_{p_1}\beta_{p_2} +  
    +\min(\beta_{p_1}^2\beta_{p},\beta_{p}^2\beta_{p_1})
 + \min(\beta_{p_2}^2\beta_{p},\beta_{p}^2\beta_{p_2}) +\min(\beta_{p_2}^2\beta_{p_1},\beta_{p_1}^2\beta_{p_2}),
\end{equation*}
where we set $\beta_p:=\beta_{p,S_p}$. Note that by~\cite{Randri}, we have $\beta_{p,S_p}=pp^*$, hence this notation agrees with the notation of the constant $C$ in~\eqref{Eqn=CFunction} used in the main body of this paper.
\end{proof}


\begin{thebibliography}{99}

\bibitem[BoFe84]{BoFe84} 
M. Bo\.{z}ejko, G. Fendler, 
\emph{Herz-Schur multipliers and completely bounded multipliers of the Fourier algebra of a locally compact group} (English, with Italian summary), 
Boll. Un. Mat. Ital. A (6) 3 (1984), no. 2, 297–302. 

\bibitem[CCP22]{CCP-JMPA}
   L. Cadilhac, J. Conde-Alonso, J. Parcet,
   \emph{Spectral multipliers in group algebras and noncommutative Calder\'on-Zygmund theory},
    J. Math. Pures Appl. (9) {\bf 163} (2022), 450--472. \doi{10.1016/j.matpur.2022.05.011}.





\bibitem[CMPS14]{CMPS}
   M. Caspers, S. Montgomery-Smith, D. Potapov, F. Sukochev,
   \emph{The best constants for operator Lipschitz functions on Schatten classes},
   J. Funct. Anal. {\bf 267} (2014), no. 10, 3557--3579. \doi{10.1016/j.jfa.2014.08.018}

\bibitem[CaSa15]{CaSa15} 
M. Caspers, M. de la Salle, 
\emph{Schur and Fourier multipliers of an amenable group acting on non-commutative $L^p$-spaces}, 
Trans. Amer. Math. Soc. {\bf 367} (2015), no. 10, 6997--7013. \doi{10.1090/S0002-9947-2015-06281-3}


\bibitem[CPSZ15]{CPSZ-JOT}
   M. Caspers, D. Potapov, F. Sukochev, D. Zanin,
   \emph{Weak type estimates for the absolute value mapping},
   J. Operator Theory {\bf 73} (2015), no. 2, 361--384. \doi{10.7900/jot.2013dec20.2021}.


\bibitem[CPSZ19]{CPSZ}
   M. Caspers, D. Potapov, F. Sukochev, D. Zanin,
   \emph{Weak type commutator and Lipschitz estimates: resolution of the Nazarov-Peller conjecture},
    Am. J. Math. {\bf 141}, No. 3, 593--610 (2019). \doi{10.1353/ajm.2019.0019}.



\bibitem[CJSZ20]{CJSZ}
   M. Caspers, M. Junge, F. Sukochev, D. Zanin,
  \emph{BMO-estimates for non-commutative vector valued Lipschitz functions},
   J. Funct. Anal. {\bf 278} (2020), no. 3, 108317. \doi{10.1016/j.jfa.2019.108317}.



\bibitem[CSZ21]{CSZ-Israel}
   M. Caspers, F. Sukochev, D. Zanin,
   \emph{Weak $(1,1)$ estimates for multiple operator integrals and generalized absolute value functions},
   Israel J. Math. {\bf 244} (2021), no. 1, 245--271. \doi{10.1007/s11856-021-2179-0}.





\bibitem[CKV22]{CKV}
  M. Caspers, A. Krishnaswamy-Usha, G. Vos,
   \emph{Multilinear transference of Fourier and Schur multipliers acting on non-commutative $L^p$-spaces},
   Canad. J. Math. {\bf 75} (2023), no. 6, 1986--2006. \doi{10.4153/S0008414X2200058X}.

\bibitem[CJKM23]{CJKM}
   M. Caspers, B. Janssens, A. Krishnaswamy-Usha,  L. Miaskiwskyi,
  \emph{Local and multilinear noncommutative de Leeuw theorems},
   Math. Ann. {\bf 388} (2024), no. 4, 4251--4305. \doi{10.1007/s00208-023-02611-z}.



 

\bibitem[CLS21]{CLS-AIF}
     C. Coine, C. Le Merdy, F. Sukochev,
   \emph{When do triple operator integrals take value in the trace class?},
    Ann. Inst. Fourier {\bf 71}, No. 4, 1393--1448 (2021). \doi{10.5802/aif.3422}


\bibitem[CLPST16]{CLPST}
  C. Coine, C.  Le Merdy, D. Potapov, F. Sukochev, A. Tomskova,
   \emph{Resolution of Peller’s problem concerning Koplienko-Neidhardt trace formulae},
   Proc. Lond. Math. Soc. (3) {\bf 113}, No. 2, 113--139 (2016). \doi{10.1112/plms/pdw024}






\bibitem[CGPT22a]{ParcetAnnals}
   J.M. Conde-Alonso, A.M. Gonz\'alez-P\'erez, J. Parcet, E. Tablate,
   \emph{Schur multipliers in Schatten-von Neumann classes},
   Ann. of Math. (2) {\bf 198} (2023), no. 3, 1229--1260. \doi{10.4007/annals.2023.198.3.5}

\bibitem[Dav88]{Davies}
   E. Davies,
   \emph{Lipschitz continuity of functions of operators in the Schatten classes},
   J. Lond. Math. Soc. {\bf 37}  (1988) 148--157. \doi{10.1112/jlms/s2-37.121.148}

\bibitem[DeLo93]{Devore}
   R. DeVore, G.  Lorentz,
   \emph{Constructive approximation},
   Grundlehren der Mathematischen Wissenschaften. 303. Berlin: Springer- Verlag. x, 449 p. (1993).


\bibitem[DLMV20a]{DiPlinioMathAnn}
   F. Di Plinio, K. Li, H. Martikainen, E. Vuorinen,
   \emph{Multilinear singular integrals on non-commutative $L^p$-spaces},
    Math. Ann. {\bf 378}, No. 3-4, 1371--1414 (2020). \doi{10.1007/s00208-020-02068-4}
    
\bibitem[DLMV20b]{diplinio_multilinear_czo}
F. Di Plinio, K. Li, H. Martikainen, E. Vuorinen,
\emph{Multilinear operator-valued Calderón-Zygmund theory}, Journal of Functional Analysis, 279(8), 108666 (2020). \doi{10.1016/j.jfa.2020.108666}


\bibitem[FiWo01]{FigielW}
   T. Figiel,  P. Wojtaszczyk,  
  \emph{Special bases in function spaces},
   Handbook of the geometry of Banach spaces, Vol. I, 561--597,
   North-Holland Publishing Co., Amsterdam, 2001.



\bibitem[GMN99]{Gesztesy}
  F. Gesztesy, K. Makarov, S. Naboko.
  The spectral shift operator. Mathematical results in quantum mechanics (Prague, 1998),
   59–90, Oper. Theory Adv. Appl., 108, Birkhäuser, Basel, 1999. \doi{10.1007/978-3-0348-8745-8_5}



\bibitem[GoKr69]{GohbergKrein}
    I.C.  Gohberg, M.G. Krein,
   \emph{Introduction to the theory of linear nonselfadjoint operators},
      Translations of Mathematical Monographs. 18. Providence, RI: American Mathematical Society (AMS). xv, 378 p. (1969).



\bibitem[GrTo02]{GrafakosTorres}
   L. Grafakos, R. Torres,
   \emph{Multilinear Calder\'on-Zygmund theory},
    Adv. Math. {\bf 165}, No. 1, 124--164 (2002). \doi{10.1006/aima.2001.2028}

\bibitem[Gra04]{GrafakosOldNew}
  L. Grafakos, 
  \emph{Classical and modern Fourier analysis},
    Pearson Education, Inc., Upper Saddle River, NJ, 2004, xii+931 pp.

\bibitem[HaHy16]{hanninen_operator-valued_2016}
T. Hänninen, T. Hytönen, 
\emph{Operator-valued dyadic shifts and the {T}(1) theorem},
{Monatshefte für Mathematik}, Vol 180, No. 2, 213--253 (2016). \doi{10.1007/s00605-016-0891-3}.

\bibitem[HNVW16]{aibs_vol_1}
T. Hytönen, J. van Neerven, M. Veraar, L. Weis,
\emph{Analysis in Banach Spaces}, Volume I: Martingales and Littlewood-Paley Theory,
ISBN {9783319485201}, {Springer}, {2016}. 

\bibitem[JuSh05]{JungeSherman}
   M. Junge, D. Sherman,
   \emph{Noncommutative $L^p$-modules},
    J. Operator Theory {\bf 53} (2005), no. 1, 3--34.

\bibitem[Kat73]{Kato} 
  T. Kato, 
  \emph{Continuity of the map $S \mapsto \vert S \vert$  for linear operators},
  Proc. Japan Acad. {\bf 49} (1973), 157--160. \doi{10.3792/pja/1195519395}

\bibitem[KeSt99]{KenigStein}
C. E. Kenig and E. M. Stein, 
\emph{Multilinear estimates and fractional integration}, 
Mathematical Research Letters {\bf 6} (1999), no. 1, 1–-15. \doi{10.4310/mrl.1999.v6.n1.a1}.

\bibitem[Kop84]{kopl_trace}
L. S. Koplienko, 
\emph{Trace formula for perturbations of nonnuclear type}, 
Sibirsk. Mat. Zh. 25 (1984), 62-71 (Russian). English transl. in Siberian Math. J. 25 (1984), 735–74. \doi{10.1007/BF00968686}

 


\bibitem[Kre53]{Krein1}
   M.G. Krein,
   \emph{On the trace formula in perturbation theory},
   {\it Mat. Sbornik N.S.} 33(75): 597–626, 1953.


\bibitem[Kre62]{Krein2}
   M.G. Krein.
   \emph{On perturbation determinants and a trace formula for unitary and self-adjoint operators}.
   \textit{Dokl. Akad. Nauk SSSR}. 144: 268--271, 1962.


\bibitem[Kre83]{krein_perturbation}
 M. G. Krein, 
 \emph{On certain new studies in the perturbation theory for selfadjoint operators}, in Topics in Differential and Integral Equations and Operator Theory, I. Gohberg, Ed. Basel: Birkhäuser
Basel, 107–-172 (1983). \doi{10.1007/978-3-0348-5416-0}.   


\bibitem[LaSa11]{LafforgueDelaSalle}
   V. Lafforgue, M. de la Salle, 
   \emph{Noncommutative $L^p$-spaces without the completely bounded approximation property},  
    Duke Math. J. {\bf 160} (2011), no. 1, 71--116.



\bibitem[LMOV18]{bilinrepthrm}
  K. Li, H. Martikainen, Y. Ou, E. Vuorinen,
   \emph{Bilinear representation theorem},
    Trans. Am. Math. Soc. {\bf 371} (2018), no. 6, 4193--4214. \doi{10.1090/tran/7505}.


\bibitem[Lif52]{Lifschitz}
    I.M. Lifschitz.
   \emph{On a problem of perturbation theory}.  \textit{Uspekhi Mat. Nauk.}. 7:171--180, 1952.



\bibitem[McDSu22]{McDonaldSukochev}
   E. McDonald, F. Sukochev, 
   \emph{Lipschitz estimates in quasi-Banach Schatten ideals},
   Math. Ann. {\bf 383} (2022), no. 1--2, 571--619. \doi{10.1007/s00208-021-02247-x}.



\bibitem[NeRi11]{NeRi11} 
 S. Neuwirth, \'E. Ricard, 
 \emph{Transfer of Fourier multipliers into Schur multipliers and sumsets in a discrete group}, 
 Canad. J. Math. {\bf 63}  (2011), no. 5, 1161--1187. \doi{10.4153/CJM-2011-053-9}.

\bibitem[Pel85]{Peller}
  V.V. Peller,
  \emph{Hankel operators in the theory of perturbations of unitary and selfadjoint operators},
   Funktsional. Anal. i Prilozhen. {\bf 19} (1985), no. 2, 37--51, 96. \doi{10.1007/BF01078390}.

\bibitem[PSS13]{PSS-Inventiones}
   D. Potapov, A. Skripka, F. Sukochev, 
   \emph{Spectral shift function of higher order},  
   Invent. Math. {\bf 193}, No. 3, 501--538 (2013). \doi{10.1007/s00222-012-0431-2}.

\bibitem[PoSu11]{PoSu11}
D. Potapov and F. Sukochev, 
\emph{Operator-Lipschitz functions in Schatten–von Neumann classes},
Acta Mathematica {\bf 207}, no. 2, pp. 375--389, Dec. 2011. \doi{10.1007/s11511-012-0072-8}.

\bibitem[Ran02]{Randri}
   N. Randrianantoanina,  
   \emph{Non-commutative martingale transforms},  
   J. Funct. Anal. {\bf 194} (2002), no.1, 181--212. \doi{/10.1006/jfan.2002.3952}.

\bibitem[Rei24]{JesseThesis}
   J. Reimann,
   \emph{Schur Multipliers of Divided Differences and Multilinear Harmonic Analysis},
   Master thesis at TU Delft,  available at \href{https://repository.tudelft.nl/islandora/object/uuid%3A2ec66360-478e-4768-be95-044c2c015d4f}{https://repository.tudelft.nl/}.


\bibitem[SkTo19]{SkripkaTomskova}
 A. Skripka, A. Tomskova, 
  \emph{Multilinear operator integrals}, 
  Lecture Notes in Math., 2250, Springer, Cham, 2019, xi+190 pp. \doi{10.1007/978-3-030-32406-3}

\end{thebibliography}
\end{document}